\documentclass[12pt]{amsart}
\usepackage{amsmath, amssymb,latexsym }
\usepackage{verbatim}

\newcommand\Ha{\operatorname{Ha}^{(1)}}

\newcommand{\tcal}{\mathcal{T}}
\newcommand{\fcal}{\mathcal{F}}

\newcommand{\kahler}{K\"ahler}
\newcommand{\D}{\mathbf D}

\def\sq2{\sqrt{2}}

\def\vol{\mathrm{vol}}

\def\t2{{\mathbb T}^2}
\def\s2{{\mathbb S}^2}

\def\R{\mathbb{R}}
\def\Z{\mathbb{Z}}
\def\C{\mathbb{C}}

\def\Lap{\triangle}

\newcommand\pa{\partial}

\newcommand{\RR}{{\mathbb R}}
\newcommand\Op{\operatorname{Op}}

\newcommand\Neu{{\operatorname{Neu}}}

\def\sq2{\sqrt{2}}

\def\vol{\mathrm{vol}}

\def\t2{{\mathbb T}^2}
\def\s2{{\mathbb S}^2}

\def\R{\mathbb{R}}
\def\Z{\mathbb{Z}}
\def\C{\mathbb{C}}

\def\Lap{\triangle}

\def\hto0{\xrightarrow{\hbar\to 0}}

\def\rto0{\xrightarrow{r\to 0}}

\newcommand{\HH}{\mathcal H}

\newcommand{\cE}{\mathcal E}

\newcommand{\cH}{\mathcal H}

\newcommand{\cP}{\mathcal P}

\newcommand{\IR}{\mathbb{R}} 

\newcommand{\X}{{\mathbf X_{\Gamma}}}

\renewcommand{\le}{\leqslant}
\renewcommand{\ge}{\geqslant}

\newcommand{\E}{\mathbb E}

\newcommand{\Hh}{{\mathbb H}}

\newcommand{\pdo}{\psi DO}

\newcommand{\dbar}{\bar\partial}
\newcommand{\ddbar}{\partial\dbar}
\newcommand{\cU}{{\mathcal U}}

\newcommand{\B}{{\mathbb B}}
\renewcommand{\Re}{{\operatorname{Re\,}}}

\newcommand{\half}{{\frac{1}{2}}}
\renewcommand{\phi}{\varphi}

\newcommand{\ccal}{\mathcal{C}}
\newcommand{\dcal}{\mathcal{D}}
\newcommand{\ecal}{\mathcal{E}}

\newcommand{\hcal}{\mathcal{H}}

\newcommand{\mcal}{\mathcal{M}}
\newcommand{\ncal}{\mathcal{N}}
\newcommand{\pcal}{\mathcal{P}}

\newcommand{\ucal}{\mathcal{U}}

\newcommand{\oit}{\operatorname{ it}}

\renewcommand{\H}{{\mathbf H}}
\newcommand{\qcal}{\mathcal{Q}}
\newcommand{\sa}{\sigma_A}

\newcommand{\varphijk}{\varphi_{j_k}}

\newcommand{\de}{\delta}

\newtheorem{theo}{{\sc Theorem}}[section]
\newtheorem{cor}[theo]{{\sc Corollary}}
\newtheorem{conj}[theo]{{\sc Conjecture}}
\newtheorem{lem}[theo]{{\sc Lemma}}

\newtheorem{prop}[theo]{{\sc Proposition}}
\newtheorem{prob}[theo]{{\sc Problem}}

\newenvironment{rem}{\medskip\noindent{\it Remark:\/} }{\medskip}
\newenvironment{defin}{\medskip\noindent{\it Definition:\/} }{\medskip}

\title[Recent developments in mathematical Quantum Chaos ]
{Recent developments in mathematical Quantum Chaos\\ Preliminary
Version/Comments Appreciated}

\author{Steve Zelditch }
\address{Department of Mathematics, Johns Hopkins University, Baltimore,
MD 21218, USA} \email{zel@math.jhu.edu}

\thanks{Research partially supported by  NSF grant
  DMS 0904252.}

\date{\today}

\begin{document}

\begin{abstract} This is a survey of recent results on quantum
ergodicity, specifically on the large energy limits of matrix
elements relative to eigenfunctions of the Laplacian. It is mainly
devoted to QUE (quantum unique ergodicity) results, i.e. results
on the possible existence of a sparse subsequence of
eigenfunctions with anomalous concentration. We cover  the lower
bounds on entropies of quantum limit measures due to Anantharaman,
Nonnenmacher, and Rivi\`ere on compact Riemannian manifolds with
Anosov flow. These lower bounds give new constraints on the
possible quantum limits. We also cover  the non-QUE result of
Hassell in the case of the Bunimovich stadium. We include some
discussion of Hecke eigenfunctions and recent results of
Soundararajan completing Lindenstrauss' QUE result,  in the
context of matrix elements for Fourier integral operators.
Finally, in answer to the potential question `why study matrix
elements' it presents an application of the author to the geometry
of nodal sets.

\end{abstract}

\maketitle

\tableofcontents

 Quantum chaos on Riemannian manifolds $(M, g)$ is
concerned with the asymptotics of eigenvalues and orthonormal
bases of eigenfunctions \begin{equation} \label{DELTAPHI} \Delta
\phi_j = \lambda_j^2 \phi_j,\;\;\;\;\;\;\;\;\langle \phi_j, \phi_k
\rangle = \delta_{jk} \end{equation} of the Laplacian $\Delta$
when the geodesic flow $g^t: S^*_g M \to S^*_g M$ is ergodic or
Anosov or in some other sense chaotic. Model examples include
compact or finite volume hyperbolic surfaces $\X = \Gamma
\backslash \H$, where $\H$ is the hyperbolic plane and $\Gamma
\subset PSL(2, \R)$ is a discrete subgroup.  Of special interest
are the arithmetic quotients when $\Gamma$ is a discrete
arithmetic subgroup. Other model example include Euclidean domains
with ergodic billiards such as the Bunimovich stadium.
 The general question is,  how does the dynamics of the
geodesic (or billiard) flow make itself felt in the eigenvalues
and eigenfunctions (or more general solutions of the wave
equation), which by the correspondence principle of quantum
mechanics must reproduce the classical limit as the eigenvalue
tends to infinity?

To avoid confusion, we emphasize that we denote eigenvalues of the
Laplacian by $\lambda_j^2$. They are usually viewed as energies
$E_j = \lambda_j^2$.  Their square roots  $\lambda_j$ are called
the frequencies.

The most basic quantities testing the asymptotics of
eigenfunctions are the matrix elements $\langle A \phi_j, \phi_j
\rangle$ of pseudo-differential operators relative to the
orthonormal basis of  eigenfunctions. As explained below, these
matrix elements measure the expected value of the observable $A$
in the energy state $\phi_j$. Much of the work in quantum
ergodicity since the pioneering work of A. I. Schnirelman
\cite{Sh.1}  is devoted to the study of the limits of $\langle A
\phi_j, \phi_j \rangle$ as $\lambda_j \to \infty$. Although
difficult to determine, these limits are still the most accessible
aspects of eigenfunctions. A sequence of eigenfunctions is called
ergodic or diffuse if the limit tends to $\int_{S^*_g M} \sigma_A
d\mu_L$ where $\sigma_A$ is the prinicipal symbol of $A$ and
$d\mu_L$ is the Liouville measure.

  The first priority of this survey is to cover the  important new
results on the so-called scarring or  QUE (quantum unique
ergodicity) problem. Roughly speaking, the problem is whether
every orthonormal basis becomes equidistributed in phase space
with respect to Liouville measure  as the eigenvalue tends to
infinity, i.e. $$\langle A \phi_j, \phi_j \rangle \to \int_{S^*_g
M} \sigma_A d\mu_L$$ or whether there exists a sparse exceptional
sequence which has a singular concentration.  One series of
positive results due to N. Anantharaman \cite{A},  N.
Anantharaman-S. Nonnenmacher \cite{AN} (see also \cite{ANK}), and
G. Rivi\`ere \cite{Riv} flows from the study of entropies of
quantum limit measures and are based on a difficult and complex
microlocal PDE analysis of the long time behavior of the wave
group on manifolds with Anosov geodesic flow. These results given
lower bounds on entropies  of quantum limit measures which imply
that the limit measures  must be to some extent
 diffuse. Another series
of results due  R. Soundararajan \cite{Sound1} use $L$-function
methods to  complete the near QUE result of E. Lindenstrauss
\cite{LIND} for  Hecke eigenfunctions in the arithmetic case.  In
the negative direction, A. Hassell \cite{Has} uses microlocal
methods to proved the long-standing conjecture that generic stadia
are not QUE but rather have exceptional sequences of bouncing ball
type modes which concentrate or scar along  a Lagrangian
submanifold.

QUE is not the only interesting problem in quantum chaos. Another
series of results in the arithmetic case are the remarkably sharp
Luo-Sarnak asymptotics \cite{LS,LS2}  of the variances $$\sum_{j:
\lambda_j \leq \lambda} |\langle A \phi_j, \phi_j \rangle -
\int_{S^*_g M} \sigma_A d\mu_L|^2 $$ for holomorphic Hecke
eigenforms and their recent generalization to smooth Maass forms
by Zhao \cite{Zh}. Physicists (see e.g. \cite{FP}) have speculated
that the variance is related to the classical auto-correlation
function of the geodesic flow. The Luo-Sarnak-Zhao results
partially confirm this conjecture in the special case of
arithmetic surfaces and Hecke eigenfunctions, but correct it with
an extra arithmetic factor.

A final direction we survey is the applications of quantum
ergodicity to problems in nodal geometry. Given the large amount
of work on matrix elements $\langle A \phi_j, \phi_j \rangle$, it
is natural for geometric or PDE  analysts   to ask how one can use
such matrix elements to study classical problems on eigenfunctions
such as growth and distributions of zeros and critical points. One
of our themes is that for real analytic $(M, g)$, logarithms $\log
|\phi_j^{\C}(\zeta)|^2$ of analytic continuations of ergodic
eigenfunctions are asymptotically `maximal pluri-subharmonic
functions' in Grauert tubes and as a result their zero sets have a
special limit distribution \cite{Z5}. Thus, ergodicity causes
maximal oscillation in the real and complex domain and gives rise
to special distributions of zeros and (conjecturally) to critical
points.

In preparing this survey, we were only too aware of the large
number of surveys and expository articles that already exist on
the QUE and entropy problems, particularly
\cite{AN2,ANK,CV3,LIND2,Sar3} (and for earlier work
\cite{Z1,Z2,M,Sar}).  To a large degree we closely follow the
original sources, emphasizing  intuition over the finer technical
details. Although we go over  much of the same material in what
are now standard ways, we also go over some topics that do not
seem as well known, and possibly could be improved, and we  fill
in some gaps in the literature. In particular, most expositions
emphasize the quantum entropic uncertainty principle of \cite{AN}
due to its structural nature. So instead we emphasize the ideas of
\cite{A}, which are less structural but in some ways more
geometric.

One of the novelties is the study in \S \ref{FIOME} of matrix
elements $\langle F \phi_j, \phi_j \rangle$ of Fourier integral
operators $F$ relative to the eigenfunctions. Hecke operators are
examples of such $F$ but restrictions to hypersurfaces provide a
different kind of $F$ which have recently come up in quantum
ergodic restriction theory \cite{TZ3}.  We give a rather simple
result on quantum limits for Fourier integral operators which
answers the question, ``what invariance property do quantum limit
measures of Hecke eigenfunctions possess?". Hecke analysts have
worked with a partial quasi-invariance principle due to
Rudnick-Sarnak, and although the exact invariance principle may
not simplify the known proofs of QUE  it is of interest to know
that one exists. We plan to put complete proofs in a future
article.

Another aspect of eigenfunctions is their possible concentration
of eigenfunctions around closed hyperbolic geodesic. Mass
concentration within thin or shrinking tubes around such geodesics
has been studied by S.Nonnenmacher-A.Voros \cite{NV}, Y. Colin de
Verdi\`ere-B. Parisse \cite{CVP}, J. Toth and the author
\cite{TZ2} and outside of such tubes by N. Burq-M. Zworski
\cite{BZ} and H. Christianson \cite{Chr}. Very striking and
surprising studies of analogous behavior in the `quantum cat map'
setting are in \cite{FNB,FN}.


A few words on the limitations of this survey. Due to the author's
lack of knowledge, the $L$-function or arithmetic methods in QUE
will not be discussed in more than a cursory way. Rather we
concentrate on the PDE or microlocal aspects of quantum
ergodicity. By necessity, microlocal analytical proofs must make
direct connections between spectrum and geodesic flow, and cannot
by-pass this obstacle by taking a special arithmetic route.
Moreover, we restrict the setting to Laplacians on Riemannian
manifolds for simplicity of exposition, but it is just a special
case of the semi-classical asymptotics (as the Planck constant
tends to zero) of spectra of Schr\"odinger operators $h^2 \Delta +
V$. Schr\"odinger eigenfunctions are equally relevant to
physicists and mathematics. Another important setting is that of
quantum `cat maps' on tori and other \kahler \; manifolds. In this
context, the quantum dynamics is defined by quantizations of
symplectic maps of K\"ahler  manifolds acting on spaces of
holomorphic sections of powers of a positive line bundle. The
mathematics is quite similar to the Riemannian setting, and
results in the cat map setting often suggest analogues in the
Riemannian setting. In particular, the scarring results of
Faure-Nonnemacher- de Bi\`evre \cite{FNB} and the variance results
of Kurlberg-Rudnick \cite{KR} and Schubert \cite{Schu3} are very
relevant to the scarring and variance results we present in the
Riemannian setting.

Although we do not discuss it here, a good portion of the
literature of quantum chaos is devoted to numerics and physicists'
heuristics. Some classics are \cite{Ber,He,He2}.  The article of
A. Barnett \cite{Bar} (and in the cat map case \cite{FNB,FN} ) is
a relatively recent discussion of the numerical results in quantum
ergodicity. There is a wealth of phenomenology standing behind the
rigorous results and heuristic proofs in this field. Much of it
still lies far beyond the scope of the current mathematical
techniques.

Finally, we thank N. Anantharaman, A. Hassell, and S. Nonnenmacher
for explanations of important points in their work.  We have
incorporated many of their clarifications in the survey.  We also
thank K. Burns, H. Hezari, J. Franks,  E. Lindenstrauss, P.
Sarnak, A. Wilkinson  and J. Wunsch for further
 clarifications  and corrections. Of course, responsibility for
 any
 errors that remain is the author's.

\section{Wave equation and geodesic flow}

Classical mechanics is the study of Hamiltonians (real valued
functions) on phase space (a symplectic manifold) and their
Hamiltonian flows.  Quantum mechanics is the study of Hermitian
operators on Hilbert spaces and the unitary groups they generate.
The model quantum Hamiltonians we will discuss are Laplacians
$\Delta$ on  compact Riemannian manifolds $(M,g)$ (with or without
boundary). Throughout we denote the dimension of $M$ by $d = \dim
M$.

\subsection{\label{GEO} Geodesic flow and classical mechanics} The classical  phase space in this setting  is the
cotangent bundle $T^*M$ of $M$, equipped with its canonical
symplectic form $\sum_i dx_i \wedge d\xi_i$. The metric defines
the Hamiltonian $H(x,\xi) = |\xi|_g =  \sqrt{\sum_{ij = 1}^n
g^{ij}(x) \xi_i \xi_j}$ on $T^*M$, where  $g_{ij} =
g(\frac{\partial}{\partial x_i},\frac{\partial}{\partial x_j}) $,
$[g^{ij}]$ is the inverse matrix to $[g_{ij}]$. We denote the
volume density of $(M, g)$ by $dVol$ and the corresponding inner
product on $L^2(M)$ by $\langle f, g \rangle$.  The unit (co-)
ball bundle is denoted $B^*M = \{(x, \xi): |\xi| \leq 1\}.$

The classical evolution (dynamics) is given by the geodesic flow
of $(M, g)$, i.e. the  Hamiltonian flow $g^t$ of $H$ on $T^*M$. By
definition, $g^t(x, \xi) = (x_t, \xi_t)$, where $(x_t, \xi_t)$ is
the terminal tangent vector at time $t$ of the unit speed geodesic
starting at $x$ in the direction $\xi$. Here and below, we often
identify $T^*M$ with the tangent bundle $TM$ using the metric to
simplify the geometric description. The geodesic flow preserves
the energy surfaces $\{H = E\}$ which are the co-sphere bundles
$S^*_E M$.
 Due to the homogeneity of $H$, the
flow on any energy surface $\{H = E\}$ is equivalent to that on
the co-sphere bundle $S^*M = \{ H=1\}.$

We define the   {\it Liouville measure} $\mu_L$  on $S^*M$ to be
the surface measure $d\mu_L = \frac{dx d\xi}{d H}$  induced by the
Hamiltonian $H = |\xi|_g$ and by the symplectic volume measure $dx
d\xi$
 on $T^*M$. The geodesic flow (like any Hamiltonian flow)
 preserves $dx d\xi$. It also preserves Liouville measue on
 $S^*M$.

 \subsection{\label{EA} Ergodic and Anosov geodesic flows}

 The geodesic flow is called {\it ergodic} if the unitary operator
 \begin{equation} V^t a(\rho) = a (g^t (\rho)) \end{equation}
 on $L^2(S^* M, d\mu_L)$ has no invariant $L^2$ functions besides
 constant functions. Equivalently, any invariant set $E \subset S^*M$  has either
 zero measure or full measure.

A geodesic flow $g^t$ is called {\it Anosov} on on $S^*_g M$ if
 the tangent bundle $T S^*_g M$
splits into $g^t$ invariant sub-bundles $E^u(\rho)\oplus E^s(\rho)
\oplus \IR\,X_H(\rho)\, $ where $E^u$ is the unstable subspace and
$E^s$ the stable subspace. They are defined by
$$\begin{array}{ll} ||d g^t v|| \leq C e^{- \lambda t} ||v||, &
\forall v \in E^s, t \geq 0, \\ & \\
||d g^t v|| \leq C e^{ \lambda t} ||v||, & \forall v \in E^u, t
\leq 0. \end{array} $$ The sub-bundles are integrable and give
stable, resp. unstable foliations $W^s, W^u$. The leaves through
$x$ are denoted by $W^s(x), W^u(x)$. Thus, the geodesic flow
contracts everything exponentially fast along the stable leaves
and expands everything exponentially fast along the unstable
leaves. Anosov geodesic flows are ergodic. We refer to \cite{Kl}
for background.

The unstable Jacobian $J^u(\rho)$ at $\rho$ is defined by
\begin{equation} \label{UNSTABLEJAC} J^u(\rho)=\det\big(dg^{-1}_{|E^u(g^1\rho)}\big). \end{equation}

Define \begin{equation} \label{LAMBDAnosov}\begin{array}{lll}
\Lambda & = & - \sup_{\rho \in S^*M} \log ||J^u(\rho)|| \\ && \\
& =  & \inf_{\mu \in \mcal_I} \left| \int_{S^*M} \log J^u(\rho)
d\mu(\rho) \right|> 0. \end{array} \end{equation} If $(M, g)$ is
of dimension $d$ and constant curvature $-1$ then $\Lambda = d -
1$. Here, $\mcal_I$ is the set of invariant probability measures
for $g^t$.

 The maximal expansion
rate of the geodesic flow is defined by
\begin{equation}\label{LAMBDAMAX}
\lambda_{\max} = \lim_{t \to \infty} \frac{1}{t} \log \sup_{\rho
\in S^* M} ||d g^t_{\rho}||. \end{equation}

The derivative $dg^t$ may be expressed in terms of Jacobi fields
(see \cite{Kl}, Lemma 3.1.17). In the hyperbolic (Anosov case),
the stable/unstable sub-bundles are spanned by unstable Jacobi
fields $Y^u_1, \dots, Y^u_{d-1}$  whose norms grow exponentially as $t \to + \infty$
and decay exponentially as $t \to - \infty$,
resp. stable Jacobi fields $Y^s_1, \dots, Y^s_{d-1}$  whose norms
behave in the opposite fashion.  The unstable Jacobian is the norm $||Y^u_1 \wedge \cdots
\wedge Y_{d-1}^u||. $

\subsection{\label{REVIEW} Eigenfunctions and eigenvalues of
$\Delta$}

The quantization of the Hamiltonian $H$ is the square root
$\sqrt{\Delta}$ of the positive Laplacian,  $$\Delta = -
\frac{1}{\sqrt{g}}\sum_{i,j=1}^n \frac{\partial}{\partial
x_i}g^{ij} g \frac{\partial}{\partial x_j} $$  of $(M,g)$. Here,
$g = {\rm det} [g_{ij}].$ The eigenvalue problem on a compact
Riemannian manifold
$$\Delta \phi_j = \lambda_j^2 \phi_j,\;\;\;\;\;\;\;\;\langle \phi_j, \phi_k
\rangle = \delta_{jk}$$ is dual under the Fourier transform to the
wave equation. Here, $\{\phi_j\}$ is a choice of orthonormal basis
of eigenfunctions, which is not unique if the eigenvalues have
multiplicities $> 1.$ The individual eigenfunctions are difficult
to study directly, and so one generally forms the spectral
projections kernel,
\begin{equation} \label{ELAMBDA} E(\lambda, x, y) = \sum_{j: \lambda_j \leq
\lambda} \phi_j(x) \phi_j(y). \end{equation} Semi-classical
asymptotics is the study of the $\lambda \to \infty$ limit of the
spectral data $\{\phi_j, \lambda_j\}$  or of $E(\lambda, x, y)$.

\subsection{\label{HUP} Eigenvalues and Planck's constant}

Although it is only a notational issue in this setting, it is
conceptually useful to identify inverse frequencies with Planck's
constant, i.e. to define
\begin{equation} \hbar_j = \lambda_j^{-1}. \end{equation}
The eigenvalue problem then takes the form of a Schr\"odinger
equation $$\hbar^2_j \Delta \phi_j = \phi_j. $$  Thus the high
frequency limit is a special case of the semi-classical limit
$\hbar \to 0$.

\subsection{Eigenvalue Multiplicities}

The multiplicity of an eigenvalue $\lambda_j^2$ is the dimension
\begin{equation} m(\lambda_j) = \dim \ker (\Delta - \lambda_j^2)
\end{equation}
of the eigenspace. As will be discussed in \S \ref{WLSECT}, the
general Weyl law gives the bound $m(\lambda_j) \leq C \lambda_j^{d
- 1}$ in dimension $d$. In the case of negatively curved
manifolds, the bound has  been improved by a $(\log
\lambda_j)^{-1}$ factor. Thus it is possible for sequences of
eigenvalues to have the huge multiplicities $\lambda_j^{d - 1}
(\log \lambda_j)^{-1}$. When such multiplicities occur, it is
possible to build up superpositions of eigenfunctions which have
special concentration properties (see \S \ref{HYP}). Hassell's
scarring theorem is based on the non-existence of such
multiplicities for certain sequences of eigenvalues (\S
\ref{HAS}).

 In fact such multiplicities do occur in
the quantum cap map setting for sparse sequences of eigenvalues
and are responsible for the very strange behavior of eigenstates
in that case \cite{FNB,FN}. It is unknown if such sequences exist
in the Riemannian Anosov setting. This is one of the fundamental
obstacles to understanding whether or how scarring occurs in the
Riemannian case.

\subsection{Quantum evolution}

Quantum evolution is given by the  wave group
$$U^t =
e^{i t \sqrt{\Delta}}, $$ which in a rigorous sense is  the
quantization of the geodesic flow $\Phi^t$. It is generated by the
pseudo-differential operator $\sqrt{\Delta}$, as defined by the
spectral theorem (it has the same eigenfunctions as $\Delta$ and
the eigenvalues $\lambda_j$).  The (Schwartz) kernel of the wave
group on a compact Riemannian manifold can be represented in terms
of the spectral data by
\begin{equation} \label{UTEFNS} U^t (x,y) = \sum_j e^{it \lambda_j} \phi_j(x)
\phi_j(y),\end{equation} or equivalently as the Fourier transform
$\int_{\R} e^{i t \lambda} dE(\lambda, x, y)$ of the spectral
projections. Hence spectral asymptotics is often studied through
the large time behavior of the wave group.

Evolution can be studied on the level of `points' or on the level
of observables. Evolution of points in classical mechanics gives
the orbits or trajectories of the geodesic flow (i.e. the
parameterized geodesics). Evolution $\psi_t = U^t \psi$ of wave
functions  gives the `Schr\"odinger picture' of  quantum
mechanics. Eigenfunctions arise in quantum mechanics as stationary
states, i.e. states $\psi$ for which the probability measure  $|
\psi(t, x)|^2 dvol$ is constant where $\psi(t, x) = U^t \psi(x)$
is the evolving state. This follows from the fact that
\begin{equation} \label{UNIT} U^t \phi_k = e^{ i t \lambda_k} \phi_k \end{equation}
and that $|e^{i t \lambda_k}| = 1$. By unitarity $U^{t*} =
U^{-t}$.

As an alternative one could define  the quantum evolution as in
\cite{AN} to be the semi-classical Schr\"odinger propagator
\begin{equation} \label{UH} U_{\hbar}^t : = e^{- i t \hbar \Delta}. \end{equation}  This replaces  the  homogeneous
Fourier integral operator $U^t$ by a semi-classical Fourier
integral operator. The theories are quite parallel; we  mainly
keep to the homogeneous theory in this survey (See \cite{DSj} and
\cite{EZ}) for background on these notions).

\subsection{\label{O} Observables}
In the classical setting, observables are (real-valued) functions
on $T^*M$ which we usually take to be homogeneous of degree zero
(or as zeroth order symbols). In the quantum setting, observables
are $\pdo$'s (pseudodifferential operators)  of all orders;  we
often restrict to the subalgebra $\Psi^0$ of $\pdo$'s of order
zero.  We denote by $\Psi^m(M)$ the subspace of pseudodifferential
operators of order $m$. The algebra is defined by constructing a
quantization $Op$ from an algebra of symbols $a \in S^m(T^*M)$ of
order $m$ (polyhomogeneous functions on $T^*M \backslash 0)$ to
$\Psi^m$. A function on $T^*M \backslash 0$ is called
polyhomogeneous or a classical symbol of order $m$, $a \in
S^m_{phg}$,  if it admits an asymptotic expansion,
$$a(x, \xi) \sim \sum_{j = 0}^{\infty} a_{m - j}(x, \xi), \;\;\;
\mbox{where}\;\; a_k(x, \tau \xi) = \tau^{k} a_k(x, \xi), \;\;
(|\xi| \geq 1  )$$ for some $m$ (called its order). The asympotics
hold in the sense that
$$a - \sum_{j = 0}^{R} a_{m - j} \in S^{m - R - 1},$$
where $\sigma \in S^k$ if $\sup (1 + |\xi|)^{j - k}| D^{\alpha}_x
D^{\beta}_{\xi} \sigma (x, \xi)| < +\infty$ for all compact set
and for all $\alpha,\beta, j$. There is a semi-classical analogue
in the $\hbar$ setting where the complete symbol $a_{\hbar}(x,
\xi)$ depends on $\hbar$ in a similar poly-homogeneous way,
$$a_{\hbar} \sim  \hbar^{-m} \sum_{j = 0}^{R} \hbar^j a_{m - j}.$$
We refer to \cite{DSj,GSt,EZ,Dui} for background on microlocal
analysis (or semi-classical analysis), i.e. for the theory of
pseudo-differential and Fourier integral operators.

The main idea is that such observables have good classical limits.
The high frequency limit is mirrored in the behavior of the symbol
as $|\xi| \to \infty$.

An useful type of observable is a smooth cutoff $\chi_U(x, \xi)$
to some open set $U \subset S^*M$ (it is homogeneous so it cuts
off to the cone through $U$ in $T^*M$). That is, a function which
is one in a large ball in $U$ and zero outside $U$. Then
$Op(\chi_U)$ is called a microlocal cutoff to $U$. Its expectation
values in the states $\phi_k$ give the phase space mass of the
state in $U$ or equivalently the probability amplitude that the
particle is in $U$. (The modulus square of the amplitude is the
probability).

There are several standard useful notations for quantizations of
symbols $a$  to operators $A$ that we use in this survey. We use
the notation $Op(a)$ or $\hat{a}$ interchangeably.

\subsection{\label{HUPSECT} Heisenberg uncertainty prinicple}

The Heisenberg uncertainty principle is the heuristic principle
that one cannot measure things in   regions of phase space where
the product of the widths in configuration and momentum directions
is  $\leq \hbar$.

It is  useful to microlocalize to sets which shrink as $\hbar \to
0$ using $\hbar $ (or $\lambda$)-dependent cutoffs such as
$\chi_U(\lambda^{\delta} (x, \xi))$. The Heisenberg uncertainty
principle is manifested in pseudo-differential calculus in the
difficulty (or impossibility) of defining pseudo-differential
cut-off operators   $Op(\chi_U(\lambda^{\delta} (x, \xi))$  when
$\delta \geq \half$. That is, such small scale cutoffs do not obey
the usual rules of semi-classical analysis (behavior of symbols
under composition).  The uncertainty principle allows one to study
eigenfunctions by microlocal methods  in configuration space balls
$B(x_0, \lambda_j^{-1})$ of radius $\hbar$ (since there is no
constraint on the height) or in a phase space ball of radius
$\lambda_j^{-1/2}$.

\subsection{Matrix elements and Wigner distributions}

These lectures are mainly concerned with the matrix elements
$\langle Op(a) \phi_j, \phi_k \rangle$ of an observable relative
to the eigenfunctions. The diagonal matrix elements
\begin{equation} \label{RHOJ} \rho_j (Op(a)): = \langle Op(a) \phi_j, \phi_j \rangle \end{equation}  are interpreted in
quantum mechanics as the expected value of the observable $Op(a)$
in the energy state $\phi_j$. The off-diagonal matrix elements are
interpreted as transition amplitudes. Here, and below, an
amplitude is a complex number whose modulus square is a
probability.

If we fix the quantization $a \to Op(a)$, then the matrix elements
can be represented by {\it Wigner distributions}. In the diagonal
case, we define  $W_k \in \dcal'(S^* M)$ by
\begin{equation} \label{WIGNER} \int_{S^* M} a dW_k : =
\langle Op(a) \phi_k, \phi_k \rangle. \end{equation}

\subsection{Evolution of Observables}

Evolution of observables  is particularly relevant in quantum
chaos and is known in physics as the `Heisenberg picture'.

The evolution of observables  in the Heisenberg picture is defined
by
\begin{equation} \label{ALPHAT} \alpha_t (A) : = U^t A U^{-t}, \;\;\; A \in
\Psi^m(M). \end{equation} and since Dirac's Principles of Quantum
Mechanics, it was known to correspond to the classical evolution
\begin{equation}\label{VT}  V^t (a) : = a \circ g^t \end{equation}
of observables $a \in C^{\infty}(S^*M)$. Egorov's theorem is the
rigorous version of this correspondence: it says that $\alpha_t$
defines an order-preserving automorphism of $\Psi^*(M)$, i.e.
$\alpha_t(A) \in \Psi^m(M)$ if $A \in \Psi^m(M)$, and that
\begin{equation}\label{EGOROV}  \sigma_{U_t A U_t^*} (x, \xi)  =  \sigma_A
(g^t(x, \xi)) : = V^t (\sigma_A), \;\;\; (x, \xi) \in T^*M
\backslash 0.
\end{equation}
This  formula is almost universally taken to be the definition of
quantization of a flow or map in the physics literature.

\subsection{Hadamard parametrix}

The link between spectral theory and geometry, and the source of
Egorov's theorem for the wave group, is the construction of a
parametrix (or WKB formula) for the wave kernel. For small times
$t$, the simplest is the Hadamard parametrix,
\begin{equation} \label{HD} U^t (x,y) \sim \int_0^{\infty} e^{ i
\theta (r^2(x,y) - t^2)} \sum_{k = 0}^{\infty} U_k (x,y)
\theta^{\frac{d-3}{2} - k} d\theta\;\;\;\;\;\;\;\;\;(t < {\rm
inj}(M,g)) \end{equation} where $r(x,y)$ is the distance between
points, $U_0(x,y) = \Theta^{-\half}(x,y)$ is the volume
1/2-density, $inj(M, g)$ is the injectivity radius, and the higher
Hadamard coefficients are obtained by solving transport equations
along geodesics. The parametrix is asymptotic to the wave kernel
in the sense of smoothness, i.e. the difference of the two sides
of (\ref{HD}) is smooth. The relation (\ref{HD}) may be iterated
using $(U^t)^m = U^{t m}$ to obtain a parametrix for long times.
For manifolds without conjugate points, the Hadamard parametrix is
globally well defined on the universal cover and can then be
summed over the deck transformation group to obtain a parametrix
on the quotient.

Egorov's theorem (\ref{EGOROV}) may be proved by explicitly
writing out the integral representation for $U^{-t} A U^t$ using
(\ref{HD}) and by applying the stationary phase method to simplify
the integral.

\subsection{Ehrenfest  time}

The aim in quantum chaos is to obtain information about the high
energy asymptotics as  $\lambda_j \to \infty$ of eigenvalues and
eigenfunctions by connecting information about $U^t$ and $g^t$.
The connection comes from comparing  (\ref{UTEFNS}) and
(\ref{HD}), or by using Egorov's theorem (\ref{EGOROV}). But to
use the hypothesis that $g^t$ is ergodic or chaotic, one needs to
exploit the connection as $\hbar \to 0$ and $t \to \infty$. The
difficulty in quantum chaos is that the approximation of $U^t$ by
$g^t$ is only a good one for $t$ less than the {\it Eherenfest
time}
\begin{equation} \label{EHRENTIME} T_E = \frac{\log |\hbar|}{\lambda_{\max}},
\end{equation}
where $\lambda_{\max}$ is defined in (\ref{LAMBDAMAX}).

Roughly speaking, the idea is that the evolution of a well
constructed ``coherent"  quantum state or particle is a moving
lump that ``tracks along" the trajectory of a classical particle
up to time $T_H$ and then slowly falls apart and stops acting like
a classical particle. Numerical studies of long time dynamics of
wave packets are given in works of E. J. Heller \cite{He,He2} and
rigorous treatments are in Bouzouina-Robert \cite{BouR},
Combescure-Robert \cite{CR} and Schurbert \cite{Schu3}.

The basic result expressed in semi-classical notation is that
there exists $\Gamma > 0$ such that
\begin{equation} \label{EGORWREM} ||U_{\hbar}^{-t} Op(a) U_{\hbar}^t - Op(a \circ
g^t)|| \leq C \hbar \; e^{t \Gamma}. \end{equation} Such an
estimate has long been known if $\Gamma$ is not specified in terms
of the geodesics. It is implicit in \cite{Be} and used explicitly
in \cite{Z4}, based on long time wave equation estimates of
Volovoy. The best constant should be $\Gamma = \lambda_{\max}$.

Thus, one only expects  good joint asymptotics as $\hbar \to 0, T
\to \infty$ for $t \leq  T_E$. As a result, one can only exploit
the approximation of $U^t$ by  $g^t$  for the relatively short
time $T_E$.

\subsection{Modes, quasi-modes and coherent states}

This survey is mainly concerned with eigenfunctions. But one of
the main tools for studying eigenfunctions is the construction of
quasi-modes or approximate eigenfunctions. Lagrangian and coherent
states are closely related to quasi-modes. We refer to
\cite{Dui,CV2,BB,Ra,Ra2} for background.

A quasi-mode $\{\psi_k\}$ of order zero is a sequence of
$L^2$-normalized functions satisfying \begin{equation} \label{QM0}
||(\Delta - \mu_k) \psi_k||_{L^2} = O(1),
\end{equation}  for a sequence of quasi-eigenvalues $\mu_k$. By
the spectral theorem, there must exist true eigenvalues in the
interval $[\mu_k - \delta, \mu_k + \delta]$ for some $\delta > 0$.
Moreover, if $E_{k, \delta}$ denotes the spectral projection for
the Laplacian corresponding to this interval, then
$$ ||E_{k, \delta} \psi_k - \psi_k||_{L^2} = O(k^{-1}). $$
A more general definition is that a quasi-mode  of order
 $s$ is a sequence satisfying
 \begin{equation} \label{QM0s} ||(\Delta - \mu_k)
\psi_k||_{L^2} = O(\mu_k^{-s}).
\end{equation} Then
$$ ||E_{k, \delta} \psi_k - \psi_k||_{L^2} = O(k^{-s -1}). $$
This definition allows for very weak versions of quasi-modes, e.g.
quasi-modes of order $-1$. They carry no information about
eigenvalues since there always exist eigenvalues in intervals of
width one. However, they do carry some information on
eigenfunctions.

In references such as \cite{Dui}, quasi-modes are constructed in
the form of oscillatory integrals or Lagrangian states,
$$\psi_{\lambda}(x) = \int_{\R^k} e^{i \lambda S(x, \xi)} a(x,
\xi) d \xi. $$ The phase generates the Lagrangian submanifold
$$\Lambda_S = \{(x, d_x S): d_{\xi} S(x, \xi) = 0\} \subset T^*M. $$
 If $\Lambda \subset S^*M$ is a closed Lagrangian submanifold invariant
 under the geodesic flow then $\psi_{\lambda}$ is a quasi-mode of order $-1$.
 In favorable circumstances, one can further determine the amplitude so that  $|| (-\Delta - \lambda^2)
\psi_{\lambda}|| = O(\lambda^{-s})$ and construct a high order
quasi-mode. Such a quasi-mode is the `quantization' of $\Lambda$
and the associated sequence concentrates on $\Lambda$.

A coherent state is a special kind of Lagrangian state, or more
precisely isotropic state, which quantizes a `point' in phase
space rather than a Lagrangian submanifold. The simplest examples
are the states $\phi^{x_0, \xi_0}_{\hbar}(x) = \hbar^{-n/2}
e^{\frac{i}{\hbar} \langle x , \xi_0 \rangle} e^{- \frac{|x -
x_0|^2}{\hbar}}$ on $\R^n$, which are localized to the smallest
possible volume in phase space, namely a unit cube of volume
$\hbar$ around $(x_0, \xi_0)$. They are Lagrangian states with
complex phases associated to positive Lagrangian submanifolds of
complexified phase space.

\subsection{\label{RRW} Riemannian random waves}

Eigenfunctions in the chaotic case are never similar to such
quasi-modes or Lagrangian states. Rather, they are conjectured to
resemble random waves. This heuristic was first proposed by M. V.
Berry \cite{Ber}, who had in mind random Euclidean plane waves.
The specific model discussed here was used by the author to show
that eigenfunctions in the ergodic case behave to some degree like
random orthonormal bases of spherical harmonics. We use the term
Riemannian random waves for this model.

We decompose the spectrum of a compact  Riemannian manifold $(M,
g)$  into intervals $I_N = [N, N + 1]$ and define the finite
dimensional Hilbert spaces $\hcal_{I_{N}}$ as the space of linear
combinations of eigenfunctions with frequencies in an interval
$I_N : = [N, N + 1]$. The precise decomposition of $\R$ into
intervals is not canonical and any choice of intervals of finite
width would work as long as the set of closed geodesics of $(M,
g)$ has measure zero. We will not discuss the other case since it
has no bearing on quantum ergodicity.  We  denote by $\{\phi_{N j}
\}_{j = 1}^{d_N}$ an orthonormal basis of $\hcal_{I_N}$ where $d_N
= \dim \hcal_{I_N}$.   We endow the real vector space
$\mathcal{H}_{I_N}$ with the Gaussian probability measure
 $\gamma_N$ defined by
\begin{equation}\label{gaussian}\gamma_N(s)=
\left(\frac{d_N}{\pi }\right)^{d_{N}/2}e^ {-d_{N}|c|^2}dc\,,\qquad
\psi=\sum_{j=1}^{d_{N}}c_j \phi_{N j}, \,\;\; d_{N} = \dim
\hcal_{I_N}.
\end{equation} Here,  $dc$
is $d_{N}$-dimensional real Lebesgue measure. The normalization is
chosen so that $\E_{\gamma_N}\; \langle \psi, \psi \rangle=1$,
where $\E_{\gamma_N}$ is the expected value with respect to
$\gamma_N$. Equivalently,  the $d_N$ real variables $\ c_j$
($j=1,\dots,d_{N}$) are independent identically distributed
(i.i.d.) random variables with mean 0 and variance
$\frac{1}{2d_{N}}$; i.e.,
$$\E_{\gamma_N} c_j = 0,\;\; \quad  \E_{\gamma_N} c_j c_k =
\frac{1}{ d_N}\de_{jk}\,.$$ We note  that the Gaussian ensemble is
equivalent to picking $\psi_N \in \hcal_{I_N}$ at random from the
unit sphere in $\hcal_{I_N}$ with respect to the $L^2$ inner
product.

Numerical results confirming some features of the random wave
model are given in \cite{HR}.

\subsection{Hyperbolic plane}

It is useful to have one concrete example in mind, and we will use
the hyperbolic plane $\HH$ or disc $\D$ as a running model.  In
the disc model $\D=\{z\in\C, |z|<1\}$, the hyperbolic metric has
the form $ds^2=\frac{4|dz|^2}{(1-|z|^2)^2}.$ Its isometry group is
$G = PSU(1,1)$; the stabilizer of $0$ is $K\simeq SO(2)$ and thus
$\D \simeq G/K$. Without comment we will also identify $\D$ with
the upper half plane and $G$ with $PSL(2, \R)$.  In hyperbolic
polar coordinates centered at the origin $0$, the (positive)
Laplacian is the operator
$$- \Lap=\frac{\partial^2}{\partial r^2} +\coth r\frac{\partial}{\partial r}+\frac{1}{\sinh^2 r}\frac{\partial^2}{\partial \theta^2} .$$
 The distance on $\D$ induced by the Riemannian
metric will be denoted $d_\D$. We denote the volume form by
$d(z)$.

The unit tangent bundle $S \D$ of the hyperbolic disc $\D$ may be
identifed with the   unit cosphere bundle $S^*\D$ by means of the
metric. We further identify  $S \D \equiv PSU(1,1)$, since
$PSU(1,1)$ acts freely and transitively on $S \D$. We identify a
unit tangent vector $(z, v)$ with a group element $g$ if $g \cdot
(0, (1,0)) = (z, v)$.

We  denote by $B=\{z\in\C, |z|=1\}$ the boundary at infinity of
$\D$. Then we can also identify  $S\D \equiv \D \times B$. Here,
we identify $(z, b)\in \D \times B$ with the unit tangent vector
$(z, v)$, where $v \in S_z \D$ is the vector tangent to the unique
geodesic through $z$ ending at $b$.

The geodesic flow $g^t$ on $S \D$ is defined by $g^t(z, v) =
(\gamma_v(t), \gamma_v'(t))$ where $\gamma_v(t)$ is the unit speed
geodesic with initial value $(z, v)$. The space of geodesics is
the quotient of $S \D$ by the action of $g^t$. Each geodesic has a
forward endpoint $b$ and a backward endpoint $b'$ in $B$, hence
the  space of geodesics of $\D$ may be identified with $B \times B
\setminus \Delta,$ where $\Delta$ denotes the diagonal in $B
\times B$: To $(b', b)\in B \times B \setminus \Delta$ there
corresponds a unique geodesic $\gamma_{b', b}$ whose forward
endpoint at infinity equals $b$ and whose backward endpoint equals
$b'$. We then have the identification $$S \D \equiv (B \times B
\setminus \Delta) \times \R. $$ The choice of time parameter is
defined -- for instance -- as follows: The point $(b', b, 0)$ is
by definition the closest point to $0$ on $\gamma_{b', b}$ and
$(b', b, t)$ denotes the point $t$ units from $(b, b', 0)$ in
signed distance towards $b$.

\subsubsection{\label{REP} Dynamics and group theory of  $G = SL(2,
\R)$}

The generators of $sl(2, \R)$ are denoted by
$$H = \left( \begin{array}{ll} 1 & 0 \\ & \\ 0 & - 1 \end{array}
\right), \;\;\; V = \left( \begin{array}{ll} 0 & 1 \\ & \\
1 & 0
\end{array} \right), \;\; W =  \left( \begin{array}{ll} 0 & -1 \\ & \\1 & 0 \end{array}
\right).$$
 We denote the associated one parameter subgroups by $A,
A_-, K$. We denote the raising/lowering operators for $K$-weights
by \begin{equation} \label{EPM} E^+ = H + i V, \;\;\; E^- = H - i
V.  \end{equation}  The Casimir operator is then given by $4 \;
\Omega = H^2 + V^2 - W^2$; on $K$-invariant functions, the Casimir
operator acts as the laplacian $\Lap$. We also put
$$X_+ = \left( \begin{array}{ll} 0 & 1 \\ & \\ 0 & 0 \end{array}
\right),\;\;\;X_- = \left( \begin{array}{ll} 0 & 0 \\ & \\ 1 & 0
\end{array} \right), $$
and denote the associated subgroups by $N, N_-$. In the
identification $S \D\equiv PSL(2, \R)$ the geodesic flow is given
by the right action $g^t(g)=ga_t$ of $A$, resp. the horocycle flow
$(h^u)_{u\in\R}$ is defined  by $h^u(g)=gn_u$,  where
\begin{equation*} a_t = \left( \begin{array}{ll} e^{t/2} & 0 \\ & \\
0 & e^{-t/2} \end{array} \right), \;\;\;\; n_u = \left( \begin{array}{ll} 1 & u \\ & \\
0 & 1\end{array} \right). \end{equation*}

The closed orbits of the geodesic flow $g^t$ on $\Gamma \backslash
G$ are denoted $\{\gamma\}$ and are in one-to-one correspondence
with the conjugacy classes of hyperbolic elements of $\Gamma$.  We
denote by $G_{\gamma}$, respectively $\Gamma_{\gamma}$, the
centralizer of $\gamma$ in $G$, respectively $\Gamma$. The group
$\Gamma_{\gamma}$ is generated by an element $\gamma_0$ which is
called  a primitive hyperbolic geodesic.  The length of $\gamma$
is denoted $L_{\gamma} > 0$ and means that
$\gamma$ is conjugate, in $G$, to \begin{equation} \label{agamma} a_{\gamma} = \left( \begin{array}{ll} e^{L_{\gamma}/2} & 0 \\ & \\
0 & e^{-L_{\gamma}/2} \end{array} \right). \end{equation} If
$\gamma = \gamma_0^k$ where $\gamma_0$ is primitive, then we call
$L_{\gamma_0}$ the primitive length of the closed geodesic
$\gamma$.

\subsubsection{\label{NEFA} Non-Euclidean Fourier analysis}

Following  \cite{Hel}, we denote by $ \langle z, b \rangle $ the
signed distance to $0$ of the horocycle through the points $z \in
\D,\, b \in B$. Equivalently,
$$e^{\langle z, b \rangle} = \frac{1 - |z|^2}{|z - b|^2} = P_\D(z, b),$$
 where  $P_\D(z, b)$ is the Poisson kernel of the unit disc.
We denote Lebesgue measure on $B$ by  $|db|$.  We then introduce
the hyperbolic plane waves $e^{(\frac12+ir)\langle z, b \rangle}$,
the analogues of the Euclidean plane waves $e^{i \langle x, \xi
\rangle}$.

The non-Euclidean Fourier transform $\fcal: C_c{\D} \to C(B \times
\R_+)$  is defined by
$$\fcal u(r, b) = \int_{\D} e^{(\frac12-ir)\langle z, b \rangle}
u(z) dVol(z). $$ The inverse Fourier transform is given by
$$\fcal^{-1} g(z)  = \int_B \int_{\R} e^{( \frac12+ir )\langle z, b \rangle} g (r, b)
r \tanh (2\pi r) dr |db|. $$ The integration measure is the
Plancherel measure. $\fcal$ extends to an isometry $\fcal: L^2(\D,
dV) \to L^2(B \times \R_+, r \tanh (2 \pi r) dr db)$.

\subsubsection{Wave kernel on $\D$ and wave kernel on $\X$}

The wave group $U^t = e^{it \sqrt{\Delta}}$ of the positive
Laplacian on $\D$ is thus given in terms of Fourier analysis by
(distribution) integral
$$U^t(z,w) = \int_B \int_{\R} e^{i (\frac{1}{4} + r^2) t} e^{( \frac12+ir )\langle z, b \rangle}
e^{(- \frac12+ir )\langle w, b \rangle} r \tanh (2\pi r) dr |db|.
$$ The integral over $B$ may be eliminated by the stationary phase
method to produce an integral over $\R$ alone; it is the Hadamard
parametrix (\ref{HD}) in this case.

Now let $\Gamma \subset G$ be a discrete subgroup. In applications
to quantum chaos we assume the quotient $\X$ is compact or of
finite area.

The wave kernel on the quotient $\X$ is obtained by automorphizing
the wave kernel on $\D$: \begin{equation} \label{UTGAMMA}
U_{\Gamma}^t(z,w) = \sum_{\gamma \in \Gamma} U_{\Gamma}^t(z,
\gamma w). \end{equation}  This representation is one of the
inputs into the Selberg trace formula. The elements of $\Gamma$
are grouped into conjugacy classes $[\gamma] \in [\Gamma]$. The
conjugacy classes correspond to closed geodesics of $\X$, i.e.
periodic orbits of $g^t$ on $S^* \X$.

Let \begin{equation} \label{THETAT} \Theta_{\Gamma}(T) = \#
\{[\gamma]: L_{\gamma} \leq T\}. \end{equation}  The prime
geodesic theorem asserts that
\begin{equation} \label{PGT} \Theta_{\Gamma}(T) \sim
\frac{e^{h_{top} T}}{T}, \end{equation} where $h_{top}$ is the
topological entropy of $g^t$. In fact, $h_{top} = 1$ for the
hyperbolic case. The exponential growth of the length spectrum
reflects the exponential growth of the geodesic flow.

The Ehrenfest  time (\ref{EHRENTIME}) is implicit in the
exponential growth rate of $\Theta_{\Gamma}(T)$.

\subsubsection{\label{REPRE} Representation theory of $G$ and spectral theory of
$\Lap$}

Let $\Gamma \subset G$ be a co-compact  discrete subgroup, and let
us consider the automorphic eigenvalue problem on $G/K$:
\begin{equation} \label{HYPEIG} \left\{ \begin{array}{l} \Lap \phi =  (\frac14 +r^2)\phi,\\ \\
\phi(\gamma z) = \phi(z) \mbox{ for all } \gamma\in\Gamma \mbox{
and for all } z. \end{array} \right. \end{equation} The solutions
are the eigenfunctions of the Laplacian on the compact surface
$\X=\Gamma\backslash \; G \;/K$. A standard notation for the
eigenvalues is $\lambda^2=s(1- s)$ where $s=\frac12+ir$.
\medskip

Eigenfunctions of the Laplacian are closely connected with the
representation theory of $G$ on $G/\Gamma$. We briefly describe
the representation theory since the problems in quantum chaos have
analogues in the discrete series as well as the unitary prinicipal
series.

In the compact case, we have the decomposition into irreducibles,
$$L^2(\Gamma \backslash G)  = \bigoplus_{j = 1}^S {\mathcal C}_{ ir_j} \oplus \bigoplus_{j = 0}^{\infty}
\pcal_{ir_j} \oplus \bigoplus_{m = 2, \; m \;even}^{\infty}
\mu_{\Gamma}(m) \dcal_{m}^+ \oplus \bigoplus_{m = 2, m \; even
}^{\infty} \mu_{\Gamma}(m) \dcal_{m}^-,
$$ where ${\mathcal C}_{ir_j}$ denotes the complementary series
representation, respectively $\pcal_{ir_j}$ denotes the unitary
principal series representation, in which the Casimir operator  $-
\Omega = - (H^2 + V^2 - W^2)$ equals $s_j(1-s_j ) =
\frac14+r_j^2$. In the complementary series case, $ir_j \in \R$
while in the principal series case $i r_j \in i \R^+$. The
irreducibles are indexed by their $K$-invariant vectors
$\{\phi_{ir_j}\}$, which is assumed to be the given orthonormal
basis of $\Lap$-eigenfunctions. Thus, the multiplicity of
$\pcal_{ir_j}$ is the same as the multiplicity of the
corresponding eigenvalue of $\Lap$.

 Further,
 $\dcal^{\pm}_m$ denotes the holomorphic (respectively
anti-holomorphic) discrete series representation with lowest
(respectively highest) weight $m$, and  $\mu_{\Gamma}(m)$ denotes
its multiplicity; it  depends only on the genus of $\X$. We denote
by $\psi_{m, j}$ ($j = 1, \dots, \mu_{\Gamma}(m))$ a choice of
orthonormal basis of the lowest weight vectors of $\mu_{\Gamma}(m)
\dcal_{m}^+ $ and write $\mu_{\Gamma}(m) \dcal_{m}^+ = \oplus_{j =
1}^{\mu_{\Gamma}(m)} \dcal^+_{m, j}$ accordingly.

There is a a direct integral decomposition for  co-finite
subgroups such as $\Gamma = PSL(2, \Z)$ or congruence subgroups.
The non-compactness gives rise to a  continuous spectral subspace
of Eisenstein series and a discrete spectral subspace of cuspidal
eigenfunctions (which is only known to be non-trivial in the case
of arithmetic $\Gamma$).

\subsubsection{Helgason Poisson formula}

The Fourier transform of an $L^2$ function on $\D$ is an $L^2$
function on $B \times \R_+$. There is an extension of the Fourier
transform and the inversion formula to tempered distributions. We
only consider the case of $\Gamma$-automorphic eigenfunctions
where $\Gamma$ is co-compact.  One then has $\fcal \phi_{ir_j} =
dT_{ir_j}(b) \otimes \delta_{r_j}(r)$, where $dT_{ir_j} $ is a
distribution on the ideal boundary $B$. The inversion formula is
Helgason's Poisson formula,
\begin{equation} \label{POISSON} \phi_{ir}(z) = \int_B e^{(\frac12
+ ir) \langle z, b \rangle } T_{ir, \phi_{ir}} (db),
\end{equation}  for all $z\in\D$.  The kernel $e^{(\frac12 + ir) \langle z, b
\rangle } = P_\D^{(\frac12 + ir)}(z,b)$ is called the generalized
Poisson kernel. The distribution \begin{equation} \label{FOURIER}
T_{ir, \phi_{ir}}(db) =  \sum_{n \in \Z} a_{n}(r) b^n |db|.
\end{equation} is called the boundary value of
$\phi_{ir}$ and is obtained from the Fourier expansion
\begin{equation} \label{GENSPH}  \phi_{ir}(z)  = \sum_{ n \in \Z}
a_{n}(r)  \Phi_{r, n}(z), \end{equation} of $\phi_{ir}$ in the
disc model in terms of the generalized spherical functions,
\begin{equation}  e^{(\frac12+i r) \langle z, b \rangle } = \sum_{ n \in \Z}
\Phi_{r, n}(z) b^n,\;\;\; b \in B.
\end{equation}
Equivalently, the $\Phi_{r,n}$  are the joint eigenfunctions of
$\Delta$ and of $K$.

When $\phi_{ir_j}$ is a $\Gamma$-invariant eigenfunction, the
boundary values $T_{ir_j}(db)$ have the following invariance
property:
\begin{equation}\label{CONFORMAL} \begin{array}{ll} \phi_{ir_j}(\gamma z) = \phi_{ir_j} (z) & \implies
e^{(\frac12+ir_j)\langle \gamma z, \gamma b \rangle} T_{ir_j}(d
\gamma b) =
e^{(\frac12+ir_j)\langle z, b \rangle} T_{ir_j} (d b)\\ &  \\
& \implies T_{ir_j}( d\gamma b) = e^{- (\frac12+ir_j) \langle
\gamma \cdot 0, \gamma \cdot b \rangle} T_{ir_j} (d b) \end{array}
\end{equation}
This follows by the identities
\begin{equation}
\label{ID1} \langle g \cdot z, g \cdot b \rangle = \langle z, b
\rangle + \langle g \cdot 0, g \cdot b \rangle,  \end{equation}
which implies
\begin{equation} \label{ID2} P_\D(g z, g b)\, |d (gb)| = P_\D(z, b)\,
|db|.\end{equation}

An interesting heuristic related to  Berry's random wave
hypothesis (see \S \ref{RRW})  is the conjecture  that the Fourier
coefficients $a_n(r)$ in (\ref{FOURIER}) should behave like
independent Gaussian random variables of mean zero and variance
one. This conjecture was stated explicitly and  tested numerically
by Aurich-Steiner in \cite{AS}. Related tests of the random wave
model for hyperbolic quotients are in Hejhal-Rackner \cite{HR}.

 Otal \cite{O} and Schmid \cite{Sch}  have shown  that
$T_{ir_j}(db)$ is the derivative of a H\"older $C^{\half}$
continuous function $F_{ir_j}$ on $B$. Since its zeroth Fourier
coefficient is non-zero, $T_{ir_j}(db)$ is not literally the
derivative of a periodic function, but it is the derivative of a
function $F_{ir_j}$ on $\R$ satisfying $F_{ir_j}(\theta + 2 \pi) =
F_{ir_j}(\theta) + C_j$ for all $\theta \in \R$. Recall that for
$0 \leq \delta \leq 1$ we say that a $2 \pi$-periodic function $F:
\R \to \C$ is $\delta$-H\"older if $|F(\theta) - F(\theta')| \leq
C |\theta - \theta'|^{\delta}.$ The smallest constant is denoted
$||F||_{\delta}$ and $\Lambda_{\delta}$ denotes the Banach space
of $\delta$-H\"older functions, up to additive constants.

\subsubsection{Boundary values and representation theory \cite{Z3}}

The distributions $dT_{ir_j}$ have a natural interpretation in
representation theory. We define   $e_{ir_j} \in \dcal'(\Gamma
\backslash PSL(2, \R))$ such that
\begin{equation} \label{EJ} e^{(\frac12+ir_j)\langle z, b \rangle} T_{ir_j} (d
b)= e_{ir_j}(z, b) P(z, b) db.
\end{equation}
 The distribution $e_{ir_j}$ is horocyclic-invariant and
$\Gamma$-invariant. It may be expanded in a $K$-Fourier series,
$$e_{ir_j}= \sum_{n \in \Z} \phi_{ir_j, n}, $$
and it is easily seen (cf. \cite{Z2}) that $\phi_{ir_j, 0} =
\phi_{ir_j}$ and that $\phi_{ir_j, n}$ is obtained by applying the
$n$th normalized raising or lowering operator (Maass operator)
$E^{\pm} = H \pm i V$ to $\phi_{ir_j}$. More precisely, one
applies $(E^{\pm})^n$  and multiplies by the normalizing factor
$\beta_{2ir_j, n} = \frac{1}{(2ir_j + 1 \pm 2n) \cdots (2i r_j + 1
\pm 2)}$.

\section{\label{WLSECT}Weyl law and local Weyl law}

We now return to the general case. A basic result in
semi-classical asymptotics  is Weyl's law on counting eigenvalues:
\begin{equation}\label{WL} N(\lambda ) = \#\{j:\lambda _j\leq \lambda \}= \frac{|B_d|}{(2\pi)^d} Vol(M, g)
\lambda^d + R(\lambda), \;\;\; \mbox{where}\;\; R(\lambda) =
O(\lambda ^{d-1}).
\end{equation} Here, $|B_d|$ is the Euclidean volume of the unit
ball and $Vol(M, g)$ is the volume of $M$ with respect to the
metric $g$. Weyl's law says that
\begin{equation} Tr E_{\lambda} \sim  \frac{Vol( |\xi|_g
\leq \lambda)}{(2\pi)^d}, \end{equation} where $Vol$ is the
symplectic volume measure relative to the natural symplectic form
$\sum_{j=1}^d dx_j \wedge d\xi_j$ on $T^*M$. Thus, the dimension
of the space where $H = \sqrt{\Delta} $ is $\leq \lambda$ is
asymptotically the volume where its symbol $|\xi|_g \leq \lambda$.

The growth of the  remainder term depends on the  long time
behavior of $g^t$.  It is sharp on the standard sphere, where all
geodesics are periodic.  By a classical theorem due to
Duistermaat-Guillemin \cite{DG} (in the boundaryless case) and to
Ivrii (in the boundary case), $$R(\lambda) = o(\lambda^{d-1}),
\;\; \mbox{when the set of periodic geodesics has Liouville
measure zero}. $$The remainder is then of small order than the
derivative of the principal term, and one  has asymptotics in
shorter intervals:
\begin{equation} \label{DGIshort} N([\lambda, \lambda + 1]) = \#\{j:\lambda _j  \in  [\lambda, \lambda + 1] \}=
n \frac{|B_d|}{(2\pi)^d} Vol(M, g) \lambda ^{d-1} +o(\lambda
^{d-1}).
\end{equation}
Then the mean spacing between the eigenvalues in this interval is
$\sim C_d Vol(M, g)^{-1} \lambda^{-(d-1)}$, where $C_d$ is a
constant depending on the dimension.

In the case of compact Riemannian manifolds of negative curvature,
a sharper estimate of the remainder is possible:
\begin{equation} \label{NEGCURVWEYL} R(\lambda) =
O(\frac{\lambda^{d-1}}{\log \lambda}). \end{equation} This
remainder is proved using (\ref{UTGAMMA}) and the exponential
growth rate in (\ref{PGT}). The estimate (\ref{NEGCURVWEYL}) was
proved by Selberg in the case of compact hyperbolic quotients and
was generalized to all compact Riemannian manifolds without
conjugate points by B\'erard \cite{Be}. The logarithm in the
remainder is a direct outcome of the fact that one can only use
the geodesic approximation to $U^t$ up to the Ehrenfest time $T_E$
(\ref{EHRENTIME}).

This estimate has not been improved in fifty years, and there are
no better results in the constant curvature case than in the
general negatively curved case.  The remainder estimate does not
rule out the implausible-seeming scenario that in the interval
$[\lambda, \lambda + 1]$ there are only $\log \lambda$ distinct
eigenvalues with multiplicities $\frac{\lambda^{d-1}}{\log
\lambda}.$ In fact, such implausibly large multiplicities do occur
for a sparse set of `Planck constants'  in the analogous case of
quantizations of hyperbolic toral automorphisms (see \cite{FNB}).
The issue of possibly high multiplicity or clustering of
eigenvalues is important in Hassell's scarring result as well as
the Anantharaman-Nonnenmacher results on entropies.

An important generalization is the {\it local Weyl law} concerning
the traces $Tr A E(\lambda)$ where $A \in \Psi^m(M).$ It asserts
that
\begin{equation} \label{LWL} \sum_{\lambda _{j} \leq
\lambda }\langle A\varphi_j, \varphi_j \rangle = \frac{1}{(2
\pi)^d}  \left(\int_{B^*M} \sigma_A dx d\xi \right) \lambda^d +
O(\lambda^{d-1}).
\end{equation}

There is also a  pointwise local Weyl law:

\begin{equation} \label{PLWL} \sum_{\lambda _{j} \leq
\lambda } |\varphi_j(x)|^2 = \frac{1}{(2 \pi)^d} |B^d| \lambda^d +
R(\lambda, x),
\end{equation}
where $R(\lambda, x) = O(\lambda^{d-1})$ uniformly in $x$. When
the periodic geodesics form a set of measure zero in $S^*M$, the
remainders are of order $o(\lambda^{d-1})$, and  one could average
over the shorter interval $[\lambda, \lambda + 1].$ In the
negatively curved case, $$R(\lambda, x) = O
(\frac{\lambda^{d-1}}{\log \lambda}). $$ Combining the Weyl and
local Weyl law, we find the surface average of $\sigma_A$ is a
limit of traces:
\begin{equation} \label{OMEGA} \begin{array}{lll} \omega(A)
&: =&\displaystyle{ \frac{1}{\mu_L (S^{*}M)}\int_{S^*M}\sa d\mu_L} \\
& & \\
 &=& \displaystyle{ \lim_{\lambda
\rightarrow\infty}\frac{1}{N(\lambda )}\sum_{\lambda _{j} \leq
\lambda }\langle A\varphi_j, \varphi_j \rangle}
\end{array} \end{equation}
Here, $d\mu_L$ is Liouville measure on $S^*M$ (\S \ref{GEO}).

\section{\label{INVSTATES} Invariant states defined by eigenfuntions and their quantum limits}

When speaking of `states' in quantum mechanics, one might refer to
a normalized  wave function $\psi$ or alternatively to the matrix
elements $\rho_{\psi}(A) = \langle A \psi, \psi \rangle$ of an
observable in the state. The latter use of `state' is a standard
notion  in C* algebras, and is central in quantum ergodicity and
mixing. The states evidently have  the properties: (i)
$\rho_{\psi}(A^*A) \geq 0; (ii) \rho_{\psi}(I) = 1; (iii)
\rho_{\psi}$ is continuous in the norm topology. The classical
analogue is a probability measure, viewed as a positive normalized
linear functional on $C(S^* M)$. In particular, eigenfunctions
define states on $\Psi^0$ (more correctly, its closure in the norm
topology) as in (\ref{RHOJ}), which we repeat:
\begin{equation} \label{RHOK} \rho_k(A) = \langle A \phi_k, \phi_k
\rangle. \end{equation} We could (and will) also consider
'transition amplitudes' $\rho_{j, k}(A) = \langle A \phi_j, \phi_k
\rangle.$

It is an immediate consequence of the fact that $U^t \phi_j = e^{i
t \lambda_j} \phi_j$ that the diagonal states $\rho_k$ are
invariant under the automorphism $\alpha_t$  (\ref{ALPHAT}):
\begin{equation}\label{INVAR}  \rho_k(U_t A U_t^*) = \rho_k(A). \end{equation}
In general, we denote by $\ecal$  the compact, convex set of
states in the vector space of continuous linear functionals on the
closure of $\Psi^0$ in its norm topology. We denote by
\begin{equation} \ecal_{\R} = \{\rho \in \ecal: \rho \circ \alpha_t = \rho \} \end{equation}
the compact convex set of invariant states.  For simplicity of
notation, we continue to denote the closure by $\Psi^0$, and refer
to \cite{Z6} for a detailed exposition.

From a mathematical point of view, such the  states $\rho_k(A)$
are important because they provide  the simplest means of studying
eigenfunctions: one tests eigenfunctions against observables  by
studying the values $\rho_k(A)$. Many standard inequalities in PDE
(e.g. Carleman estimates) have the form of testing eigenfunctions
against well chosen observables. Of course, the high eigenvalue
asymptotics  is not simple. One would like to know the behavior as
$\lambda_j \to \infty$ (or $\hbar \to 0$ of the diagonal matrix
elements $\langle A \phi_j, \phi_j \rangle$  and the transition
amplitudes $\langle A \phi_i, \phi_j \rangle$ between states for $
A \in \Psi^0(M)$. One of the principal problems of quantum chaos
is the following:

\begin{prob}\label{W*} Determine the set   ${\mathcal Q}$ of `quantum limits', i.e.
weak* limit points of the sequence of invariant eigenfunction
stats $\{\rho_k\}$ or equivalently of the  Wigner distributions
$\{W_k\}$.
\end{prob}

 As will be illustrated below in simple examples,  weak limits
reflect the concentration and oscillation properties of
eigenfunctions.

Off-diagonal matrix elements
\begin{equation} \rho_{jk} (A) = \langle A \phi_i, \phi_j \rangle \end{equation} are also important
as  transition amplitudes between states. They  no longer define
states since $\rho_{jk}(I) = 0$, are no longer positive, and are
no longer invariant. Indeed, $\rho_{jk}(U_t A U_t^*) = e^{i t
(\lambda_j - \lambda_k)} \rho_{jk}(A),  $ so they are eigenvectors
of the automorphism $\alpha_t$ of (\ref{ALPHAT}). A sequence of
such matrix elements cannot have a weak limit unless the spectral
gap $\lambda_j - \lambda_k$ tends to a limit $\tau \in \R$.
Problem \ref{W*}  has the following extension to off-diagonal
elements:
\medskip

 \begin{prob} Determine the set
${\mathcal Q}_{\tau} $ of `quantum limits', i.e.  weak* limit
points of the sequence $\{W_{kj}\}$ of distributions on the
classical phase space $S^*M$, defined by
$$\int_X a dW_{kj} := \langle Op(a) \phi_k, \phi_j \rangle$$
where $\lambda_j - \lambda_k = \tau + o(1)$ and where $a \in
C^{\infty}(S^*M)$, or equivalently of the functionals
$\rho_{jk}$.\end{prob}

\subsection{Simplest properties of quantum limits}

 The first is that  $\langle K \phi_k, \phi_k
\rangle \to 0$ for any compact operator $K$.  Indeed, any
orthonormal basis such as $\{\phi_k\}$ tends to zero weakly in
$L^2$.  Hence $\{K \phi_k\}$ tends to zero in $L^2$ for any
compact operator and in particular the diagonal matrix elements
tend to zero. It follows that for any $A \in \Psi^0(M)$, any limit
of a sequence of $\langle A \phi_k, \phi_k \rangle$ is equally a
limit of $\langle (A + K) \phi_k, \phi_k \rangle$.

Any two choices of $Op$, i.e. of  quantizations of homogeneous
symbols (of order zero) as pseudo-differential operators, are the
same to leading order. Hence their difference is compact. Since a
negative order pseudo-differential operator is compact,  $\qcal$
is independent of the definition of $Op$.

These properties do not use the fact that $\phi_j$ are
eigenfunctions. The next property of the $\rho_k$ are consequences
of the fact $\rho_k \in \ecal_{\R}$.
\begin{prop} $\qcal
\subset \mcal_I$, where $\mcal_I$ is the convex set of invariant
probability measures for the geodesic flow. They are also
time-reversal invariant.
\end{prop}

To see this, we first observe that any weak * limit of of
$\{\rho_k\}$ is a positive linear functional on symbols, which we
identify with homogeneous functions of order zero on $T^*M$ or
with smooth functions on $S^*M$. Indeed, any limit of  $\langle A
\phi_k, \phi_k \rangle$, is bounded by $\inf_K ||A + K||$ (the
infimum taken over compact operators), and for any $A \in \Psi^0$,
$||\sigma_A||_{L^{\infty}} = \inf_K ||A + K||$. Hence any weak
limit is bounded by a constant times $||\sigma_A||_{L^{\infty}} $
and is therefore continuous on $C(S^*M)$. It is a positive
functional since each $\rho_j$ is and hence any limit is a
probability measure. The invariance under $g^t$  follows from
$\rho_k \in \ecal_{\R}$ by Egorov's theorem: any limit of
$\rho_k(A)$ is a limit of $\rho_k(Op (\sigma_A \circ g^t))$ and
hence the limit measure is $g^t$ invariant. Furthermore, the limit
measures  are time-reversal invariant, i.e. invariant under $(x,
\xi) \to (x, - \xi)$ since the eigenfunctions are real-valued.

Problem \ref{W*}  is thus to identify which invariant measures in
$\mcal_I$ show up as weak limits of the functionals $\rho_k$ or
equivalently of the Wigner distributions $dW_k$.  The problem is
that $\mcal_I$ can be a very large set.  Examples of invariant
probability measures for the geodesic flow include:

\begin{enumerate}

\item Normalized Liouville measure $d\mu_L$. In fact, the
functional $\omega$ of (\ref{OMEGA}) is also a state on $\Psi^0$
for the reason explained above. A subsequence $\{\varphijk\}$ of
eigenfunctions is considered diffuse if $\rho_{j_k} \to \omega$.

\item A periodic orbit measure $\mu_{\gamma}$ defined by
$\mu_{\gamma}(A) = \frac{1}{L_{\gamma}} \int_{\gamma} \sigma_A ds$
where $L_{\gamma}$ is the length of $\gamma$. A sequence of
eigenfunctions for which $\rho_{k_j} \to \mu_{\gamma}$ obviously
concentrates (or strongly `scars') on the closed geodesic.

\item A finite convex combination  $\sum_{j =1}^M c_j
d\mu_{\gamma_j}$ of periodic orbit measures.

\item A mixed measure such as $\half d\mu_L + \half \sum_{j =1}^M
c_j d\mu_{\gamma_j}$

\item A delta-function along an invariant Lagrangian manifold
$\Lambda \subset S^*M$. The associated eigenfunctions are viewed
as {\it localizing} along $\Lambda$.

\item There are many additional kinds of singular measures in the
Anosov case.

\end{enumerate}

Thus, the constraint in the Proposition is far from enough to
determine $\qcal$.
\medskip

\noindent{\it To pin down $\qcal$, it is necessary to find more
constraints on the states $\rho_k$ and their weak* limits. We must
use the quantum mechanics of $\rho_k$ to pin down the possible
classical limits.}
\medskip

What possible additional constraints are there on the $\rho_k$? To
date, only two are known.

\begin{itemize}

\item Further symmetries of the $\phi_j$. In general there are
none. But in special cases they exist. The most significant case
is that of Hecke eigenfunctions, which carry an infinite number of
further symmetries. See \S \ref{HECKE} and \cite{LIND,Sound1}.

\item Entropies of limit measures. Lower bounds were obtained by
N. Anantharaman and S. Nonnemacher \cite{A,AN} in the case of $(M,
g)$ with Anosov geodesic flow (see also \cite{ANK,Riv}.)

\end{itemize}

These constraints are of a very different nature. The Hecke
symmetries are special to arithmetic hyperbolic quotients, where
one supplement the geodesic flow--wave group connection with the
Hecke symmetries and connections to number theory.   The entropy
bounds are very general and only use the geodesic flow--wave group
connection.

\subsection{Ergodic sequences of eigenfunctions}

A subsequence $\{\phi_{j_k}\}$  of eigenfunctions is called {\it
ergodic} if the only weak * limit of the sequence of $\rho_{j_k}$
is $d\mu_L$ or equivalently the Liouville  state $\omega$.
 If $d W _{k_j} \to \omega$ then in particular, we have
$$\frac{1}{Vol(M)} \int_E |\phi_{k_j}(x)|^2 dVol \to
\frac{Vol(E)}{Vol(M)} $$ for any measurable set $E$ whose boundary
has measure zero. In the interpretation of $|\phi_{k_j}(x)|^2
dVol$ as the probability density of finding a particle of energy
$\lambda_k^2$ at $x$, this says that the sequence of probabilities
tends to uniform measure.

However,  $W_{k_j} \to \omega$ is much stronger since it says that
the eigenfunctions become diffuse on the energy surface $S^*M$ and
not just on the configuration space $M$. One can quantize
characteristic functions ${\bf 1}_E$ of open sets in $S^*M$ whose
boundaries have measure zero. Then
$$\langle Op({\bf 1}_E) \phi_j, \phi_j \rangle = \mbox{the
amplitude that the particle in energy state} \; \lambda_j^2\;\;
\mbox{lies in E}. $$ For an ergodic sequence of eigenfunctions,
$$\langle Op({\bf 1}_E) \phi_{j_k}, \phi_{j_k} \rangle \to
\frac{\mu_L(E)}{\mu_L(S^*M)}, $$ so that the particle becomes
diffuse, i.e. uniformly distributed on $S^* M$. This is the
quantum analogue of the property of uniform distribution of
typical geodesics of ergodic geodesic flows (Birkhoff's ergodic
theorem).

\subsection{QUE}

The Laplacian $\Delta$ or $(M, g)$ is said to be QUE (quantum
uniquely ergodic) if $\qcal = \{\mu_L\}$, i.e. the only quantum
limit measure for any orthonormal basis of eigenfunctions is
Liouville measure.

\begin{conj} (Rudnick-Sarnak, \cite{RS})
Let $(M, g)$  be a  negatively curved manifold. Then $\Delta$ is
QUE. \end{conj}

In \cite{RS},  the case of arithmetic manifolds is investigated
and the role of the Hecke operators is clarified and exploited. In
particular the arithmetic   QUE Conjecture refers to the limits of
Hecke eigen-states, i.e.   eigenfunctions of the arithmetic
symmetries called Hecke operators. As reviewed in \S \ref{HECKE},
in  this case the conjecture  in all its  forms is more or less
completely solved. E. Lindenstrauss \cite{LIND}  (together with
the recent final step in the noncompact case by Soundararajan
\cite{Sound1}) proves this arithmetic QUE for arithmetic surfaces.
If the multiplicity of the eigenvalues are uniformly bounded then
one can deduce full QUE from arithmetic  QUE, i.e. QUE  for any
orthonormal basis  of eigenfunctions  (see \S \ref{HECKE}). Such
uniform bounds on multiplicities are however far out of the range
of current technology .
   The analogue of arithmetic QUE for Hecek holomorphic forms
    on noncompact arithmetic surfaces has recently been settled
by Holowinsky and Soundararajan \cite{Hol,HS}. In this case,  the
Hecke condition cannot be dropped due to the high multiplicity of
such forms.  Their methods are entirely arithmetical and we wont
discuss them further here.

Although we are not discussing quantum cat maps in detail, it
should be emphasized that quantizations of hyperbolic (Anosov)
symplectic maps of the torus) are not QUE. For a sparse sequence
of Planck constants $\hbar_k$, there exist eigenfunctions of the
quantum cat map which partly scar on a hyperbolic fixed point (see
Faure-Nonnenmacher-de Bi\`evre \cite{FNB}). The multiplicities of
the corresponding eigenvalues are of order $\hbar_k^{-1} / |\log
\hbar|$. It is unknown if anything analogous can occur in the
Riemannian setting, but as yet there is nothing to rule it out.

As pointed out in \cite{Z10},  QUE would follow if one could prove
that quantum limits were invariant under a uniquely ergodic flow
such as the horocycle flow of a compact hyperbolic quotient. A
problem in trying to use this approach is that the horocycle flow
is not Hamiltonian with respect to the standard symplectic form on
$T^* \X$, i.e. it cannot be quantized.  It is however Hamiltonian
with respect to a modified symplectic structure. The modified
symplectic structure corresponds to letting the weight of
automorphic forms vary with the $\Delta$ eigenvalue. As a result,
one does get large sequences of automorphic forms which are
Liouville distributed. But QUE in the standard sense refers to
sequences with fixed weight. It turns out that the weight has to
grow so quickly  for the QUE sequence  that one cannot seem to
relate the QUE result to quantum limits of forms with fixed weight
\cite{Z10}.

\subsection{Simplest example: $S^1$}

The only computable example of ergodic eigenfunctions is the
sequence of normalized eigenfunctions $2 \sin k \pi x, 2 \cos k
\psi x$ on $S^1$. Note that they are real valued; the exponentials
$e^{i k x}$ are not quantum ergodic.

The energy surface of $T^*S^1 = S^1 \times \R_{\xi}$ is the pair
of circles $|\xi| = \pm 1. $ The geodesic flow has two invariant
sets of positive measure (the two components), but the flow is
time-reversal invariant under $(x, \xi) \to (x - \xi)$ and the
quotient flow is ergodic. The real eigenfunctions are invariant
under complex conjugation and therefore the quotient quantum
system is ergodic.

Quantum ergodicity in this case amounts to
$$\frac{1}{\pi} \int_0^{2 \pi} V(x) (\sin k x)^2 dx \to \frac{1}{2
\pi} \int_0^{2 \pi} V dx. $$ This is obvious by writing $\sin k x$
in terms of exponentials and using the Riemann-Lebesgue Lemma.

This simple example illustrates an important aspect of quantum
limits: They are weak limits which owe to the fast oscillation of
the eigenfunction squares. In the limit the oscillating functions
tend to their mean values in the weak sense.

In particular, it illustrates why we only consider squares and not
other powers of eigenfunctions: weak limits are not preserved
under non-linear functionals such as powers. The study of $L^p$
norms of chaotic eigenfunctions is very difficult (there are
 Iwaniec-Sarnak for Hecke eigenfunctions).

\subsection{Example of quantum limits: the flat torus}

The only examples where one can compute quantum limits directly
are the completely integrable ones such as the standard sphere,
torus or symmetric spaces. These examples of course lie at the
opposite extreme from chaotic or ergodic dynamics. We  use the
simplest one, the flat torus $\R^n/\Z^n$, to illustrate the
definition of weak* limits.

  An orthonormal basis of eigenfunctions is
furnished by the standard exponentials $e^{2 \pi i \langle k, x
\rangle}$ with $k \in \Z^n$. Obviously, $|e^{2 \pi i \langle k, x
\rangle}|^2 = 1$, so the eigenfunctions are  diffuse in
configuration space. But  they are far from diffuse in phase
space. For any pseudodifferential operator, $A e^{2 \pi i \langle
k, x \rangle} = a(x, k) e^{2 \pi i \langle k, x \rangle}$ where
$a(x, k)$ is the complete symbol. Thus, $$\langle A e^{2 \pi i
\langle k, x \rangle}, e^{2 \pi i \langle k, x \rangle} \rangle =
\int_{\R^n/\Z^n} a(x, k) dx \sim \int_{\R^n/\Z^n} \sigma_A(x,
\frac{k}{|k|}) dx .
$$
Thus, the Wigner distribution is $\delta_{\xi - k}$.  A
subsequence $e^{2 \pi i \langle k_j, x \rangle}$ of eigenfunctions
has a weak limit if and only if $\frac{k_j}{|k_j|}$ tends to a
limit vector $\xi_0$ in the unit sphere in $\R^n$. In this case,
the associated weak* limit is $\int_{\R^n/\Z^n} \sigma_A(x, \xi_0)
dx $, i.e. the delta-function on the invariant torus $T_{\xi_0}
\subset S^*M$ defined by the constant momentum condition $\xi =
\xi_0.$ The eigenfunctions are said to localize under this
invariant torus for $g^t$.

The invariant torus is a Lagrangian submanifold of $T^*
\R^n/\Z^n$, i.e. a submanifold of dimension $n$ on which the
standard symplectic form $dx \wedge d \xi$ restricts to zero. The
exponentials are special cases of WKB or Lagrangian states
$a_{\hbar} e^{\frac{i}{\hbar} S}$, where $a_{\hbar}$ is a
semi-classical symbol $a_{\hbar} \sim a_j \hbar^j$. The associated
Lagrangian submanifold is the graph $(x, dS(x))$ of $S$. Thus,
$\langle k, x \rangle$ generates the Lagrangian submanifold $\xi =
\frac{k}{|k|}.$

In general, one says that a sequence $\{\phi_{j_k}\}$  of
eigenfunctions concentrates microlocally on a Lagrangian
submanifold $\Lambda \subset S^* M$ if
\begin{equation} \langle A \phi_{j_k}, \phi_{j_k} \rangle \to
\int_{\Lambda} \sigma_A d\nu, \end{equation} for some probability
 measure $\nu$ on $\Lambda$. Necessarily, $\Lambda$ is invariant
under $g^t$.

\subsection{Hyperbolic case}

Using the Helgason Poisson integral formula, the  Wigner
distributions can be expressed in terms of the $dT_{ir_h}$ and
$e_{ir_j}$ as follows. As in \cite{Z3}, we define the hyperbolic
calculus of pseudo-differential operators $Op(a)$ on $\D$ by
$$Op (a) e^{( \frac12+ir)\langle z, b \rangle} = a(z, b, r) e^{(\frac12 +ir)\langle z, b
\rangle}. $$ We assume that the complete symbol $a$ is a
polyhomogeneous function of $r$   in the classical sense that
$$a(z, b, r) \sim \sum_{j = 0}^{\infty} a_j(z, b) r^{-j + m}$$
for some $m$ (called its order). By asymptotics is meant that
$$a(z, b, r) - \sum_{j = 0}^{R} a_j(z, b) r^{-j + m} \in S^{m - R - 1}$$
where $\sigma \in S^k$ if $\sup (1 + r)^{j - k}| D^{\alpha}_z
D^{\beta}_b D_r^j \sigma(z,b, r)| < +\infty$ for all compact set
and for all $\alpha,\beta, j$.

 The non-Euclidean Fourier
inversion formula then extends the definition of $Op(a)$ to
$C_c^{\infty}(\D)$:
$$Op(a)u(z)=\int_B \int_{\R} a(z, b, r)e^{(\frac12+ir)\langle z, b \rangle} \fcal u(r, b)
r \tanh(2 \pi r) dr |db|.$$

A key property of $Op$ is that $Op(a)$ commutes with the action of
an element $\gamma\in G$ ($T_\gamma u(z)=u(\gamma z)$)  if and
only if $a(\gamma z, \gamma b, r) = a(z, b, r)$.
$\Gamma$-equivariant pseudodifferential operators then define
operators on the quotient $\X$.

By  the Helgason-Poisson formula one has  a relative explicit
formula for the Wigner distributions
 $W_{ir_j}\in \dcal'(S^* \X)) $  defined by
\begin{equation} \label{WIGDEF} \langle a, W_{ir_j} \rangle  = \int_{S^*\X} a(g)W_{ir_j}(dg) :=\langle Op(a)\phi_{ir_j},
\phi_{ir_j}\rangle_{L^2(\X)},\;\;\; a \in C^{\infty}(S^* \X).
\end{equation}
Equivalently we have \begin{equation} \label{WE} W_{ir_j} =
\phi_{ir_j} e_{ir_j}. \end{equation} Thus, the Wigner
distributions are far more diffuse in phase space than in the case
of the flat torus. But this does not rule out that their weak*
limits could have a singular concentration.

\subsubsection{Patterson-Sullivan distributions}

Egorov's theorem implies that Wigner distributions tend to
invariant measures for the geodesic flow. The question arises
whether there exist $g^t$-invariant distributions constructed from
eigenfunctions which are asymptotic to the Wigner distributions.
In \cite{AZ} such distributions were constructed on hyperbolic
surfaces and termed ``Patterson-Sullivan distributions" by analogy
with their construction of boundary measures associated to ground
states on infinite volume hyperbolic manifolds.

\begin{defin}
\label{PS1}The Patterson-Sullivan distribution associated to a
real eigenfunction  $\phi_{ir_j}$ is the distribution on $B \times
B \setminus \Delta$ defined by
$$ps_{ir_j}(db', db):= \frac{T_{ir_j}(db) T_{ir_j}(db')}{|b - b'|^{1 +  2i r_j}}$$
\end{defin}

If $\phi_{ir_j}$ is $\Gamma$-automorphic, then $ps_{ir_j}(db',
db)$ is $\Gamma$-invariant and time reversal invariant.

 To obtain a $g^t$-invariant distribution on $S^*\X$, we  tensor $ps_{ir_j}$ with $dt$. We then
 normalize by dividing by the integral against $1$.  The result is an  invariant
distribution $\hat{PS}_{ir_j}$ for $g^t$ constructed as a
quadratic expression in the eigenfunctions. In \cite{AZ} Theorem
1.1,  it is proved (theorem that
$$\int_{S^* \X} a \hat{PS}_{ir_j} = \int_{S^* \X} a W_{ir_j} +
O(r_j^{-1}). $$ Hence the quantum limits problem is equally one of
determining the weak* limits of the Patterson-Sullivan
distributions. It is shown in \cite{AZ} that they are residues of
dynamical L-functions and hence have a purely classical
definition.

In fact, there is an explicit intertwining operator $L_{r_j}$
mapping $PS_{ir_j} \to W_{ir_j}$ and we have
\begin{equation}\label{WIGPS}  \langle a, W_{ir_j} \rangle = \langle a,
\hat{PS}_{ir_j} \rangle + r_j^{-1} \langle L_2(a), \hat{PS}_{ir_j}
\rangle + O(r_j^{-2}). \end{equation}

\section{\label{QuE} Quantum ergodicity and mixing of eigenfunctions}

In this section, we review a basic result on quantum ergodicity.
We assume that the geodesic flow of $(M, g)$ is ergodic.
Ergodicity of $g^t$  means that Liouville measure $d\mu_L$ is an
ergodic measure for $g^t$ on $S^*M$, i.e. an extreme point of
$\mcal_I$. That is, there any $g^t$-invariant set has Liouville
measure zero or one. Ergodicity is a spectral property of the
operator $V^t$ of (\ref{VT}) on $L^2(S^*M, d\mu_L)$, namely that
$V^t$ has $1$ as an eigenvalue of multiplicity one. That is, the
only invariant $L^2$ functions (with respect to Liouville measure)
are the constant functions.

 In
this case, there is a general result which originated in the work
of A. I. Schnirelman and was developed into the following theorem
by S. Zelditch, Y. Colin de Verdi\`ere on manifolds without
boundary  and by P. G\'erard-E. Leichtnam and S. Zelditch-M.
Zworski on manifolds with boundary.

\begin{theo} \label{QE}   Let $(M,g)$ be a compact
Riemannian manifold (possibly with boundary), and let
$\{\lambda_j, \phi_j\}$ be the spectral data of its Laplacian
$\Delta.$ Then the geodesic flow
 $G^t$ is ergodic  on $(S^*M,d\mu_L)$ if and only if, for every
$A \in \Psi^o(M)$,  we have:
\medskip

\begin{itemize}

 \item (i)\; $\lim_{\lambda \rightarrow \infty}\frac{1}{N(\lambda)}
 \sum_{\lambda_j \leq \lambda}
|(A\phi_j,\phi_j)-\omega(A)|^2=0.$
\medskip

 \item (ii)\;$(\forall \epsilon)(\exists \delta)
\limsup_{\lambda \rightarrow \infty} \frac{1} {N(\lambda)}
\sum_{{j \not= k: \lambda_j, \lambda_k \leq \lambda}\atop {
|\lambda_j - \lambda_k| < \delta}} |( A \phi_j, \phi_k )|^2 <
\epsilon $
\end{itemize}

\end{theo}

The diagonal result may be interpreted as a variance result for
the local Weyl law. Since all the terms are positive, the
asymptotic  is equivalent to the existence of a s subsequence
$\{\phi_{j_k}\}$ of eigenfunctions whose indices $j_k$ have
counting density one for which $\langle A \phi_{j_k},
\phi_{j_k}\rangle \to \omega(A)$ for any $A \in \Psi^0(M)$. As
above, such a sequence of eigenfunctions is called ergodic. One
can sharpen the results  by averaging over eigenvalues in the
shorter interval $[\lambda, \lambda + 1]$  rather than in $[0,
\lambda]$.

The off-diagonal statement was proved in \cite{Z9} and the fact
that its proof can be reversed to prove the converse direction was
observed by Sunada in \cite{Su}. A generalization to finite area
hyperbolic surfaces is in \cite{Z8}.

  The first
statement (i) is essentially a convexity result. It remains true
if one replaces the square by any convex function $F$ on the
spectrum of $A$,
\begin{equation} \label{CONVEX} \frac{1}{N(E)} \sum_{\lambda_j
\leq E} F (\langle A \phi_k, \phi_k\rangle  - \omega(A)) \to 0.
\end{equation}

\subsection{Quantum ergodicity in terms of operator time and space
averages}

The diagonal variance asymptotics may be interpreted as a relation
between operator time and space averages.

\noindent{\bf Definition}~~~{\it Let $A \in \Psi^0$ be an
observable and define its time average to be:
 $$\langle A \rangle := \lim_{T \rightarrow \infty} \langle A \rangle_T, $$
 where $$\langle A \rangle_T : =  \frac{1}{2T} \int_{-T}^T
U^{t} A U^{-t} dt$$ and its space average to be scalar operator
$$\omega (A) \cdot I$$ }
Then  Theorem \ref{QE} (1) is (almost) equivalent to,
\begin{equation} \langle A \rangle = \omega(A) I +
K,\;\;\;\;\;\;\mbox{where}\;\;\;\;\; \lim_{\lambda \rightarrow
\infty} \omega_{\lambda}(K^*K) \rightarrow 0, \end{equation} where
$\omega_{\lambda}(A) = \frac{1}{N(\lambda)}  Tr E(\lambda) A. $
Thus, the time average equals the space average plus a term $K$
which is semi-classically small in the sense that its
Hilbert-Schmidt norm square $||E_{\lambda} K||_{HS}^2$ in the span
of the eigenfunctions of eigenvalue $\leq \lambda$ is
$o(N(\lambda)).$

This is not exactly equivalent to Theorem \ref{QE} (1) since it is
independent of the choice of orthonormal basis, while the previous
result depends on the choice of basis. However, when all
eigenvalues have multiplicity one, then the two are equivalent. To
see the equivalence, note that $\langle A \rangle$ commutes with
$\sqrt{\Delta}$ and hence is diagonal in the basis $\{\phi_j\}$ of
joint eigenfunctions of $\langle A \rangle$ and of $U_t$. Hence
$K$ is the diagonal matrix with entries $\langle A \phi_k,
\phi_k\rangle - \omega(A)$. The condition is therefore equivalent
to
$$\lim_{E \rightarrow \infty} \frac{1}{N(E)} \sum_{\lambda_j \leq E} |\langle A
\phi_k, \phi_k\rangle  - \omega(A)|^2 = 0.$$

\subsection{Heuristic proof of Theorem \ref{QE} (i) }

There is a simple picture of eigenfunction states which makes
Theorem \ref{QE} seem obvious. Justifying the picture   is more
difficult  than the formal proof below but the reader may find it
illuminating and convincing.

First, one should re-formulate the  ergodicity of $g^t$ as a
property of the Liouville measure $d\mu_L$: ergodicity is
 equivalent to the statement $d\mu_L$ is an extreme point of the compact convex set $\mcal_I$.
Moreover, it  implies that the Liouville state $\omega$ on
$\Psi^0(M)$ is an extreme point of the compact convex set
${\mathcal E}_{\R}$ of invariant states for $\alpha_t$ of
(\ref{ALPHAT}); see \cite{Ru} for background. But the local Weyl
law says that $\omega$ is also the limit of the convex combination
$\frac{1}{N(E)} \sum_{\lambda_j \leq E} \rho_j.$  An extreme point
cannot be written as a convex combination of other states unless
all the states in the combination are equal to it. In our case,
$\omega$ is only a limit of convex combinations so it need not
(and does not) equal each term. However, almost all terms in the
sequence must tend to $\omega$, and that is equivalent to (1).

One could  make this argument rigorous by considering whether
Liouville measure is an {\it exposed point} of $\ecal_I$ and
$\mcal_I$. Namely, is there a linear functional $\Lambda$ which is
equal to zero at $\omega$ and is $< 0$ everywhere else on
$\ecal_I$? If so, the fact that $ \frac{1}{N(E)} \sum_{\lambda_j
\leq E} \Lambda(\rho_j) \to 0$ implies that $\Lambda(\rho_j) \to
0$ for a subsequence of density one. For one gets an obvious
contradiction if $\Lambda(\rho_{j_k}) \leq - \epsilon < 0$ for
some $\epsilon > 0$ and a subsequence of positive density. But
then $\rho_{j_k} \to \omega$ since $\omega$ is the unique state
with $\Lambda(\rho) = 0$.

In \cite{J} it is proved that Liouville measure (or any ergodic
measure) is exposed in $\mcal_I$. It is stated in the following
form: For any ergodic invariant probability measure $\mu$, there
exists a continuous function $f$ on $S^* M$ so that $\mu$ is the
unique $f$-maximizing measure in the sense that
$$\int f d\mu = \sup\left\{ \int f dm : m \in \mcal_I \right\}. $$
To complete the proof, one  would need to show that the extreme
point $\omega$ is exposed in $\ecal_I$ for the C* algebra defined
by the norm-closure of $\Psi^0(M)$.

\subsection{ Sketch of Proof of Theorem \ref{QE} (i)}

We now sketch the proof of (\ref{CONVEX}). By time averaging, we
have
\begin{equation} \sum_{\lambda_j \leq E} F (\langle A \phi_k, \phi_k\rangle
- \omega(A)) = \sum_{\lambda_j \leq E} F (\langle \langle A
\rangle_T - \omega(A) \phi_k, \phi_k\rangle ). \end{equation}  We
then apply the Peierls--Bogoliubov inequality
$$\sum_{j=1}^n F ((B \phi_j, \phi_j)) \leq {\rm Tr\,} F (B)$$
with $B = \Pi_E [\langle A \rangle_T - \omega(A)]\Pi_E $ to get:
\begin{equation} \sum_{\lambda_j \leq E} F (\langle \langle A \rangle_T -
\omega(A) \phi_k, \phi_k\rangle  ) \leq {\rm Tr\,} F (\Pi_E
[\langle A \rangle_T - \omega(A)]\Pi_E ).\end{equation}

  Here, $\Pi_E$ is the
spectral projection for $\hat{H}$ corresponding to the interval
$[0, E].$  By the Berezin inequality   (if $F(0) = 0$):
$$\begin{array}{lll} \frac{1}{N(E)} {\rm Tr\,} F (\Pi_E [\langle A \rangle_T - \omega(A)]\Pi_E )
&\leq& \frac{1}{N(E)}  {\rm Tr\,} \Pi_E F ([\langle A \rangle_T -
\omega(A)]) \Pi_E \\ && \\ &= & \omega_E(\phi(\langle A \rangle_T
- \omega(A))). \end{array}$$ As long as $F$ is smooth, $F(\langle
A \rangle_T - \omega(A))$ is a pseudodifferential operator of
order zero with principal symbol $F(\langle \sigma_A \rangle_T -
\omega(A)).$  By the assumption that $\omega_E \rightarrow \omega$
we get
$$\lim_{E \rightarrow \infty}\frac{1}{N(E)} \sum_{\lambda_j \leq E}
F (\langle A \phi_k, \phi_k\rangle  - \omega(A)) \leq \int_{\{H =
1\}} F(\langle \sigma_A \rangle_T - \omega(A)) d\mu_L.$$ As $T
\rightarrow \infty$ the right side approaches $\phi (0)$ by the
dominated convergence theorem and by Birkhoff's ergodic theorem.
Since the left hand side is independent of $T$, this implies that
$$\lim_{E \rightarrow \infty}\frac{1}{N(E)} \sum_{\lambda_j \leq E}
F (\langle A \phi_k, \phi_k\rangle  - \omega(A)) = 0 $$ for any
smooth convex $F$ on ${\rm Spec}(A)$ with  $F (0) = 0.$ \qed
\bigskip

This proof  can only be used directly for scalar Laplacians on
manifolds without boundary, but it still works as a template in
more involved situations. For instance, on  manifolds with
boundary, conjugation by the wave group is not a true automorphism
of the observable algebra. In quantum ergodic restriction theorems
(see \S \ref{BQE}), the appropriate conjugation is an endomorphism
but not an automorphism.  Or when $\Delta$ has continuous spectrum
(as in finite area hyperbolic surfaces), one must adapt the proof
to states which are not $L^2$-normalized \cite{Z8}.

\subsection{QUE in terms of time and space averages}

The quantum unique ergodicity problem (the term is due to
Rudnick-Sarnak \cite{RS}) is the following:

\begin{prob} Suppose the geodesic flow $g^t$ of $(M, g)$ is ergodic on
$S^*M$. Is the operator $K$ in
$$\langle A \rangle = \omega(A) + K$$
 a compact operator? Equivalently is $\qcal = \{d\mu_L\}$?  In this case,  $\sqrt{\Delta}$ is said to be QUE (quantum uniquely ergodic)  \end{prob}

Compactness of $K$ implies that $\langle K \phi_k,  \phi_k \rangle
\to 0$, hence $\langle A \phi_k, \phi_k \rangle \to \omega(A)$
along the entire sequence.

Rudnick-Sarnak conjectured that $\Delta$ of negatively curved manifolds are QUE, i.e. that for any orthonormal
basis of eigenfunctions, the Liouville measure is the only quantum limit \cite{RS}.

\subsection{Converse QE}

So far we have not mentioned Theorem \ref{QE} (2).  An interesting
open problem is the extent to which (2) is actually necessary for
the equivalence to classical ergodicity.

\begin{prob}  Suppose that $\sqrt{\Delta}$ is quantum ergodic in the sense that (1) holds
in Theorem \ref{QE}. What are the properties of the geodesic flow
$g^t$. Is it ergodic (in the generic case)?
\end{prob}

In the larger class of Schr\"odinger operators, there is a simple
example of a Hamiltonian system which is quantum ergodic but not
classically ergodic: namely, a Schr\"odinger operator with a
symmetric double well potential $W$. That is, $W$ is a $W$ shaped
potential with two wells and a $\Z_2$ symmetry exchanging the
wells. The low energy levels consist of two connected components
interchanged by the symmetry, and hence the classical Hamiltonian
flow is not ergodic. However, all eigenfunctions of the
Schr\"odinger operator $- \frac{d^2}{dx^2} + W$ are either even or
odd and thus have the same mass in both wells. It is easy to see
that the quantum Hamiltonian is quantum ergodic.

Recently, B. Gutkin \cite{Gut} has given a two dimensional example
of a domain with boundary which is quantum ergodic but not
classically ergodic and which is a two dimensional analogue of a
double well potential. The domain is a so-called hippodrome
(race-track) stadium. Similarly to the double well potential,
there are two invariant sets interchanged by a $\Z_2$ symmetry.
They correspond to the two orientations with which the race could
occur. Hence the classical billiard flow on the domain is not
ergodic. After dividing by the $\Z^2$ symmetry the hippodrome has
ergodic billiards, hence by Theorem \ref{QE}, the quotient domain
is quantum ergodic. But  the The eigenfunctions are again either
even or odd. Hence the hippodrome is quantum ergodic but not
classically ergodic.

Little is known about converse quantum ergodicity in the abscence
of symmetry.  It is known that if there exists an open set in
$S^*M$ filled by periodic orbits, then the Laplacian cannot be
quantum ergodic (see \cite{MOZ} for recent results and
references). But it is not even known at this time whether
 KAM systems, which have Cantor-like   invariant sets of positive measure, are not
 quantum ergodic. It is known that there exist a positive proportion of approximate
eigenfunctions (quasi-modes) which localize on the invariant tori,
but  it has not been proved that a positive proportion of actual
eigenfunctions have this localization property.
\bigskip

\subsection{\label{QWMS} Quantum weak mixing}

There are parallel results on quantizations of weak-mixing
geodesic flows. We recall that  the geodesic flow of $(M, g)$ is
weak mixing if the operator $V^t$ has purely continuous spectrum
on the orthogonal complement of the constant functions in
$L^2(S^*M, d\mu_L)$.

\begin{theo} \label{QWM} \cite{Z8} The geodesic flow  $g^t$ of $(M, g)$
is weak mixing if and only if the conditions (1)-(2) of Theorem
\ref{QE}  hold and additionally, for any $A \in \Psi^o(M)$,
$$(\forall \epsilon)(\exists \delta)
\limsup_{\lambda \rightarrow \infty} \frac{1} {N(\lambda)}
\sum_{{j\not= k: \lambda_j, \lambda_k \leq \lambda}\atop {
|\lambda_j - \lambda_k-\tau| < \delta}} |( A \phi_j, \phi_k )|^2 <
\epsilon \;\;\;\;\;\;\; (\forall \tau  \in \R )$$
\end{theo}

The restriction $j\not =k$ is of course redundant unless $\tau =
0$, in which case the statement coincides with quantum ergodicity.
This result follows from the general asymptotic formula, valid for
any compact Riemannian manifold $(M, g)$, that \begin{equation}
\label{QMF} \begin{array}{l}  \frac{1}{N(\lambda)}  \sum_{i \not=
j, \lambda_i, \lambda_j \leq \lambda} |\langle A \phi_i, \phi_j
\rangle|^2 \left|\frac{\sin T(\lambda_i -\lambda_j -\tau)}
{T(\lambda_i -\lambda_j -\tau)}\right|^2 \\ \\
 \sim ||\frac{1}{2T}
\int_{- T}^T e^{i t \tau} V_t(\sigma_A) ||_2^2 - |\frac{\sin T
\tau}{T \tau}|^2 \omega(A)^2. \end{array} \end{equation}  In the
case of weak-mixing geodesic flows, the right hand side $\to 0$ as
$T \to \infty$. As with diagonal sums, the sharper result is true
where  one averages  over the short intervals $[\lambda, \lambda +
1]$.

Theorem \ref{QWM}  is based on expressing the spectral measures of
the geodesic flow in terms of matrix elements. The main limit
formula is:

\begin{equation} \label{SPECMEAS} \int^{\tau +\varepsilon }_{\tau-\varepsilon
} d\mu_{\sigma_A}:=\lim_{\lambda \rightarrow
\infty}\frac{1}{N(\lambda )}\sum_{i,j: \;\; \lambda _j\leq
\lambda,  \;\; |\lambda _i-\lambda _j-\tau|<\varepsilon
\\}\;
 |\langle A\varphi_i,
\varphi_j \rangle|^2\;\;,  \end{equation} where $d\mu_{\sigma_A}$
is the spectral measure for the geodesic flow corresponding to the
principal symbol of $A$,  $\sigma_A \in C^{\infty} (S^*M, d\mu_L)$.
Recall that the spectral measure of $V^t$ corresponding to $f\in
L^2$ is the measure $d\mu_f$ defined by
$$\langle V^t f,f \rangle_{L^2(S^*M)}  = \int_{\R} e^{\oit\tau} d\mu_f(\tau)\;.$$

\subsection{Evolution of Lagrangian states}

In this section, we briefly review results on evolution of
Lagrangian states and coherent states. We follow in particular the
article of R. Schubert \cite{Schu3}.

A simple Lagrangian or WKB state has the form $\psi_{\hbar}(x) =
a(\hbar, x) e^{\frac{i}{\hbar} S(x)}$ where $a(\hbar, x)$ is a
semi-classical symbol $a \sim \sum_{j = 0}^{\infty} \hbar^j
a_j(x). $ The phase $S$ generates the Lagrangian submanifold $(x,
dS(x)) \subset T^*M$.

It is proved in Theorem 1 of \cite{Schu3} that if $g^t$ is Anosov
and if $\Lambda$ is transversal to the stable foliation $W^s$
(except on a set of codimension one), then there exists $C, \tau
> 0$ so that for every smooth density on $\Lambda$ and every
smooth function $a \in C^{\infty}(S^*M)$, the Lagrangian state
$\psi$ with symbol $\sigma_{\psi}$ satisfies,
\begin{equation} \left| \langle U^t \psi, A U^t \psi \rangle -
\int_{S^*M} \sigma_A  d\mu_L \int_{\Lambda} |\sigma_{\psi}|^2
\right| \leq C h e^{\Gamma |t|} + c e^{- t \tau}. \end{equation}
In order that the right side tends to zero as $\hbar \to 0, t \to
\infty$ it is necessary and sufficient that
$$t \leq \frac{1 - \epsilon}{\Gamma} |\log \hbar|. $$

\section{\label{HYP} Concentration of eigenfunctions around hyperbolic closed geodesics}

As mentioned above, the  quantum ergodicity Theorem \ref{QE}
leaves open the possible existence of a sparse (zero density)
subsequence of eigenfunctions which `weakly scar' on a hyperbolic
orbit $\gamma$  in the sense that its quantum limit $\nu_0$
contains a non-trivial multiple of the periodic orbit measure $c
\mu_{\gamma}$ as a non-zero ergodic component. The Anantharaman
entropy bound shows that when $(M, g)$ is Anosov, there cannot
exist such a sequence of eigenfuntions (or even quasi-modes) which
tend to $\mu_{\gamma}$ itself, but a quantum limit could have the
form $c_1 \mu_{\gamma} + c_2 \mu_L$ for certain $c_1, c_2 $
satisfying $c_1 + c_2 = 1$. The question we address in this
section is the possible mass profile of such a scarring
eigenfunction in a neighborhood of $\gamma$. More precisely, {\it
how much mass does $\phi_{\lambda}$ have in a shrinking
$\hbar^{1/2 - \epsilon}$ tube around a hyperbolic closed
geodesic?} This question will surface again in \S \ref{FGAMMA}.

We first note that the existence of a quantum limit of the form
$c_1 \mu_{\gamma} + c_2 \mu_L$ for $(M, g)$ with Anosov geodesi
flow is not so implausible. Analogous eigenfunction sequences do
exist for the so-called quantum cat map \cite{FNB,FN}. And
exceptional sequences `scarring' on a certain 1-parameter family
of periodic orbits exists for the Bunimovich stadium \cite{Has}.
At this time,
 there are no known examples of sequences of
eigenfunctions or quasi-modes  for $(M, g)$ with ergodic geodesic
flow that `weakly scar' along a hyperbolic closed orbit $\gamma$
and no results prohibiting them.

For simplicity, assume that  $(M, g)$ is a Riemannian manifold of
dimension $2$ with a closed geodesic $\gamma $ of length $L$. We
assume $\gamma$ is an embedded (non self-intersecting) curve.  If
$\phi_j$ is a Laplace eigenfunction, we  define  the mass profile
of $\phi_{j}$ near $\gamma$ to be the function
$$M(\phi_j) (r) = \int_{d(x, \gamma) = r} |\phi_j|^2 dS, $$
where $dS \wedge d r = dV$.

Before considering possible mass profiles of eigenfunctions
scarring on hyperbolic closed geodesics, let us recall the
opposite and much better known case of scarring of  Gaussian beams
along elliptic closed geodesics on surfaces \cite{Ra,Ra2,BB}. When
$\gamma$ is an elliptic closed geodesic, then there always exists
a sequence of quasi-modes (Gaussian beams) which concentrates on
$\gamma$. As the name suggests, Gaussian beams $\psi_{\lambda}$
oscillate like $e^{i \lambda s}$ along $\gamma$ and resemble
Gaussians $\sqrt{\lambda} e^{- \lambda \langle A(s) y, y\rangle} $
(for some positive symmetric matrix $A(s)$) in the transverse
direction with height $\lambda^{1/4}$ and concentrated in a
$\frac{1}{\sqrt{\lambda}}$ tube around $\gamma$. Thus, the mass
profile is a Gaussian probability measure with mean zero and
variance $\lambda^{-1/2}$. The local model for such quasi-modes is
that of a harmonic oscillator in the fibers of $N_{\gamma}$. One
can construct the quasi-mode so that it is of infinite order. Of
course, stable elliptic orbits do not exist when $(M, g)$ has
ergodic geodesic flow (by the KAM theorem).

Now consider hyperbolic closed orbits. It was pointed out by
Duistermaat \cite{Dui} (section 1.5) that one cannot construct
analogous quasi-modes associated to hyperbolic closed geodesics as
Lagrangian states.  The stable/unstable manifolds $\Lambda_{\pm}$
of $\gamma$, containing the geodesics which spiral in towards
$\gamma$, are invariant Lagrangian submanifolds, but the only
invariant half-density on the Lagrangians is the `delta'-density
on the closed geodesic.

Further, there are apriori limitations on the degree to which
eigenfunctions sequences can concentrate around  hyperbolic closed
geodesic on any $(M, g)$.

\begin{theo} \label{BZ} \cite{BZ,Chr,CVP} Let $(M, g)$ be a compact Riemannian manifold, and let
$\gamma$ be a hyperbolic closed geodesic. Let $U$ be any tubular
neighborhood of $\gamma$ in $M$.  Then for any eigenfunction
$\phi_{\lambda}$, there exists a constant $C$ depending only on
$U$ such that
$$\int_{M \backslash U} |\phi_{\lambda}|^2 dV_g  \geq \frac{C}{\log \lambda}
||\phi_{\lambda}||^2_{L^2}. $$

More generally, let  $A \in \Psi^0(M)$ be a pseudo-differential
orbit whose symbol equals one in a neighborhood of $\gamma$ in
$S^*_g M$ and equals zero outside another neighborhood. Then for
any eigenfunction $\phi_{\lambda}$ $||(I - A)
\phi_{\lambda}||_{L^2} \geq \frac{C}{\sqrt{\log \lambda}}
||\phi_{\lambda}||_{L^2}. $
\end{theo}

This allows scarring sequences along a hyperbolic orbit to occur,
it just limits the rate at which the mass concentrates near
$\gamma$. It implies  that the mass profile of a sequence of
eigenfunctions concentrating on a hyperbolic close geodesic has
``long tails'', i.e. there is a fairly large amount of mass far
away from the geodesic, although sequences with the quantum limit
$\mu_{\gamma}$  must tend to zero outside of any tube around the
closed geodesic. Note that this result makes no dynamical
hypotheses. It applies equally to $(M, g)$ with integrable
geodesic flow and to those with Anosov geodesic flow. To the
author's knowledge, there do not exist more precise results in the
Anosov case.

An obvious question at this point is whether there are any
examples of $(M, g)$, with any type of geodesic flow, possessing a
sequence of eigenfunctions scarring on a closed hyperbolic orbit.
The answer to this question is `yes'. It is simple to see that
such eigenfunctions exist  in the opposite extreme of completely
integrable systems, for instances surfaces of revolution like
peanuts with hyperbolic waists.
 Examples include joint eigenfunctions
of the square root of the Laplacian and rotation on surfaces of
revolution with a hyperbolic waist. A truncated  hyperbolic
cylinder is another example studied in \cite{CVP}.  In this case,
the joint spectrum fills out the image of the moment map
$(p_{\theta}, |\xi|): T^* M \to \R^2$, where $p_{\theta}(x, \xi) =
\langle \xi, \frac{\partial}{\partial \theta} \rangle$ is the
angular momentum. At critical distances to the axis of rotation,
the lattitude circles are closed geodesics and the level set of
the moment map becomes singular. If the surface is shaped like a
peanut, the waist is a hyperbolic closed geodesic. Joint
eigenfunctions whose joint eigenvalues are asymptotic to the
singular levels of the moment always exist. The modes concentrate
on the level sets of the moment map, and in fact they concentrate
on the hyperbolic closed geodesic.

\subsection{\label{MASSGAMMA} Mass concentration of special eigenfunctions on
hyperbolic orbits in the quantum integrable case}

The mass profile of scarring eigenfunctions near a hyperbolic   in
the completely integrable case is studied in \cite{CVP} on tubes
of fixed radius and in  \cite{NV,TZ2} on tubes of shrinking
radius. Let $\gamma \subset S^* M$ be a closed hyperbolic geodesic
of an $(M, g)$ with
 completely integrable geodesic flow and for which $\Delta_g$ is
 quantum integrable (i.e. commutes with a maximal set of
 pseudo-differential operators; see \cite{TZ2} for background). We
 then consider joint eigenfunctions $\Delta_g$ and  of these
 operators. It is known (see e.g. \cite{TZ2}, Lemma 6) that there
 exists a special sequence of eigenfunctions concentrating on the
 momentum level set of $\gamma$. We will call them (in these notes) the
 {\it $\gamma$-sequence}.

  Assume for simplicity that the moment level set of $\gamma$ just
consists of the orbit together with its stable/unstable manifolds.
Then it is proved in \cite{TZ2} that the mass of $\phi_{\mu}$ in
the shrinking tube of radius $h^{\delta}$ around $\gamma$ with
$\delta < \half$ is $\simeq (1 - 2 \delta)$ (see also \cite{NV}
for a closely related result in two dimensions). Thus, the mass
profile of such scarring integrable eigenfunctions only differs by
the numerical factor $(1 - 2 \delta)$ from the mass profile   of
Gaussian beams. The difference is that the `tails' in the
hyperbolic case are longer. Also the peak is logarithmically
smaller than in the elliptic case (a somewhat weaker statement is
proved in \cite{TZ2}).

Let us state the result precisely and briefly sketch the argument.
It makes an interesting comparison to the situation discussed
later on of possible scarring in the Anosov case.

 We denote by $\pi: S^* M \to M$ the standard projection and let
 $\pi(\gamma)$ be the image of $\gamma$ in $M$.
We denote by $T_{\epsilon}(\pi(\Lambda))$ the tube of radius  $
\epsilon$ around  $\pi(\Lambda)$. For $0 < \delta < 1/2$, we
introduce a cutoff $\chi_{1}^{\delta} (x;\hbar) \in
C^{\infty}_{0}(M)$ with $0 \leq \chi_{1}^{\delta} \leq 1,$
satisfying
\begin{itemize} \label{tube3}
\item (i) supp $\chi_{1}^{\delta} \subset T_{\hbar^{\delta} }
(\pi(\gamma))$ \item (ii)
 $\chi_{1}^{\delta} = 1$ on $ T_{3/4 \hbar^{\delta}           }  (\pi(\gamma))$.
\end{itemize}

\begin{theo} \label{SSM} Let $\gamma$ be a hyperbolic closed orbit in
$(M, g)$ with quantum integrable $\Delta_g$, and let
$\{\phi_{\mu}\}$ be an $L^2$ normalized $\gamma$-sequence of joint
eigenfunctions
  Then for any   $ 0 \leq \delta <1/2$, $
\lim_{\hbar \to 0} (Op_{\hbar}(\chi_{1}^{\delta}) \phi_{\mu},
\phi_{\mu}) \geq (1 - 2 \delta).$
\end{theo}

\subsubsection{Outline of proof}

For simplicity we assume $\dim M = 2$.
 Let  $\chi_{2}^{\delta}(x,\xi;\hbar)
\in C^{\infty}_{0}(T^{*}M ; [0,1])$ be a
 second cutoff supported in a radius $\hbar^{\delta}$ tube,  $\Omega (\hbar)$, around $\gamma$ with
$ \Omega(\hbar) \subset supp \chi_{1}^{\delta}$
 and such that
$ \chi_{1}^{\delta}  = 1 \,\,\mbox{on} \,\, supp
\chi_{2}^{\delta}.$ Thus, $ \chi_{1}^{\delta}(x,\xi) \geq
\chi_{2}^{\delta}(x,\xi),$  for any $(x,\xi) \in T^{*}M$.
 By the Garding inequality, there exists a constant
$C_{1}>0$ such that:
\begin{equation} \label{WG}
(Op_{\hbar}(\chi_{1}^{\delta}) \phi_{\mu},\phi_{\mu)}) \geq
(Op_{\hbar}(\chi_{2}^{\delta})\phi_{\mu}, \phi_{\mu}) - C_{1}
\hbar^{1-2\delta}.
\end{equation}

We now conjugate the right side to the model setting of $S^1
\times \R^1$, i.e. the normal bundle $N_{\gamma}$ to $\gamma$. The
conjugation is done by $\hbar$ Fourier integral operators and is
known as conjugation to quantum Birkhoff normal form. In the model
space, the conjugate of $\Delta_g$ is  a function of $D_s =
\frac{\partial}{i\partial s}$ along $S^1$ and the dilation
operator $\hat{I}^{h} := \hbar (D_{y} y + y D_{y}) $ along $\R$.
By Egorov's theorem
\begin{equation} \label{bound}
(Op_{\hbar}(\chi_{2}^{\delta}) \phi_{\mu}, \phi_{\mu}) =
|c(\hbar)|^2  (Op_{\hbar}(\chi_{2}^{\delta} \circ \kappa) u_{\mu},
u_{\mu}) - C_{3} \hbar^{1-2\delta}
\end{equation}
\noindent where $ u_{\mu} (y,s;\hbar)$ is the model joint
eigenfunction of $D_s, \hat{I}^{h}$, and $c(\hbar)$ is a
normalizing constant.  This reduces our problem to estimating the
explicit matrix elements $(Op_{\hbar}(\chi_{2}^{\delta} \circ
\kappa) u_{\mu}, u_{\mu})$ of the  special eigenfunctions in the
model setting. The operator $\hat{I}^h$ has a continuous spectrum
with generalized eigenfunctions $y^{-1/2 + i\lambda/\hbar}$. The
eigenfunctions on the `singular' level $\gamma$ correspond to
$\lambda \sim E \hbar$.  A caluclation shows that the mass in the
model setting is given by
\begin{equation} \label{modelmass1}
M_{h} = \frac{ 1} { \log \hbar} \, \left( \int_{0}^{\infty}
\chi(\hbar \xi/\hbar^{\delta}) \left| \int_{0}^{\infty} e^{-ix}
x^{-1/2 + i \lambda/\hbar} \chi(x/\hbar^{\delta}\xi) dx
\right|^{2} \frac{d\xi}{\xi} \right).
\end{equation} Classical analysis shows that the right side tends
to $1 - 2 \delta$ as $\hbar \to 0$ if $\lambda \sim E \hbar$.

\subsection{Comparison to Anosov case}

This large  mass profile may be  a special feature of integrable
systems, reflecting the fact that the stable and unstable
manifolds of the hyperbolic closed orbit coincide. A heuristic
picture of the mass concentration in this case is as follows:
Since the eigenfunction is a stationary state, its mass must be
asymptotically invariant under the geodesic flow. Since the flow
compresses things exponentially in the stable direction and
expands things in the unstable direction, the mass can only
concentrate on the fixed closed geodesic and on the unstable
manifold $W^u$. But $W^u$ returns to $\gamma$  as the stable
manifold $W^u$ (like a figure $8$). Hence, the only invariant
measure is the one supported on the closed geodesic and the the
mass can only concentrate there. Although it does not seem to have
been proved in detail yet, it is very plausible that the mass
concentration in the integrable case provides an upper bound for
any $(M, g)$, i.e. it has `extremal' mass concentration.

   In the Anosov case, the stable and unstable manifolds of $\gamma$ are transverse,
   so the dynamical picture is completely different.
    First,    there is no obvious
   mechanism as in the integrable case forcing  mass of any sequence of eigenfunctions to
   concentrate on the Lagrangian manifold formed by $\gamma$ and $W^u$
   in the Anosov case. If mass did concentrate around $\gamma$, it
   would
   still be  forced to
   concentrate on $\gamma$ and on $W^u$, but $W^u$ becomes dense
   in $S^*M$. Hence  some of the mass must spread out uniformly
   over $S^*M$ and is lost from a neighborhood of $\gamma$. This makes it plausible that one does not get mass
   concentration for eigenfunctions around hyperbolic closed orbits of Anosov
   systems.

     Yet,    in the ``cat map" analogue, there do exist  scarring
eigenfunctions  \cite{FNB,FN} for a special sparse sequence of
Planck constants. The multiplicities of the eigenvalues of the cat
map for this sequence saturate the bound $\hbar^{-1}/|\log
\hbar|$, and therefore one can build up eigenfunctions with very
unexpected properties. There is a surprising  quantum mechanism
forcing producing  concentration of special modes at hyperbolic
fixed points which was discovered by Faure-Nonnenmacher-de
Bi\`evre. The eigenfunction amplitude spreads out along
$\hbar^{-1/2}$ segments of  $W^u$. These segments are
$\hbar^{1/2}$ dense and so there is interference between the
amplitudes on close pieces of $W^u$. The interference is
constructive along $W^s$ and the mass builds up there and then as
in the integrable case gets recycled back to the hyperbolic orbit.

It is not known whether this phenomenon occurs in the Riemannian
setting. It is presumably related to the existence of sparse
subsequences of eigenvalues with the same large multiplicities.

\section{\label{BQE} Boundary quantum ergodicity and quantum ergodic
restriction}

In this section, we briefly discuss quantum ergodic restriction
theorems. The general question is whether restrictions of quantum
ergodic eigenfunctions to hypersurfaces (or microlocally, to
cross-sections of the geodesic flow) are quantum ergodic on the
hypersurfaces. We only consider here  the case where the
hypersurface is the boundary of a domain with boundary, which was
studied in \cite{GL,HZ,Bu}. More general quantum ergodic
restriction theorems are  given in \cite{TZ3}. The boundary
results play an important role in the recent scarring results of
A. Hassell for eigenfunctions on the stadium.

We thus consider the  boundary values $u_j^{b}$  of interior
eigenfunctions
$$\left\{ \begin{array}{l} \Delta \; u_j = \lambda_j^2\;  u_j\;\; \mbox{in} \; \Omega, \;\; \langle u_j, u_k
 \rangle_{L^2(
\Omega)} = \delta_{jk}, \\ \\B u_j |_{Y} = 0,\;\;\;\; Y = \partial
\Omega
\end{array} \right.
$$ of the Euclidean  Laplacian $\Delta$ on a compact piecewise smooth domain $\Omega \subset
\R^n$ and with  classically ergodic billiard map $\beta:  B^* Y
\to B^*Y$, where $Y = \partial \Omega$. Here $A_h$ is a zeroth
order semiclassical pseudodifferential operator on $Y$. The
relevant notion of boundary values  (i.e. Cauchy data) $u_j^{b}$
depends on the boundary condition. We only consider the boundary
conditions
\begin{equation} B u = \left\{ \begin{array}{ll}  u|_Y, & \mbox{Dirichlet} \\ & \\
  \partial_{\nu} u |Y, & \mbox{Neumann} \end{array} \right.
\label{K-bc}\end{equation}
 Let $\Delta_B$
 denote the positive Laplacian on $\Omega$ with boundary
conditions $B u = 0$. Then $\Delta_B$ has discrete spectrum $0 <
\lambda_1 < \lambda_2 \leq \dots \to \infty$, where we repeat each
eigenvalue according to its multiplicity, and for each $\lambda_j$
we may choose an $L^2$ normalized eigenfunction $u_j$.

 The
 algebra of observables in the boundary setting is the algebra
$\Psi_h^0(Y)$ of zeroth order semiclassical pseudodifferential
operators on $Y,$ depending on the parameter $h \in [0, h_0]$. We
denote the symbol of  $A = A_h \in \Psi_h^0(Y)$ by  $a =
a(y,\eta,h)$. Thus $a(y, \eta) = a(y, \eta, 0)$ is a smooth
function on $T^*Y$.

To each boundary condition $B$ corresponds
\begin{itemize}

\item  A specific notion of  boundary value $u_j^{b}$ of the
eigenfunctions $u_j$. We denote the $L^2$-normalized boundary
values by $\hat u_j^{b} = u_j^b/||u_j^b||.$

\item A specific measure $d\mu_B$ on $B^* (Y). $

\item A specific state $\omega_B$ on the space $\Psi_h^0(Y)$ of
semiclassical pseuodifferential operators of order zero defined by
\begin{equation}\begin{aligned}
\omega_B(A) &= \frac{4}{\vol(S^{n-1})\vol(\Omega)} \int_{B^*Y}
a(y,\eta)  d\mu_B.
\end{aligned}\label{omegaB}\end{equation}
\end{itemize}

 Here is a table
of the relevant boundary value notions.  In the table, $d \sigma$
is the natural symplectic volume measure on $B^*Y$. We also define
the function $\gamma(q)$ on $B^*Y$ by
\begin{equation}
\gamma(q) = \sqrt{1 - |\eta|^2}, \quad q = (y,\eta).
\label{a-defn}\end{equation}

\bigskip
\noindent \hskip 50pt {\large \begin{tabular}{|c|c|c|c|c|} \hline
\multicolumn{4}{|c|}{\bf Boundary Values} \\ \hline
 B & $Bu$  &   $u^{b}$ & $d\mu_B$   \\
\hline
 Dirichlet &  $u|_{Y}$ &  $\lambda^{-1}\; \partial_{\nu} u |_{Y}$ & $\gamma(q) d\sigma $ \\
\hline
 Neumann &   $\partial_{\nu} u |_{Y}$  &  $u|_{Y}$  & $\gamma(q)^{-1}
 d\sigma$\\
\hline
 \end{tabular}}
\bigskip

The limit states are  determined by  dictated by the local Weyl
law
 for the boundary condition $B$.

\begin{lem} \label{LWLb}
 Let $A_h$ be either the identity operator on $Y$ or a zeroth order semiclassical
  operator on $Y$ with kernel supported away from the singular set.   Then for any of the above boundary conditions $B$,
 we have:
 \begin{equation}\begin{gathered}
\lim_{\lambda \to \infty} \frac{1}{N(\lambda)} \sum_{\lambda_j
\leq \lambda}  \langle A_{h_j} u_j^b, u_j^b \rangle =
\omega_B(A) , \quad  B = \text{ Neumann}, \\
\lim_{\lambda \to \infty} \frac{1}{N(\lambda)} \sum_{\lambda_j
\leq \lambda}  \langle A_{h_j} u_j^b, u_j^b \rangle =
\omega_{B}(A), \quad B = \text{ Dirichlet}.\end{gathered}
\label{Weyl-limit}\end{equation}
\end{lem}

The boundary quantum ergodic theorem is:

\begin{theo}\label{mainb} \cite{HZ,GL,Bu}
Let  $\Omega \subset \RR^n$ be a bounded piecewise smooth manifold
 with  ergodic billiard map. Let
$\{u_j^{b}\}$ be the boundary values of the eigenfunctions
$\{u_j\}$ of $\Delta_B$ on $L^2(\Omega)$ in the sense of the table
above.
 Let $A_h$ be a semiclassical operator of order zero on $Y$. Then there is a subset $S$ of the positive integers, of density one,
such that
\begin{equation}\begin{gathered}
\lim_{j \to \infty, j \in S} \langle A_{h_j} u_j^b,
u_j^b \rangle = \omega_B(A), \quad B = \text{ Neumann}, \\
\lim_{j \to \infty, j \in S}  \langle A_{h_j} u_j^b, u_j^b \rangle
= \omega_{B}(A), \quad B = \text{ Dirichlet},\end{gathered}
\label{main-eqn}\end{equation} where $h_j = \lambda_j^{-1}$ and
$\omega_B$ is as in \eqref{omegaB}.
\end{theo}

In the case $A = I$ and  for the Neumann boundary condition, we
have
$$
\lim_{j \to \infty, j \in S} \| u_j^b \|_{L^2(Y)}^2 = \frac{2
\vol(Y)}{\vol(\Omega)},
$$
while for the Dirichlet boundary condition,
$$
\lim_{j \to \infty, j \in S} \| u_j^b \|_{L^2(Y)}^2 = \frac{2
\vol(Y)}{n\vol(\Omega)}.
$$

\subsection{Sketch of proof}

The fact that the quantum limit state $\omega_B$ and the
corresponding measure $d\mu_B$ are not the natural symplectic
volume measure $d \sigma$ on $B^*Y$ is due  to the fact that the
quantum dynamics is defined by an endomorphism rather than an
automorphism of the observable algebra.  In the Neumann case, the
dynamics are generated by the operator $F_h$ on $Y$ with kernel
\begin{equation} \begin{gathered}
F_h(y,y') = 2\frac{\pa}{\pa \nu_y} G_0(y,y',h^{-1}), \quad y \neq
y'
\in Y, \text{ where }\\
G_0(y,y',\lambda) =  \frac{i}{4} \lambda^{d-2} (2 \pi \lambda
|z-z'|)^{-(d-2)/2} \Ha_{d/2-1}(\lambda | z-z' |)
\end{gathered}\label{Neumann-F}\end{equation}
is the free outgoing Green function on $\RR^d$. These are (almost)
semi-classical Fourier integral operators whose phase functions,
the boundary distance function $d_b(y, y') = |y - y'$ ($y, y' \in
Y$) generates the billiard map.  It is well known that this
operator leaves the boundary values of Neumann eigenfunctions
$u_j^{b}$ invariant:
\begin{equation} \label{FONE}
F_{h_j} u_j^{b} = u_j^{b}, \quad j = 1,2, \dots
\end{equation}
It follows that the states
\begin{equation} \rho_j(A) :=  \langle A_{h_j} u_j^b, u_j^b \rangle
\end{equation}
are invariant for $F_{h_j}$. The family $\{F_{h}\}$ defines a
semiclassical Fourier integral operator associated to the billiard
map $\beta$ (for convex $\Omega$), plus some negligible terms. The
quantum dynamics on $\Psi^0_h (Y)$ is thus generated by the
conjugation
\begin{equation} \label{QDF} \alpha_{h_j} (A_{h_j}) =
F_{h_j}^*\;A_{h_j} \; F_{h_j}
\end{equation}
This is analogous to the interior dynamics generated by the wave
group $U^t_B$ with the boundary conditions,  but  unlike $U^t_B$,
$F_h$ is not unitary or even normal. Indeed, the zeroth order part
of $F_h^* F_h$ is a semiclassical pseudodifferential operator with
a non-constant symbol.

The Egorov type result for the operator $F_h$ is as follows: Let
$\beta$ denote the billiard map on $B^* Y^o$ and let $A_h =
\Op(a_h)$ be a zeroth order operator whose symbol $a(y,\eta,0)$.
Let $\gamma$ be given by \eqref{a-defn}. Then
$$
F_h^* A_h  F_h = \tilde A_h + S_h,
$$
where $\tilde A_h$ is a zeroth order pseudodifferential operator
and $\| S_h \|_{L^2 \to L^2} \leq C h$. The symbol of $\tilde A_h$
is
\begin{equation}
 \tilde a = \begin{cases}
\gamma(q) \gamma(\beta(q))^{-1} a(\beta(q)) , \quad q \in B^*Y \\
0 , \phantom{\gamma(q)^{-1} \gamma(\beta(q)) b(\beta(q)) , \ } q
\notin B^*Y. \end{cases} \label{Egorov-formula}\end{equation} This
is a rigorous version of the statement that $F_h$ quantizes the
billiard ball map.  This Egorov theorem is relevant to the Neumann
boundary problem. In the Dirichlet case, the relevant operator is
$F_h^*$.

The unusual transformation law of the symbol reflects the fact
that \eqref{QDF} is not an automorphism. In the spectral theory of
dynamical systems, one studies the dynamics of the billiard map
$\beta$ on $B^* Y$ through the associated `Koopman' operator
$$\mathcal{U}: L^2(B^*Y, d \sigma)
\to L^2(B^* Y, d\sigma),\;\; \mathcal{U} f (\zeta) =  f(\beta
(\zeta)).$$
 Here,
$d\sigma$ denotes the usual  $\beta$-invariant symplectic volume
measure on $B^* Y$. From the symplectic  invariance it follows
that $\mathcal{U}$ is a unitary operator.  When $\beta$ is
ergodic, the unique invariant $L^2$-normalized eigenfunction is a
constant $c$.

However, Egorov's theorem (\ref{Egorov-formula}) in the boundary
reduction involves the positive function $\gamma \in
C^{\infty}(B^*Y)$, and the relevant  Koopman operator is
$$T f (\zeta) = \frac{\gamma (\zeta)}{(\gamma(\beta(\zeta)))} f(\beta(\zeta)).$$ Then $T$ is not  unitary
on $L^2(B^*Y, d \sigma)$. Instead one has:

\begin{itemize}
\item (i) The unique positive $T$-invariant $L^1$ function is
given by  $\gamma.$ \item (ii)  $T$ is unitary relative to the
inner product $\langle \langle, \rangle \rangle$ on $B^* Y$
defined by the measure $d\nu = \gamma^{-2} d \sigma$. \item (iii)
When $\beta$ is ergodic, the   orthogonal projection $P$ onto the
invariant $L^2$-eigenvectors has the form $$ P(f) = \frac{\langle
\langle f, \gamma \rangle \rangle}{ \langle \langle \gamma, \gamma
\rangle \rangle} \gamma = \frac1{\vol(B^*Y)} [\int_{B^* Y} f
\gamma^{-1} d \sigma] \; \gamma = c \omega_{\Neu}(f) \; \gamma
$$
where $c$ is as in \eqref{c}.
\end{itemize}

The proof of Theorem \ref{mainb} then  runs along similar lines to
that of Theorem \ref{QE}  but adjusted to the fact that that $T$
is not unitary.

\section{\label{HAS} Hassell's scarring result for stadia}

This section is an exposition of Hassell's scarring result for the
Bunimovich stadium. We follow \cite{Has} and \cite{Z7}.

 A stadium is a domain $X = R \cup W  \subset
\R^2$ which is formed by a rectangle $R = [- \alpha, \alpha]_x
\times [- \beta, \beta]_y$ and where $W = W_{- \beta} \cup
W_{\beta}$ are half-discs of radius $\beta$ attached at either
end. We fix the height $\beta = \pi/2$ and let $\alpha =  t \beta$
with $t \in [1, 2]$. The resulting stadium is denoted $X_t$.

It has long been suspected that there exist exceptional sequences
of eigenfunctions of $X$ which have a singular concentration on
the set of ``bouncing ball" orbits of $R$. These are the vertical
orbits in the central rectangle that repeatedly bounce
orthogonally against the flat part of the boundary. The unit
tangent vectors to the orbits define an invariant Lagrangian
submanifold with boundary $\Lambda \subset S^* X$. It is easy to
construct approximate eigenfunctions which concentrate
microlocally on this  Lagrangian submanifold. Namely, let
$\chi(x)$ be a smooth cutoff supported in the central rectangle
and form $v_n = \chi(x) \sin n y$. Then for any
pseudo-differential operator $A$ properly supported in $X$,
$$\langle A v_n, v_n \rangle \to \int_{\Lambda} \sigma_A \chi d \nu$$
where $d\nu$ is the unique normalized invariant measure on
$\Lambda$.

Numerical studies suggested that there also existed genuine
eigenfunctions with the same limit. Recently, A. Hassell has
proved this to be correct for almost all stadia.

\begin{theo} \label{HH} The Laplacian on $X_t$ is not QUE for
almost every $t \in [1, 2].$ \end{theo}

We now sketch the proof and develop related ideas on quantum
unique ergodicty. The main idea is that the existence of the
scarring bouncing ball quasi-modes implies that either

\begin{itemize}

\item There exist actual modes with a similar scarring property,
or

\item The spectrum has exceptional clustering around the bouncing
ball quasi-eigenvalues $n^2$.

\end{itemize}

Hassell then proves that the second alternative cannot occur for
most stadia. We now explain the ideas in more detail.

We first recall that a quasi-mode $\{\psi_k\}$ is a sequence of
$L^2$-normalized functions which solve
$$||(\Delta - \mu_k^2) \psi_k||_{L^2} = O(1),  $$
for a sequence of quasi-eigenvalues $\mu_k^2$. By the spectral
theorem it follows that there must exist true eigenvalues in the
interval $[\mu_k^2 - K, \mu_k^2 + K]$ for some $K> 0$. Moreover,
if $\tilde{E}_{k, K}$ denotes the spectral projection for $\Delta$
corresponding to this interval, then
$$ ||\tilde{E}_{k, K} \psi_k - \psi_k||_{L^2} = O(K^{-1}). $$
To maintain consistency with (\ref{ELAMBDA}), i.e. with our use of
frequencies $\mu_k$ rather than energies $\mu_k^2$, we re-phrase
this in terms of the projection $E_{k, K}$ for $\sqrt{\Delta}$ in
the interval $[\sqrt{\mu_k^2 - K}, \sqrt{\mu_k^2 + K}]$. For fixed
$K$, this latter interval has width $\frac{K}{\mu_k}$.

 Given a quasimode
$\{\psi_k\}$, the  question arises of how many true eigenfunctions
it takes to build the quasi-mode up to a  small error.

\begin{defin}  \label{ESS}  We say that a quasimode $\{\psi_k \}$ of order $0$ with $||\psi_k||_{L^2} = 1$ has
 $n(k)$ essential frequencies if
\begin{equation} \label{PSIK} \psi_k = \sum_{j = 1}^{n(k) } c_{kj} \phi_j + \eta_k,\;\;\;
||\eta_k||_{L^2} = o(1). \end{equation}
\end{defin}
To be a quasi-mode of order zero, the frequencies $\lambda_j$ of
the $\phi_j$ must come from an interval  $[\mu_k -
\frac{K}{\mu_k}, \mu_k + \frac{K}{\mu_k} ]$. Hence the number of
essential frequencies is bounded above by the number $n(k) \leq
N(k, \frac{K}{k})$  of eigenvalues in the interval. Weyl's law for
$\sqrt{\Delta}$ allows considerable clustering and only gives
$N(k,  \frac{K}{k}) = o(k)$ in the case where periodic orbits have
measure zero.  For instance, the quasi-eigenvalue might be a true
eigenvalue with multiplicity saturating the Weyl bound.  But  a
typical interval has a uniformly bounded number of
$\Delta$-eigenvalues in dimension $2$ or equivalently a frequency
interval of with $O(\frac{1}{\mu_k})$ has a uniformly bounded
number of frequencies.  The dichotomy above reflects the dichotomy
as to whether exceptional clustering of eigenvalues occurs around
the quasi-eigenvalues $n^2$ of $\Delta$ or whether there is a
uniform bound on $N(k, \delta)$.

\begin{prop} \label{BOUNDED} If there exists a quasi-mode $\{\psi_k\}$ of order $0$
for $\Delta$ with the properties: \begin{itemize}

\item  (i)  $n(k)  \leq C,\; \forall \; k$;

\item (ii) $\langle A \psi_k, \psi_k \rangle \to \int_{S^*M}
\sigma_A d \mu$ where $d \mu \not= d\mu_L$.

\end{itemize}

Then $\Delta$ is not QUE.
\end{prop}

The proof is based on the following lemma pertaining to near
off-diagonal Wigner distributions. It gives an ``everywhere''
version of the off-diagonal part of Theorem \ref{QE} (2).

   \begin{lem}  \label{PONEa} Suppose that $g^t$ is ergodic and $\Delta$ is QUE. Suppose that
 $\{(\lambda_{i_r}, \lambda_{j_r}), \; i_r \not= j_r\}$ is a sequence of pairs
 of eigenvalues of $\sqrt{\Delta}$ such that
 $\lambda_{i_r} - \lambda_{j_r} \to 0$ as $r \to \infty$.  Then $d W_{i_r, j_r} \to 0$.
 \end{lem}

 \begin{proof}  Let $\{\lambda_i, \lambda_j\}$ be any
 sequence of pairs with the gap $\lambda_i - \lambda_j \to 0$.
 Then by Egorov's theorem,
  any weak* limit $d \nu $ of the sequence $\{d W_{i,j} \}$ is a measure invariant under the geodesic
 flow. The weak limit is defined by the property that
 \begin{equation} \label{QUE}   \langle A^*A \phi_i, \phi_j \rangle  \to \int_{S^*M}
 |\sigma_A|^2 d \nu. \end{equation}
 If the eigenfunctions are real, then $d \nu$ is a real (signed)
 measure.

 We now observe that  any such weak* limit must be a
 constant multiple of Liouville measure $d\mu_L$. Indeed, we first
 have:
 \begin{equation}| \langle A^*A \phi_i, \phi_j \rangle | \leq | \langle A^*A \phi_i, \phi_i \rangle
 |^{1/2} \; | \langle A^*A \phi_j, \phi_j \rangle |^{1/2}.
 \end{equation}
 Taking the limit along the sequence of pairs, we obtain
 \begin{equation} |\int_{S^*M}
 |\sigma_A|^2 d \nu| \leq \int_{S^*M} |\sigma_A|^2 d \mu_L.
 \end{equation}
 It follows that $d \nu << d \mu_L$ (absolutely continuous). But $d\mu_L$
 is an ergodic measure, so if $d \nu = f d \mu_L$ is an invariant measure  with $f \in L^1 (d \mu_L)$, then $f $ is constant.
 Thus,
 \begin{equation} \label{C} d \nu = C d \mu_L, \;\;\; \mbox{for some constant}\; C.  \end{equation}

 We now observe  that $C = 0$ if $\phi_i \bot \phi_j$ (i.e. if $i \not= j)$.  This follows if we substitute $A =
 I$ in (\ref{QUE}), use orthogonality and (\ref{C}).

\end{proof}

We now complete the proof of the Proposition by arguing by
contradiction. The frequencies must come from a shrinking
frequency interval, so the hypothesis of the Proposition is
satisfied. If $\Delta$ were QUE, we would have (in the notation of
(\ref{PSIK}):
$$\begin{array}{ll}\langle A \psi_k, \psi_k \rangle & = \sum_{i, j = 1}^{n(k)} c_{kj} c_{k i}\langle A
\phi_i, \phi_j \rangle + o(1) \\ & \\&  = \sum_{j = 1}^{n(k)}
c_{kj}^2 \langle A \phi_j, \phi_j \rangle + \sum_{i\not= j =
1}^{n(k)} c_{kj} c_{k i}\langle A \phi_i, \phi_j
\rangle + o(1)\\ & \\
& = \int_{S^*M} \sigma_A d\mu_L + o(1),
\end{array} $$
by Proposition \ref{PONEa}. This contradicts (ii). In the last
line we used $\sum_{j = 1}^{n(k)} |c_{kj}|^2  = 1 + o(1)$, since
$||\psi_k||_{L^2} = 1$.

\medskip

QED

\medskip

\subsection{Proof of Hassell's scarring result}

We apply and develop this reasoning in the case of the stadium.
The quasi-eigenvalues of the Bunimovich stadium corresponding to
bouncing ball quasi-modes are $n^2$ independently of the diameter
$t$ of the inner rectangle.

By the above, it suffices to show that  that there exists a
sequence $n_j \to \infty$ and a constant $M$ (independent of $j$)
so that there exist $\leq M$ eigenvalues of $\Delta$ in $[n_j^2 -
K, n_j^2 + K]$. An somewhat different argument is given in
\cite{Has} in this case: For each $n_j$ there exists a normalized
eigenfunction $u_{k_j}$ so that $\langle u_{k_j}, v_{k_j} \rangle
\geq \sqrt{\frac{3}{4} M}. $ It suffices to choose the
eigenfunction with eigenvalue in the interval with the largest
component in the direction of $v_{k_j}$. There exists one since
$$||\tilde{E}_{[n^2 - K, n^2 + K } v_n || \geq \frac{3}{4}. $$
The sequence $\{u_{n_k}\}$ cannot be Liouville distributed.
Indeed, for any $\epsilon > 0$, let $A$ be a self-adjoint
semi-classical pseudo-differential operator properly supported in
the rectangle so that $\sigma_A \leq 1$ and so that $||(Id -
A)v_n|| \leq \epsilon$ for large enough $n$. Then
$$\begin{array}{lll} \langle A^2 u_{k_j}, u_{k_j} \rangle = ||A u_{k_j}||^2  & \geq &
\left| \langle A u_{k_j}, v_{k_j} \rangle \right|^2  \\ && \\
& = & \left| \langle  u_{k_j}, A v_{k_j} \rangle \right|^2 \geq
\left(|\langle u_{k_j}, v_{k_j} \rangle | - \epsilon \right)^2
\geq \left( \sqrt{\frac{3}{4} M} - \epsilon \right)^2.
\end{array} $$
Choose a sequence of operators $A$ such that $||(Id - A) v_n|| \to
0$ and so that the support of $\sigma_A$ shrinks to the set of
bouncing ball covectors. Then the mass of any quantum limit of
$\{u_{n_k}\}$ must have mass $\geq \frac{3}{4}M$ on $\Lambda$.

Thus, the main point is to eliminate the possibility of
exceptional clustering of eigenvalues around the
quasi-eigenvalues. In fact, no reason is known why no exceptional
clustering should occur. Hassell's idea is that it can however
only occur for a measure zero set of diameters of the inner
rectangle. The proof is based on Hadamard's variational formula
for the variation of Dirichlet or Neumann  eigenvalues under a
variation of a domain. In the case at hand, the stadium is varied
by horizontally (but not vertically) expanding the inner
rectangle. In the simplest case of Dirichlet boundary conditions,
the eigenvalues are forced to decrease as the rectangle is
expanded. The QUE hypothesis forces them to decrease at a uniform
rate. But then they can only rarely cluster at the fixed
quasi-eigenvalues $n^2$. If this ever happened, the cluster would
move left of $n^2$ and there would not be time for a new cluster
to arrive.

Here is a more detailed sketch. Under the variation of $X_t$ with
infinitesimal variation vector field $\rho_t$,  Hadamard's
variational
 formula gives,
$$\frac{d E_j(t)}{dt} = \int_{\partial X_t} \rho_t(s) (\partial_n
u_j(t)(s))^2 ds. $$ Then
$$E_j^{-1} \frac{d}{dt} E_j(t) = - \int_{\partial X_t} \rho_t(s)
u^b_j(s)^2 ds. $$

Let $A(t) $ be the area of $S_t$. By Weyl's law, $E_j(t) \sim c
\frac{j}{A(t)}. $ Since the area of $X_t$ grows linearly, we have
on average  $\dot{E}_j \sim - C \frac{E_j}{A(t)}. $ Theorem
\ref{mainb} gives the asymptotics individually for almost all
eigenvalues.  Let
$$f_j(t) = \int_{\partial X_t} \rho_t(s) |u^b_j(t; s)|^2 ds. $$
Then $\dot{E}_j = - E_j f_j$. Then Theorem \ref{mainb}  implies
that $|u^b_j|^2 \to \frac{1}{A(t)}$ weakly on the boundary along a
subsequence of density one. QUE is the hypothesis that this occurs
for the entire sequence, i.e. $$f_j(t) \to \frac{k}{A(t)} > 0,
\;\;\; k:= \int_{\partial S_t} \rho_t(s) ds.$$  Hence,
$$\frac{\dot{E}_j}{E}  = - k A(t) (1 + o(1)), \;\; j \to \infty.
$$
Hence there is a lower bound to the velocity with which
eigenvalues decrease as $A(t)$ increases. Eigenvalues can
therefore not concentrate in the fixed quasi-mode intervals $[n^2
- K, n^2 + K]$ for all $t$. But then there are only a bounded
number of eigenvalues in this interval; so  Proposition
\ref{BOUNDED} implies QUE for the other $X_t$. A more detailed
analysis shows that QUE holds for almost all $t$.

\section{\label{FIOME} Matrix elements of Fourier integral operators}

Our main focus in this survey  goal is on the limits of diagonal
matrix elements (\ref{RHOJ}) of  pseudodifferential
 operators $A$ or order zero  relative to an orthonormal basis $\{\phi_j\}$ of
 eigenfunctions. Difficult as it is, the limiting behavior of such
 matrix elements is one of the more accessible properties of
 eigenfunctions. To obtain more information, it would be useful to
 expand the class of operators $A$ for which one can study matrix
 elements. In \S \ref{HYP}, we were essentially expanding the
 class from classical polyhomogeneous pseudo-differential
 operators to those in the ``small-scale" calculus, i.e. supported
 in $h^{\delta}$ tubes around closed geodesics. Another way to
 expand the class of $A$ is to consider matrix elements of Fourier integral
 operators. Motivation to consider matrix elements  of Fourier
 integral operators comes partly from the fact that Hecke
 operators are Fourier integral operators, and this section is a preparation
 for the next section on Hecke operators \S \ref{HECKE}. Other examples arise in
 quantum ergodic restriction theorems. Details on the claims to
 follow will be given in a forthcoming paper.

  By an FIO (Fourier
 integral operator) we mean an operator whose
 Schwartz kernel of $F$ may be locally represented as a finite
sum of oscillatory integrals,
$$K_F(x,y) \sim \int_{\R^N} e^{i \phi(x, y, \theta)} a(x, y,
\theta) d\theta$$ for some homogeneous phase $\phi$ and amplitude
$a$.  The only example we have seen so far in this survey is the
wave group $U^t$ and of course pseudo-differential operators.

The phase is said to generate  the  canonical relation
$$C = \{(x, \phi'_x, y, - \phi'_y):  \phi'_{\xi}(x, y, \theta) =
0\} \subset T^*M \times T^*M. $$ The oscillatory integral also
determines the principal $\sigma_F$ of $F$, a $1/2$ density along
$C$. The data $(C, \sigma_F)$ determines $F$ up to a compact
operator. In the case of the wave group, $C$ is the graph of $g^t$
and the symbol is the canonical volume half-density on the graph.
In the case of $\psi$DO's, the canonical relation is the diagonal
(i.e. the graph of the identity map) and the symbol is the one
defined above. We denote by $I^0(M, \times M, C)$ the class of
Fourier integral operators of order zero and canonical relation
$C$.

Hecke operators are also FIO's of a simple kind. As discussed in
the next section, in that case $C$ is a local canonical graph,
i.e. the projections
$$\pi_X: C \to T^* M, \;\;\; \pi_Y: C \to T^*M $$
are branched covering maps. Equivalently,  $C$ is the graph of a
finitely many to one  correspondence $\chi: T^*M \to T^* M$.
Simpler examples of the same type include (i) $F = T_g$ is
translation by an isometry of a Riemannian manifold $(M, g)$
possessing an isometry, or
 (ii)  $T f(x) = \sum_{j = 1}^k f(g_j x) +
f(g_j^{-1} x) $ on $S^n$ corresponding to a finite set $\{g_1,
\dots, g_k\}$ of isometries. The latter has been studied by
Lubotzky-Phillips-Sarnak.

\subsection{Matrix elements as linear functionals}

 The matrix elements
\begin{equation} \rho_j(F): = \langle F \phi_j, \phi_j \rangle \end{equation}  define  continuous linear
functionals \begin{equation} \rho_j:  I^0(M \times M, C) \to \C
\end{equation} with respect to the operator norm. It is simple to see that any
limit $\rho_{\infty}$ of the sequence of functionals $\rho_j$ are
functionals only of the principal symbol data $(C, \sigma_F)$
which is bounded by the supremum of $\sigma_F$. Thus, as in the
pseudo-differential case, $\rho_{\infty}$ is a measure on the
space $\Omega^{1/2}_C$ of continuous half-densities on $C$.

The limiting behavior of $\rho_j(F)$ depends greatly on whether
$[F, \sqrt{\Delta}] = 0$ (or is of negative order) or whether it
has the same order as $F$. In effect, this is the issue of whether
the canonical relation $C$ underlying $F$ is invariant under the
geodesic flow or not. For instance, if $C$ is the graph of a
canonical transformation $\chi$ and if $\chi$ moves the energy
surface $S^*_g M = \{(x, \xi): |\xi|_g = 1\}$ to a new surface
$\chi(S^*_g M)$  disjoint from $S^*_g M$ then $\langle F \phi_j,
\phi_j \rangle \to 0$. On the other hand, $\langle A e^{i t
\sqrt{\Delta}} \phi_j, \phi_j \rangle = e^{it \lambda_j} \langle A
\phi_j, \phi_j \rangle$ so one gets the same quantum limits for $F
= A e^{i t \sqrt{\Delta}}$ as for $A$ but must pick sparser
subsequences to get the limit.

Hecke operators have the property that $[F, \Delta] = 0.$ This
implies that the canonical relation $C$ is invariant under $g^t$.
In general this is the only condition we need to study weak
limits. The eigenfunction linear functionals $\rho_k$   are
invariant in the the sense that $\rho_k(U^{-t}F U^t) = \rho_k(F)$.

\begin{prop}\label{ONE} Let $C \subset T^*M \times T^*M$ be a local canonical graph, equipped with the pull-back of the
symplectic volume measure on $T^* M$. Assume $C$ is invariant
under $g^t$.  Let $F \in I^0(M \times M, C)$ and identify its
symbol with a scalar function relative to the graph half-density.
Let $\rho_{\infty}$ be a weak* limit of the functionals $\rho_j$
on $I^0(M \times N, C)$. Then there exists a complex measure $\nu$
on $SC = C \cap S^* M \times S^* M$ of mass $\leq 1$ such that
$$\rho_{\infty}(F) =  \int_{SC} \sigma_F d\nu, $$
which satisfies $g^t_* d\nu = d\nu. $
\end{prop}

The proof is similar to that in  the $\psi$DO case.

We note that $I^0(M \times M, C)$ is a right and left module over
$\Psi^0(M)$. Given $F \in I^0(M \times M, C)$ we consider the
operators $AF, FA \in I^0(M \times M, C)$ where $A \in \Psi^0(X)$.
We obtain  a useful improvement on Proposition \ref{ONE}.

\begin{prop} \label{TWO} With the same hypotheses as in Proposition \ref{ONE}. Assume further
that $F$ is self-adjoint and that $\rho_j(AF) \sim \rho_j(FA)$.
Then we have:
$$\rho_{\infty}(\pi_X^* a \sigma) = \rho_{\infty}(\pi_Y^* a
\sigma). $$
\end{prop}

This holds since
$$\rho_{\infty}(\pi_X^* a \sigma) \sim \rho_j(A F) = \langle A F \phi_j, \phi_j \rangle = \overline{\langle F^*
A^* \phi_j, \phi_j \rangle} = \overline{\langle F A^* \phi_j,
\phi_j \rangle} \sim \pi_Y^* a \sigma $$ since $\sigma_{A^*} \sim
\overline{\sigma_A}$ and since $\overline{\sigma_F} = \sigma_F.$

Let us consider some simple examples. First, suppose that $F = T_g
+ T_{g}^*$ where $g$ is an isometry. Then $C$ is the union of the
graph of the lift of $g$ to $T^*M$ (by its derivative) and the
inverse graph. The above Proposition then says that $g_* \nu =
\nu$. This is obvious by Egorov's theorem, and we can regard the
Proposition as a generalization of Egorov's theorem to symplectic
correspondences.  As a second example, let $F = U^t$. Then
$\langle A F \phi_j, \phi_j \rangle = e^{i t \lambda_j} \langle A
\phi_j, \phi_j \rangle$, so we see that sequences with unique
quantum limits are sparser than in the pseudo-differential case.

\section{\label{HECKE} QUE of Hecke eigenfunctions}

Our next topic is Hecke eigenfunctions. A detailed treatment would
have to involve adelic dynamics, higher rank measure rigidity in
the presence of positive entropy \cite{LIND,LIND2} and L-functions
\cite{HS,Sound1}. Both areas are outside the scope of this survey.
Fortunately, Lindenstrauss has written an expository article on
adelic dynamics and higher rank rigidity for an earlier Current
Developments in Mathematics \cite{LIND2}, and Soundararajan has
recently lectured on his L-function results in a 2009 Clay
lecture. Sarnak has recently written  exposition of the new
results of Soundararajan and  Holowinsky-Soundararajan results is
\cite{Sar3}.

So we continue in the same vein of explaining the microlocal
(phase space) features of the Hecke eigenvalue problem. The
distinguishing feature of Hecke eigenfunctions is that they have a
special type of symmetry. As a result, their quantum limit
measures  have a special type of symmetry additional to geodesic
flow invariance. The analysis of these additional symmetries is  a
key input into Lindenstrauss' QUE result.  The main purpose of
this section is to derive the exact symmetry. In fact, it appears
that the exact  symmetry was not  previously determined;   only a
quasi-invariance condition of \cite{RS} has been employed.

We begin by recalling the definition of a Hecke operator. Let
$\Gamma$ be a co-compact or cofinite discrete subgroup of $G =
PSL(2, \R)$ and let  $\X$ be the corresponding compact (or finite
area) hyperbolic surface. An element $g \in G, g \notin \Gamma$ is
said to be in the commensurator $Comm(\Gamma)$ if
$$\Gamma'(g) : = \Gamma \cap g^{-1} \Gamma g$$
is of finite index in $\Gamma$ and $g^{-1} \Gamma g$. More
precisely,
$$\Gamma = \bigcup_{j = 1}^d \Gamma'(g) \gamma_j, \;\;
(\mbox{disjoint}), $$ or equivalently
$$\Gamma g \Gamma = \bigcup_{j = 1}^d \Gamma \alpha_j, \;\; \mbox{where}\;\;\alpha_j = g \gamma_j. $$
It is also possible to choose $\alpha_j$ so that $\Gamma g \Gamma
=   \bigcup_{j = 1}^d \alpha_j \Gamma. $ We then have a diagram of
finite (non-Galois) covers:

\begin{equation}\label{DIAGRAM1} \begin{array}{ccccc} & & \Gamma'(g) \backslash \H & &  \\ & & & & \\
& \pi \swarrow & & \searrow \rho & \\ & & & & \\
\Gamma \backslash \H & & \iff & & g^{-1} \Gamma g \backslash \H .
\end{array}
\end{equation}
Here,
$$\pi(\Gamma'(g) z) = \Gamma z, \;\; \rho(\Gamma'(g) z) = g \Gamma
g^{-1} (g \gamma_j z) = g \Gamma z,  $$ where in the definition of
$\rho$ any of the $\gamma_j$ could be used.  The horizontal map is
$z \to g^{-1} z$.

A Hecke operator \begin{equation} \label{HECKO} T_g f(x) = \sum_{j
= 1}^d f(g \gamma_j x) : L^2 (\X) \to L^2(\X) \end{equation} is
the Radon transform $\rho_* \pi^*$ of the diagram.  We note as a
result that
$$T_{g^{-1}} u(z) = u(g^{-1} z)$$
takes $\Gamma$-invariant functions to $g^{-1} \Gamma g$-invariant
functions. The Hecke operator is thus an  averaging operator over
orbits of the Hecke correspondence, which is the multi-valued
holomorphic map
$$C_g (z) = \{\alpha_1 z, \dots, \alpha_d z\}. $$
Its graph
$$\Gamma_g = \{(z, \alpha_j z): z \in \X)\}$$
is an algebraic curve in $\X \times g^{-1} X g$.

The  Hecke operators $T_g$ form a commutative ring as $g$ varies
over $Comm(\Gamma)$. By a Hecke eigenfunction is meant a joint
eigenfunction of $\Delta$ and the ring of Hecke operators:
$$T_g u_j = \mu_j(g) u_j, \;\; \Delta u_j = - \lambda_j^2 u_j. $$

\subsection{Quantum limits on $\Gamma'(g) \backslash \H$}

We now considered symmetries of quantum limits of Hecke
eigenfunctions. We assume $u_j$ is a Hecke eigenfunction.  It is
convenient to modify the definitions of the functionals $\rho_k$
as follows:

\begin{equation} \left\{ \begin{array}{l} \rho'_j(A) = \frac{\langle A \pi^* u_j, \rho^*
u_j \rangle}{\langle \pi^* u_j, \rho^* u_j \rangle} =
\frac{\langle \rho_* A \pi^* u_j,
u_j \rangle}{\langle \pi^* u_j, \rho^* u_j \rangle}, \\ \\
 \sigma'_j(A) = \frac{\langle A \pi^* u_j, \pi^*
u_j \rangle}{\langle \pi^* u_j, \rho^* u_j \rangle}, \\ \\
 \tau'_j(A) = \frac{\langle A \pi^* u_j, \rho^*
u_j \rangle}{\langle \rho^* u_j, \rho^* u_j \rangle} \end{array}
\right. \end{equation} on $A  \in \Psi^0(\Gamma'(g) \backslash
\H). $  Thus, if $A_0 \in \Psi^0(\X)$ we have,
\begin{equation} \frac{\langle A_0 \rho_* \pi^* u_j, u_j \rangle
}{\langle \rho_* \pi^* u_j, u_j \rangle} = \langle A_0 u_j, u_j
\rangle. \end{equation}

Since Hecke operators are FIO's, limit measures of $\langle A T_g
u_j, u_j \rangle$ and $\langle T_g A u_j, u_j \rangle$ live on the
canonical relation underlying $T_g$, which is the lift of the
graph of the Hecke correspondence to $T^* \X$. It is convenient to
restate the result by forming a second diagram

\begin{equation}\label{DIAGRAMa} \begin{array}{ccccc} & & \Gamma_g \subset \X
\times g^{-1} \Gamma g \backslash \H & &  \\ & & & & \\
& \pi_1 \swarrow & & \searrow \rho_1 & \\ & & & & \\
\Gamma \backslash \H & & \iff & & g^{-1} \Gamma g \backslash \H .
\end{array}
\end{equation}
Here, $\pi_1, \rho_1 $ are the natural projections. In the case of
arithmetic groups such as $SL(2, \Z)$, $Comm(\Gamma)$ is dense in
$G$, i.e. there are many Hecke operators. The  map
$$\iota: \Gamma'(g)
\backslash \H \to \Gamma_g$$ defined by  $$\iota(z) = (\pi(z),
\rho(z)), $$ is a local diffeomorphic parametrization, and $\pi_1
\iota = \pi, \rho_1 \iota = \rho. $

The general Proposition \ref{TWO} then specializes to:

\begin{prop} If $\omega'$ is a weak* limit of $\rho_j'$,
and if $\omega$ is the weak* limit on $\X$ then $\pi_* \omega' =
\rho_* \omega' = \omega. $ \end{prop}

We may reformulate the proposition as follows: if we lift measures
and functions to the universal cover $\Hh$, then the data of the
Hecke limit measures are a finite set of real signed  measures
$\{d\nu_k\}$ on the different copies of $S^*\Hh$. The latter must
project to the former under both projections, so we get the exact
invariance property,
\begin{equation} \label{EXACT} \nu = \sum_{j = 1}^d \nu_{\alpha_j} =  \sum_{j = 1}^d  \alpha_{j*}
\nu_{\alpha_j},
\end{equation}
for some complex measures  $\nu_j$ on the $j$ sheet of the cover.
The sum of the masses of the $\nu_j$ must add up to one.  This is
simpler to see for Hecke operators on $S^2$. There we also have
some number $r$ of copies of $S^2$ and $r$ isometries $g_j$ and
define the Hecke operator (of Lubotzky-Phillips-Sarnak) as in
(\ref{HECKO}). Then a quantum limit measure $\nu_0$ on $S^* S^2$
is the projection of $r$ limit measures on the copies of $S^*S^2$
under both projections.

Previously, only a quasi-invariance property of quantum limit
measures of Hecke eigenfunctions was proved by Rudnick-Sarnak
\cite{RS}, and used in \cite{BL,W3} and elsewhere. In the notation
above, it may be stated as
\begin{equation} \label{EXACTb} \nu(E)  \leq  d \sum_j |\alpha_{j*}
\nu_{\alpha_j} (E)|.
\end{equation} As a simple application of
(\ref{EXACT}) (which was also clear from the quasi-invariance
condition), we sketch a proof that a quantum limit measure  $\nu$
of a sequence of Hecke eigenfunctions for $\Gamma = SL(2, \Z)$
cannot be a periodic orbit measure $\mu_{\gamma}$. If it were, it
would have to be $\sum \nu_j$ as above. Each $\nu_j$ is an
invariant signed measure so by ergodic decomposition, each must
have the form $c_j \mu_{\gamma} + \tau_j$ where $\sum c_j = 1$,
where $\tau_j$ are singular with respect to $\mu_{\gamma}$  and
$\sum_j \tau_j = 0$. But also the images of $\nu_j$ under the
$\alpha_j$ must have the same property.  This is only possible if
$\alpha_j (\gamma) = \gamma$ up to a translation by $\beta \in
\Gamma$ whenever $\nu_j \not= 0$. At least one of the $c_j \not=
0$. At this point, one must actually consider the elements
$\alpha_j$ in the Hecke operator. In the case of $PSL(2, \Z)$ they
are parabolic elements plus one elliptic element. But only a
hyperbolic element can fix a geodesic (its axis).

It would be interesting to see if one can determine how the
$\nu_{\alpha_j}$ are related to each other. Lindenstrauss's QUE
result (in the next section) implies that  that $(T_g)_* \nu =
\nu$, which appears to say that the $\nu_{\alpha_j}$ are all the
same. An apriori proof of this would simplify the proof of QUE.

\subsection{\label{LINDEN} QUE results of Lindenstrauss, Soundararajan  (and Holowinsky-Soundararajan)}

We now briefly recall the QUE results on Hecke eigenfunctions.

\begin{theo} (E. Lindenstrauss \cite{LIND}) Let $\X$ be a compact arithmetic hyperbolic surface.
Then QUE holds for Hecke eigenfunctions, i.e. the entire sequence
of Wigner distributions of Hecke eigenfunctions tends to Liouville
measure. \end{theo}

In the non-compact finite area case of $\Gamma = PSL(2, \Z)$,
Lindenstrauss proved:
\begin{theo} (E. Lindenstrauss \cite{LIND}) Let $\X$ be the  arithmetic hyperbolic
surface defined by a congruence subgroup.  Then any weak* limit of
the   sequence of Wigner distributions of Hecke eigenfunctions is
a constant multiple of Liouville measure (the constant may depend
on the sequence).
\end{theo}

Note that the quantum ergodicity theorem does not apply in this
finite area case. In the arithmetic case there is a discrete
spectrum corresponding to cuspidal eigenfunctions and a continuous
spectrum corresponding to Eisenstein series. The following was
proved in  \cite{Z8}:

\begin{theo}  Let $\X$ be the  arithmetic hyperbolic
surface defined by a congruence subgroup, and let $\{\phi_j\}$ be
any orthonormal basis of cuspidal eigenfunctions. Then for any
pseudo-differential operator of order zero and with compactly
supported symbol,
$$\frac{1}{N_c(\lambda)} \sum_{j: \lambda_j \leq \lambda} |\langle
A \phi_j, \phi_j \rangle - \omega(A)|^2 \to 0, $$ where
$N_c(\lambda) \sim |\X| \lambda^2$ is the number of cuspidal
eigenvalues $\leq \lambda. $
\end{theo}

Recent work of Soundararajan that the constant equals one, i.e.
there is no mass leakage at infinity. Hence the theorem is
improved to

\begin{theo} (E. Lindenstrauss \cite{LIND}, Soundarajan \cite{Sound1}) Let $\X$ be the  arithmetic hyperbolic
surface defined by a congruence subgroup.  Then any weak* limit of
the   sequence of Wigner distributions of Hecke eigenfunctions is
is normalized  Liouville measure.
\end{theo}

Holowinsky \cite{Hol} and Holowinsky-Soundararajan \cite{HS} have
proven  QUE  results for holomorphic forms (i.e. for the discrete
series; see \S \ref{REPRE}). However, the methods are entirely
different; they are based on the theory of L-functions and
Poincar\'e series. The Hecke property is exploited through
multiplicativity of Fourier coefficients.

\subsection{\label{QUEMULT} Hecke QUE and multiplicities} A natural question (raised at
the Clay talk of Soundararajan) is whether QUE holds for {\it any}
orthonormal basis of eigenfunctions of an arithmetic quotient if
it holds for the Hecke eigenbasis and if the discrete spectrum of
$\Delta$ has bounded multiplicities. The proof that this is true
is a simple modification of Proposition \ref{BOUNDED} and we pause
to sketch it now.

 \begin{prop}  \label{PONEb} Suppose that one orthonormal basis of  $\Delta$ is QUE, and that
 the eigenvalues have uniformly bounded multiplicities. Then all orthonormal bases are QUE.
 \end{prop}

Proof: As in  Proposition \ref{BOUNDED}, consider sequences
 $\{(\lambda_{i_r}, \lambda_{j_r}), \; i_r \not= j_r\}$ of pairs
 of eigenvalues of $\sqrt{\Delta}$ with $\lambda_{i_r} =
 \lambda_{j_r}. $ We then have that
$d W_{i_r, j_r} \to 0$.

This has implications for all orthonormal basis as long as
eigenvalue multiplicities are uniformly bounded. We run the
argument regarding modes and quasi-modes but use the QUE
orthonormal basis in the role of modes and any other orthonormal
basis $\{\psi_k\}$ in the role of quasi-modes.

\begin{lem} If the eigenvalue multiplicities are uniformly bounded
and  there exists a sequence of eigenfunctions $\{\psi_k\}$
\begin{itemize}

\item  (i)  $n(k)  \leq C,\; \forall \; k$;

\item (ii) $\langle A \psi_k, \psi_k \rangle \to \int_{S^*M}
\sigma_A d \mu$ where $d \mu \not= d \mu_L$.

\end{itemize}

Then no orthonormal basis of eigenfunctions of $\Delta$ is  QUE.
\end{lem}

\begin{proof}

We  argue by contradiction.  If $\{\phi_j\}$  were QUE, we would
have (in the notation of (\ref{PSIK}):
$$\begin{array}{ll}\langle A \psi_k, \psi_k \rangle & = \sum_{i, j = 1}^{n(k)} c_{kj} c_{k i}\langle A
\phi_i, \phi_j \rangle + o(1) \\ & \\&  = \sum_{j = 1}^{n(k)}
c_{kj}^2 \langle A \phi_j, \phi_j \rangle + \sum_{i\not= j =
1}^{n(k)} c_{kj} c_{k i}\langle A \phi_i, \phi_j
\rangle + o(1)\\ & \\
& = \int_{S^*M} \sigma_A d\mu_L + o(1),
\end{array} $$
by Proposition \ref{PONEb}. This contradicts (ii). In the last
line we used $\sum_{j = 1}^{n(k)} |c_{kj}|^2  = 1 + o(1)$, since
$||\psi_k||_{L^2} = 1$.

\end{proof}

\section{Variance estimates: Rate of quantum ergodicity and mixing}

A quantitative refinement of quantum ergodicity is to  ask at what
rate the sums in Theorem \ref{QE}(1) tend to zero, i.e. to
establish a rate of quantum ergodicity. More generally, we
consider `variances' of matrix elements.  For diagonal matrix
elements, we define:
\begin{equation} \label{diag} V_A(\lambda) : =
\frac{1}{N(\lambda)} \sum_{j:  \lambda_j \leq \lambda} |\langle A
\phi_j, \phi_j) - \omega(A)|^2.
\end{equation}
In the off-diagonal case one may view $|\langle A\varphi_i,
\varphi_j \rangle|^2$ as analogous to $|\langle A \phi_j, \phi_j)
- \omega(A)|^2$. However, the sums in (\ref{SPECMEAS}) are double
sums while those of (\ref{diag}) are single. One may also average
over the shorter intervals $[\lambda, \lambda + 1].$

\subsection{Quantum chaos conjectures}

It  is implicitly conjectured by Feingold-Peres  in \cite{FP} (11)
that
\begin{equation} \label{FPCONJ} |\langle A \phi_i, \phi_j
\rangle|^2 \simeq \frac{C_A (\frac{E_i - E_j)}{\hbar})}{2 \pi
\rho(E)},
\end{equation} where $C_A(\tau) = \int_{- \infty}^{\infty} e^{- i
\tau t} \langle V^t \sigma_A, \sigma_A \rangle dt. $ In our
notation, $\lambda_j = \hbar^{-1} E_j$ (with $\hbar =
\lambda_j^{-1}$ and $E_j = \lambda_j^2$) and $\rho(E) dE \sim d
N(\lambda)$.

On the basis of the analogy between $|\langle A\varphi_i,
\varphi_j \rangle|^2$ and  $|\langle A \phi_j, \phi_j) -
\omega(A)|^2$, it is  conjectured in \cite{FP}  that
$$\ V_A(\lambda)  \sim \frac{ C_{A - \omega(A) I}(0) }{\lambda^{n-1} vol(\Omega)}.
$$
The idea is that $\phi_{\pm} = \frac{1}{\sqrt{2}} (\phi_i \pm
\phi_j)$ have the same  matrix element asymptotics as
eigenfunctions when $\lambda_i - \lambda_j$ is sufficiently small.
But then $2 \langle A \phi_+, \phi_- \rangle = \langle A \phi_i,
\phi_i \rangle - \langle A \phi_j, \phi_j \rangle$ when $A^* = A$.
Since we are taking a difference, we  may replace each matrix
element by $\langle A \phi_i, \phi_i \rangle $ by $\langle A
\phi_i, \phi_i \rangle - \omega(A)$ (and also for $\phi_j$). The
conjecture then assumes that $ \langle A \phi_i, \phi_i \rangle -
\omega(A)$ has the same order of magnitude as $ \langle A \phi_i,
\phi_i \rangle - \langle A \phi_j, \phi_j \rangle$.

\subsection{Rigorous results for Hecke eigenfunctions}

Let  $\X$ be the arithmetic modular surface, with
$\Gamma=SL(2,\mathbb{Z})$, and let  $T_n$, $n\ge 1$ denote the
family of Hecke operators. In addition to the Laplace eigenvalue
problem (\ref{HYPEIG}), the Maass-Hecke eigenforms are
eigenfunctions of the Hecke operators,
$$\quad T_n\phi=\lambda_{\phi}(n)\phi.$$
In the cusp, they have Fourier series expansions
$$\phi(z)=\sqrt{y}\sum_{n\neq 0}\rho_{\phi}(n)K_{ir}(2\pi|n|y)e(nx)$$
where $K_{ir}$ is the $K$-Bessel function and
$\rho_{\phi}(n)=\lambda_{\phi}(n)\rho_{\phi}(1).$ A special case
of the variance sums (\ref{diag}) is the configuration space
variance sum,
$$S_{\psi}(\lambda)=\sum_{\lambda_j\leq \lambda}|\int \psi \phi^2 dV |^2$$
Following the notation of \cite{LS2} and \cite{Zh},we put $d\mu_n
= \phi_{r_n}^2 dV$.

Let  $\psi\in C_{0,0}^\infty(S^* \X)$ and consider the
distribution of
$$\frac{1}{\sqrt{T}}\int_0^T \psi(g^t(x, \xi))dt.$$ Ratner's
central limit theorem for geodesic flows implies that this random
variable tends to a  Gaussian with mean $0$ and the variance given
by the non-negative Hermitian form on $C_{0,0}^\infty(Y)$:
$$V_{\mathrm{C}}(\phi,\psi)=\int_{-\infty}^{\infty}\int_{\Gamma\setminus
SL(2,\mathbf{R})}\phi\left(g\left(\begin{array}{cc}
e^{\frac{t}{2}}
& 0\\
0& e^{-\frac{t}{2}}\end{array}\right)\right)\overline{\psi(g)}dg
dt.$$ Restrict $V_C$ to $C_{0,0}^{\infty}(X)$. Different
irreducible representations are orthogonal under the  quadratic
form.    Maass cusp forms $\phi$ are eigenvectors of $V_C$. If the
Laplace eigenvalue is $\frac14+r^2$, then the  $V_C$ eigenvalue is
$$V_{\mathrm{C}}(\phi,\phi)=\frac{|\Gamma(\frac14-\frac{ir}{2})|^4}{2\pi|\Gamma(\frac12-ir)|^2}$$

 In \cite{LS,LS2}, Luo and Sarnak studied the
quantum variance for holomorhic Hecke eigenforms, i.e. holomorphic
cusp forms in $S_k(\Gamma)$ of even integral weight $k$ for
$\Gamma$. In this setting, the weight of the cusp form plays the
role of the Laplace eigenvalue. They proved that for $\phi,\psi\in
C_{0,0}^\infty(X)$, as the weight $K\to \infty$,
$$\sum_{k\leq K}\sum_{f\in H_k}L(1,\text{Sym}^2
f)\mu_f(\phi)\overline{\mu_f(\psi)} \sim B(\phi,\psi) K,$$ where
$B(\phi,\psi)$ is a non-negative Hermitian form on
$C_{0,0}^\infty(X)$.  $B$ is diagonalized by the orthonormal basis
of Maass-Hecke cusp forms and the eigenvalues of $B$ at $\psi_j$
is $\frac{\pi}{2}L(\frac12,\psi_j)$. For the notation
$L(1,\text{Sym}^2)$ we refer to \cite{LS2,Zh}.

In \cite{Zh}, P. Zhao  generalized  the result in \cite{LS,LS2} to
Maass cusp forms. Let $\varphi_j(z)$ be the $j$-th Maass-Hecke
eigenform, with the Laplacian eigenvalues $\frac14+r_j^2.$
\begin{theo} \cite{Zh} Fix any
$\epsilon>0$. Then we have
\begin{eqnarray}
& &
\quad\sum_{t_j}(e^{-(\frac{r_j-T}{T^{1-\epsilon}})^2}+e^{-(\frac{r_j+T}{T^{1-\epsilon}})^2})L(1,\mathrm{sym}^2\varphi_j)\mu_j(\phi)\overline{\mu_j}(\psi)\nonumber\\
&=&T^{1-\epsilon}V(\phi,\psi)+O(T^{\frac12+\epsilon}),\nonumber
\end{eqnarray}
$V$ is diagonalized by the orthonormal basis $\{\phi_j\}$ of
Maass-Hecke cusp forms and the eigenvalue of $V$ at a Maass-Hecke
cusp form $\phi$ is
$$L(\frac12,\phi)V_{\mathrm{C}}(\phi,\phi).$$
\end{theo}

It would be desirable to remove the weights. The unweighted
version should be (P. Zhao, personal communication; to appear)

\begin{equation}
\sum_{\lambda_j\leq \lambda}\mu_j(\phi)\overline{\mu}_j(\psi)\sim
\lambda V(\phi,\psi).
\end{equation}

\subsection{General case}

The only  rigorous result to date which is valid on general
Riemannian manifolds with hyperbolic geodesic flow  is the
logarithmic decay:

\begin{theo} \cite{Z4} For any $(M, g)$ with hyperbolic geodesic flow,
$$ \frac{1}{N(\lambda)}
 \sum_{\lambda_j \leq \lambda}
|(A\phi_j,\phi_j)-\omega(A)|^{2p} = O(\frac{1}{(\log \lambda)^p}).
$$
\end{theo}
It is again based on Ratner's central limit theorem for the
geodesic flow.  The logarithm reflects the exponential blow up in
time of remainder estimates for traces involving the wave group
associated to hyperbolic flows and thus the necessity of keeping
the time less than the Ehrenfest time (\ref{EHRENTIME}).  It would
be surprising if the logarithmic decay is sharp for Laplacians.
However,  R. Schubert shows in \cite{Schu,Schu2} that the estimate
is sharp in the case of two-dimensional hyperbolic quantum cat
maps. Hence the estimate cannot be improved by semi-classical
arguments that hold in both settings.

\subsection{Variance and Patterson-Sullivan distributions}

When $\int_{S^* \X} a d\mu_L = 0$, the Luo-Sarnak variance results
for Hecke eigenfunctions have the form,
\begin{equation} \frac{1}{N(\lambda)} \sum_{j: |r_j| \leq \lambda} \left| \langle a, dW_{ir_j}
\rangle \right|^2 \sim \frac{1}{\lambda} V(a, a). \end{equation}
Even if one knew that asymptotics of this kind should occur, it
does not seem apriori obvious that the bilinear form on the right
side should be invariant under $g^t$ because $V(a,a)$ occurs in
the $\frac{1}{\lambda}$ term in the variance asymptotics and only
only knows that the Wigner distributions are invariant modulo
terms of order $\frac{1}{\lambda_j}$. A quick way to see that the
bilinear form must be $g^t$-invariant is via the
Patterson-Sullivan distributions. From (\ref{WIGPS}), we have
\begin{equation} \frac{1}{N(\lambda)} \sum_{j: |r_j| \leq \lambda}
|\langle a, dW_{ir_j} \rangle|^2 = \frac{1}{N(\lambda)} \sum_{j:
|r_j| \leq \lambda} |\langle a, d\hat{PS}_{ir_j} \rangle|^2 +
O(\lambda^{-2}). \end{equation} Indeed, the difference is of the
form $\frac{1}{N(\lambda)} \sum_{j: |r_j| \leq \lambda}
 \frac{2}{r_j} \Re \langle a, d \hat{PS}_{ir_j} \rangle \cdot
\overline{ \langle L_2 a, d\hat{PS}_{ir_j} \rangle}  +
O(\lambda^{-2}).$ The sum in the last expression is also of order
$O(\lambda^{-2})$ (e.g. one can go back and express each factor of
$PS $ with one of $W_{ir_j}$ and apply the Luo-Sarnak asymptotics
again).  Since $\hat{PS}_{ir_j}$ is $g^t$-invariant, it follows
that $V(a,a)$ must be $g^t$ invariant. This argument is equally
valid on any  compact hyperbolic surface, but of course there is
no proof of such asymptotics except in the arithmetic Hecke case.
It would be interesting to draw further relations between $V(a,a)$
and $PS_{ir_j}$.

\section{\label{ENTROPY}  Entropy of quantum limits on manifolds with
Anosov geodesic flow}

So far, the results on quantum limits have basically used the
symmetry (or invariance) properties of the limits. But generic
chaotic systems have no symmetries. What is there to constrain the
huge number of possible limits?

The  recent results of  Anantharaman \cite{A},
 Anantharaman- Nonnenmacher \cite{AN} and Rivi\`ere \cite{Riv}, give a
 very interesting
 answer to this question (see also \cite{ANK})  for $(M, g)$ with Anosov geodesic flow. They use
 the quantum mechanics to prove
lower bounds on  entropies of  quantum limit measures.  The lower
bounds eliminate many of the possible limits, e.g. they disqualify
finite sums of periodic orbit measures.

The purpose of this section is to present the  results of \cite{A}
and \cite{AN} in some detail. Both articles are based on a
hyperbolic dispersive estimate (called the {\it main estimate} \;
in \cite{A}) which, roughly speaking, measures the norm of
quantized cylinder set operators in terms of $\hbar$ and the
length of the cylinder.  But they use the estimate in quite
different ways. In \cite{A}, it is used in combination with an
analysis of certain special covers of $S^*M$ by cylinder sets that
are adapted to the eigenfunctions. One of the important results is
an estimate on the topological entropy $h_{top}(\mbox{supp}
(\nu_0))$ of the support of any quantum limit measure.  In
\cite{AN} the key tool is the ``entropic uncertainty principle".
It leads to a lower bound for the Kolomogorv-Sinai entropy
$h_{KS}(\nu_0)$ of the quantum limit.

There now exist several excellent and authoritative expositions of
the KS entropy bounds and the entropic uncertainty principle
\cite{A2,ANK,AN2,CV3} in addition to the well written initial
articles \cite{A,AN} (which take considerable care to give
intuitive explanations of technical steps). For this reason, we
emphasize the approach in \cite{A}.  We closely follow the
original references in discussing heuristic reasons for the lower
bounds and outlining rigorous proofs. We also discuss earlier
entropy lower bounds of Bourgain-Lindenstrauss \cite{BL} and
Wolpert \cite{W3} in the case of Hecke eigenfunctions.

Before stating the results, we review the various notions  of
entropy.

\subsection{KS entropy}

We first  recall the definition of the KS entropy of an  invariant
probability measure $\mu$ for the geodesic flow. Roughly speaking,
the entropy measures  the average complexity of $\mu$-typical
orbits. In the Kolmogorov-Sinai entropy, one starts with a
partition $\pcal = (E_1, \dots, E_k)$ of $S^*_g M$ and defines the
Shannon entropy of the partition by $h_{\pcal}(\mu) = \sum_{j =
1}^k \mu(E_j) \log \mu(E_j). $ Under iterates of  the time one map
$g$ of the geodesic flow, one refines the partition from the sets
$E_j$ to the cylinder sets of length $n$:
$$\pcal^{n}: = \{ \pcal^{ n}_{{\bf \alpha}}: = E_{\alpha_0} \cap g^{-1} E_{\alpha_1} \cap
\cdots \cap g^{-n + 1}  E_{\alpha_{n-1}}:\;\; [\alpha_0, \dots,
\alpha_{n-1}] \in {\bf N}^n \}.
$$ One defines $h_n(\pcal, \mu)$ to be the Shannon entropy of this
partition and then defines $h_{KS}(\mu, \pcal) = \lim_{n \to
\infty} \frac{1}{n} h_n(\mu, \pcal)$. Then $h_{KS}(\mu) =
\sup_{\pcal} h_{KS}(\mu , \pcal)$.

The measure $\mu( E_{\alpha_0} \cap g^{-1} E_{\alpha_1} \cap
\cdots \cap g^{-n + 1}  E_{\alpha_{n-1}})$ is the probability with
respect to $\mu$ that an orbit successively visits $E_{\alpha_0},
E_{\alpha_1}, \cdots. $ The entropy measures the exponential decay
rate of these probabilities for large times.

\subsection{ Symbolic coding and cylinder sets }

In the $(M,g)$ setting, we fix a partition $\{M_k\}$ of $M$ and a
corresponding partition $T^* M_k$ of $T^*M$. Let $P_{\alpha_k}$ be
the characteristic function of $T^*M_k$. (Later it must to be
smoothed out). Let $\Sigma = \{1, \dots, \ell\}^{\Z}$ where $\ell$
is the number of elements of the partition $\pcal^{0)}$. To each
tangent vector $v \in S^*M$, one can associated a unique element
$I(v) =(\alpha_k) \in \Sigma$ so that $g^n v \in P_{\alpha_j}$ for
all $n \in \Z$. This gives a symbolic coding map $I : S^*M \to
\Sigma$. The time one map $g^1$ then conjugates under the coding
map to the shift $\sigma ((\alpha_j)) = ((\alpha_{j + 1})) $ on
admissible sequences, i.e. sequences in the image of the coding
map.  An invariant measure $\nu_0$ thus corresponds to a shift
invariant measure $\mu_0$ on $\Sigma$.

A cylinder set $[\alpha_0, \dots, \alpha_{n-1}] \subset \Sigma$ of
length $n$ is the subset of $\Sigma$ formed by sequences with the
given initial segment. The set of such cylinder sets of length $n$
is denoted $\Sigma_n$. Cylinder sets for $\Sigma$ are not the same
as cylinder sets for $g^1$, which have the form $P_{\alpha_0} \cap
g^{-1} P_{\alpha_1} \cap \cdots \cap g^{-n} P_{\alpha_n}. $

The  $\mu_0$ measure of a $\Sigma$-cylinder set  is by definition,
$$\mu_0([\alpha_0, \dots, \alpha_n]) = \nu_0(P_{\alpha_0} \cap
g^{-1} P_{\alpha_1} \cap \cdots \cap g^{-n + 1} P_{\alpha_{n-1}}).
$$

\subsection{Bowen balls}

Cylinder sets are closely related to Bowen balls, i.e. balls in
the Bowen metric.
 For any smooth map  $f: X \to X$, the Bowen metric at time $n$ is
 defined by
$$d_n^f(x,y) = \max_{0 \leq i \leq n -1} d(f^i x, f^i y). $$
Let $B_n^f(x, r)$ be the $r$ ball around $x$ in this metric.

In the continuous time case of the geodesic flow $g^t$, one
defines the $d_T$ metric
\begin{equation} \label{dT} d_T(\rho, \rho') = \max_{t \in [0
T]} d(g^t (\rho), g^t(\rho')). \end{equation} The  Bowen ball
$B_T(\rho, r)$  at time $T$ is the $r$- ball in this metric.

As with cylinder sets, if $\rho,\rho'$ lie in $B_T(\rho,
\epsilon)$ then their orbits are $\epsilon$-close for the interval
$[0, T]$. This is not quite the same as running through the same
elements of a partition but for large $T$ the balls and cylinders
are rather similar. This is because the  geodesic flow $g^t$
stretches everything by a factor of $e^t$ in the unstable
direction, and contracts everything  by $e^{-t}$ in the stable
direction; it preserves distances along the geodesic flow. In
variable curvature and in higher dimensions, the `cube' becomes a
rectangular parallelopiped whose axes are determined by the
Lyapunov exponents.

 In the case of geodesic flows of compact hyperbolic manifolds of
 constant curvature $-1$,
Bowen balls $B_T(\rho, r)$  are roughly of radius  $ r e^{- T}$ in
the unstable direction, and $ r$ in the stable direction and
geodesic flow directions.  In  the hyperbolic case of $G/\Gamma$
with co-compact $\Gamma$, the tube (in the notation of \S
\ref{REP}) is
\begin{equation} \label{NALINIS} B_T(\rho, \epsilon)) = a((-
r, r)) n_-(r,r))n_+((- e^{-T} r , e^{-T} r)).
\end{equation}

To make the ball symmetric with respect to the stable and unstable
directions is to make it symmetric with respect to time reversal.
One uses the time  interval $[-T/2, T/2]$ instead of $[0, T]$ and
defines the new distance,
\begin{equation} \label{dTsym} d'_T(\rho, \rho') = \max_{t \in -[T/2,
T/2]} d'(g^t (\rho), g^t(\rho')), \end{equation} We then denote
the $r$ ball by $B'_T(\rho, r)$.
 In  the hyperbolic case,
\begin{equation} \label{NALINIShyp} B'_T(\rho, \epsilon)) = a((-
\epsilon, \epsilon)) n_-((-e^{-T/2} \epsilon, e^{- T/2}\epsilon
))n_+((- e^{-T/2} \epsilon , e^{-T/2} \epsilon)).
\end{equation}

Thus, the Bowen ball in the symmetric case (in constant curvature)
is a ball (or cube) of radius $r e^{- T}$ in the transverse
direction to the geodesic flow. The length along the flow is not
important.

\subsection{Brin-Katok local entropy}

Kolmogorov-Sinai entropy of an invariant measure is related to the
local dimension of the measure on Bowen balls. This  is the
approach in Bourgain-Lindenstrauss \cite{BL} and in \cite{LIND},
and stems from work of Young and Ledrappier-Young.

Define the local entropy on an invariant measure $\mu$ for a map
$f$ by
$$h_{\mu}(f, x) = \lim_{\delta \to 0} \limsup_{n \to \infty}
\frac{- \log \mu(B_n^f(x, \delta))}{n}  = \lim_{\delta \to 0}
\liminf_{n \to \infty} \frac{- \log \mu(B_n^f(x, \delta))}{n}. $$

Brin-Katok proved that both limits exist,  that the local entropy
$h_{\mu}(f, x)$ is $f$ invariant and
$$h_{\mu}(f) = \int_X h_{\mu}(f, x)d \mu(x). $$

 The definition and result is similar in  the geodesic flow case.
 The  local entropy of $\mu$ at $\rho$ is defined by,
$$\lim_{\epsilon \to 0} \liminf_{T \to \infty} - \frac{1}{T} \log
\mu(B_T(\rho, \epsilon)) = \lim_{\epsilon \to 0} \limsup_{T \to
\infty} - \frac{1}{T} \log \mu(B_T(\rho, \epsilon)) := h_{KS}
(\mu, \rho). $$

\subsection{Ergodic decomposition}

Let $\nu$ be an invariant measure. It may be expressed as a convex
combination of ergodic measures, which are extreme points of the
compact convex set of invariant measures. There is a concrete
formula,
\begin{equation} \label{ERGDEC} \nu = \int \nu_x^{\ecal} d\nu(x)
\end{equation} for its  ergodic decomposition, where  $\nu_x^{\ecal}$
is the orbital average through $x \in G/\Gamma$.  It is a fact
that $\nu_x^{\ecal}$  ergodic for $\nu$-a.e. $x$. Then the local
entropy is
\begin{equation} h_{a(t)}(\tilde{\mu}_x^{\ecal} ) = \lim_{\epsilon
\to 0} \frac{\tilde{\mu}(T(x, \epsilon))}{\log \epsilon}, \;\;
a.s. \tilde{\mu}. \end{equation} The Brin-Katok theorem states
that the global KS entropy is
$$h(\nu) = \int h(\nu_x^{\ecal}) d\nu(x). $$

\subsection{\label{TOPENTa}Topological entropy of an invariant subset}

Anantharaman's first entropy bound  refers to the topological
entropy of an invariant set $F$.  Then, by definition,
 \begin{equation}
\label{htop}\begin{array}{lll}  h_{top}(F) \leq \lambda &\iff&
\forall \delta
> 0,\; \exists C > 0: F \;\; \mbox{can be covered by at most} \\&&\\
&&\; C e^{n(\lambda + \delta)} \;\;\mbox{cylinders of length}\; n,
\;\; \forall n. \end{array} \end{equation}

\subsection{Bourgain-Lindenstrauss entropy bound}

Before discussing the work of Anantharman-Nonnenmacher, we review
a simpler entropy bound of Bourgain-Lindenstrauss. After initial
work of  Wolpert \cite{W3}, Bourgain-Lindenstrauss \cite{BL}
obtained a strong lower bound on entropies of quantum limit
measures associated to Hecke eigenfunctions. Although it was later
surpassed by Lindenstrauss' QUE theorem \cite{LIND}, it is simple
to state and illustrates the use of the  local entropy and local
dimensions to measure the KS entropy.

The setting is a compact or finite area arithmetic hyperbolic
quotient. We have not reviewed local entropy in the non-compact
finite area case, but proceed by analogy with the compact case.
Then $$B(\epsilon, \tau_0) = = a((- \tau_0, \tau_0))
n_-(\epsilon,\epsilon))n_+((- \epsilon , \epsilon)) $$  is a Bowen
ball around the identity element and $x B(\epsilon, \tau_0)$ is
its left translate by $x \in G. $ Here $\epsilon = r e^{- T/2}$.

\begin{theo}\label{BL}  \cite{BL} Let $\Gamma \subset SL(2, \Z)$ be a congruence
subgroup. Let $\mu$ be a quantum limit. Then for any compact
subset $K \subset \Gamma \backslash G$ and any $x \in K$,
$$\mu (x B(\epsilon, \tau_0)) \leq \epsilon^{2/9}. $$
\end{theo}

In the flow-box notation above,
\begin{equation} \mu_n(B_T(\rho, \epsilon)) \leq C e^{- T/9}.
\end{equation}
This implies that any ergodic component of any quantum limit has
entropy $\geq \frac{1}{9}. $

\begin{cor} \cite{BL} Almost every component of a quantum limit measure has
entropy $\geq 1/9$. The Hausdorff dimenson of the support of the
limit measure is $\geq 1 + 1/9$.

\end{cor}

Theorem \ref{BL} is a consequence of the following uniform upper
bound for masses in small tubes around geodesic segments in
configuration space:

\begin{theo}\cite{BL} Let $\Gamma \subset SL(2, \Z)$ be a congruence
subgroup. Let $\mu$ be a quantum limit. Then for any compact
subset $K \subset \Gamma \backslash G$ and any $x \in K$, and for
any Hecke eigenfunction $\phi_{\lambda_j}$,
$$\int_{x B(r
, \tau_0)} |\phi_{\lambda_j}|^2 dV \leq r^{2/9}. $$
\end{theo}

One may rewrite this result in terms of a Riesz-energy of the
quantum limit measures:

\begin{cor} Let $\mu$ be a quantum limit of a sequence of
Hecke-Maass eigenfunctions. Then for $\kappa < 2/9$,
$$\int_M \int_M \frac{d\mu(x) d\mu(y)}{d_M(x, y)^{\kappa + 1}} <
\infty. $$
\end{cor}

Although these results were superseded by Lindenstrauss' QUE
bound, we briefly sketch the idea of the proof since it makes an
interesting contrast to the Anantharaman-Nonnenmacher bound.  They
prove that there exists a set $W$ of integers of size
$\epsilon^{-2/9}$ so that for $n \in W$ one has (roughly) that the
translates $y B(\epsilon,\tau_0)$ of the small balls by $y $ in
the Hecke correspondence for $T_n$ are all pairwise disjoint.
Since
$$|\phi_{\lambda}(y)|^2 \leq \sum_{z \in T_n(y)}
|\phi_{\lambda}(z)|^2, $$ one has
$$\begin{array}{lll} \int_{x B(\epsilon, \tau_0)}
|\phi_{\lambda}(y)|^2 dV & \leq & C  \int_{x B(\epsilon, \tau_0)}
 \sum_{z \in T_n(y)}
|\phi_{\lambda}(z)|^2 dV \\ && \\ & = &
 \sum_{z \in T_n(x)} \int_{z B(\epsilon, \tau_0)}
|\phi_{\lambda}(z)|^2 dV. \end{array}$$ Now sum over $n \in W$ and
use the disjointness of the small balls $z B(\epsilon, \tau_0)$
and the $L^2$ normalization of the eigenfunction to obtain
$$\begin{array}{lll} \int_{x B(\epsilon, \tau_0)}
|\phi_{\lambda}(y)|^2 dV & \leq & \frac{1}{|W|} \sum _{n \in W}
 \sum_{z \in T_n(x)} \int_{z B(\epsilon, \tau_0)}
|\phi_{\lambda}(z)|^2 dV \\ &&\\
& \leq & C  \epsilon^{2/9} \int_X |\phi_{\lambda}^2 dV \leq C
\epsilon^{2/9}.
\end{array}$$

\subsection{Anantharman and Anantharaman-Nonnenacher lower bound}

We now turn to the  lower bounds on entropy in the general Anosov
case given in \cite{A} and in \cite{AN,ANK}.
 The mechanisms
proving the lower bounds and the results are somewhat different
between the two articles. The result of \cite{A} gives a lower
bound for the topological entropy $h_{top}(supp \nu_0)$ of the
support of  a quantum limit while the main  result of \cite{AN}
gives a lower bound on the metric or Kolmogorov Sinai entropy $h_g
= h_{KS}$ of the quantum limits.  For the definition of
$h_{top}(F)$, see (\ref{htop}).

\begin{theo}\label{A} \cite{A} Let $\nu_0 = \int_{S^* M} \nu_0^x d\nu_0(x)$ be the
ergodic decomposition of $\nu_0$. Then there exists $\kappa > 0$
and two continuous decreasing functions $\tau: [0, 1] \to [0, 1]$
and $\vartheta: (0, 1] \to \R_+$ with $\tau(0) = 1, \vartheta(0) =
\infty$, such that
$$\nu_0 [\left\{ x: h_g(\nu_0^x) \geq \frac{\Lambda}{2}(1 - \delta)
\right\} ] \geq (\frac{\kappa}{\vartheta(\delta)} )^2 (1 -
\tau(\delta)).
$$ Hence

\begin{itemize}

\item  $h_g(\nu_0) > 0$;

\item $h_{top}(\mbox{supp}\; \nu_0) \geq \frac{\Lambda}{2}.$
\end{itemize}
\end{theo}

It follows that a positive proportion of the ergodic components of
$\nu_0$ must have KS entropy arbitrarily close to
$\frac{\Lambda}{2}$.

This is proved as a consequence of a Proposition  which may have
other applications:

\begin{theo} \label{prop}\cite{A} (Proposition 2.0.4) With the same notation as in Theorem \ref{A} and Remark \ref{REMARK}.  Let
 $F$ be a subset with $h_{top}(F) \leq
\frac{\Lambda}{2} (1 - \delta).$ Then,
$$\nu_0 (F) \leq  \left(1 - (\frac{\kappa}{\vartheta(\delta)})^2 \right) (1 -
\tau(\delta)) < 1.
$$
Hence supp $\nu_0$ cannot equal $F$.

\end{theo}

In \cite{A} Theorem 1.1.2,  there is a generalization to
quasi-modes which introduces a constant $c$. For simplicity we
only consider eigenfunctions.

In the subsequent article of Anantharaman-Nonnenmacher,
\cite{AN,ANK} the authors obtain a quantitative lower bound on the
KS entropy:
\begin{theo}\label{thethm} \cite{AN,ANK}
Let $\mu$ be a semiclassical measure associated to the
eigenfunctions of the Laplacian on $M$. Then its metric entropy
satisfies
\begin{equation}\label{e:main1}
h_{KS}(\mu)\geq  \left|\int_{S^*M} \log J^u(\rho)d\mu(\rho)
\right|- \frac{(d-1)}{2} \lambda_{\max}\,,
\end{equation}
where $d=\dim M$.  $\lambda_{\max}=\lim_{|t| \to
\infty}\frac{1}t \log \sup_{\rho\in \cE} |dg^t_\rho|$ is the
maximal expansion rate of the geodesic flow on $\cE$  (\ref{LAMBDAMAX})
and $J^u$ is the unstable Jacobian determinant (\ref{UNSTABLEJAC}).

\end{theo}

When $M$ has constant sectional curvature $-1$, the theorem
implies  that
\begin{equation}\label{e:main2}
h_{KS}(\mu)\geq \frac{d-1}2.
\end{equation}

They state the conjecture
\begin{conj} \cite{AN} Let $(M, g)$ have Anosov geodesic flow. Then for any
quantum limit measure $\nu_0$,
$$h_g(\nu_0) \geq \frac{1}{2} \left| \int_{S^*M} \log J^u(v)
d\nu_0(v) \right|. $$
\end{conj}

This conjecture has recently been proved by G. Rivi\`ere for
surfaces of negative curvature \cite{Riv}.

This conjecture does not imply QUE. In fact the counterexamples in
\cite{FNB,FN} to QUE in the case of ``quantum cat maps'' satisfy
the condition of the conjecture. So there is some chance that the
conjecture is best possible for $(M, g)$ with ergodic geodesic
flow.

To the author's knowledge, the corollary that quantum limits in
the Anosov case cannot be pure periodic orbit measures has not
been proved a simpler way that by applying the theorems above. In
\S \ref{FGAMMA} we will go over the proof in that special case and
see how it ties together with mass concentration of eigenfunctions
around hyperbolic closed geodesics in \S \ref{HYP}.

\subsection{Problem with semi-classical estimates and Bowen balls}

As a first attempt to estimate KS entropies of quantum limits, one
might try to study the local entropy formula as in
Bourgain-Lindenstrauss. Let $B_T(\rho, \epsilon)$ be the small
tube around $\rho \in S^*M$, i.e. of length $\epsilon$ in the flow
direction and $e^{- T/2}$ in the stable/unstable directions. As in
the Bourgain-Lindenstrauss bound, one would like to understand the
decay of $\mu_n (B_T(\rho, \epsilon))$ as $T \to \infty$ and $n
\to \infty$. In the brief expository article \cite{AN2}, the
authors give  the  heuristic estimate
\begin{equation}\label{BTEST}  \mu_n(B_T(\rho, \epsilon)) \leq C
\lambda_n^{\frac{d-1}{2}} e^{ - \frac{(d - 1) T}{2}}, \;\;\; d =
\dim M
\end{equation}
on the  local dimensions of the quantum measures for a sequence of
eigenfunctions.  Unlike the uniform Bourgain-Lindenstrauss
estimate, the estimate  is $\lambda_n$-dependent and  is only
non-trivial for times $T
> \log \lambda_n$. But then the ball radius (in the
stable/unstable directions) is $e^{- T} = \lambda_n^{-1/2}$, which
is just below the minimal scale  allowed by the uncertainty
principle (\S \ref{HUP}).  To obtain a useful  entropy bound one
needs  to take $T = 2 \log \lambda_n$ and then $e^{- T} =
\lambda_n^{-1}$, well below the Heisenberg uncertainty threshold.
So although it is attractive, the local entropy is difficult to
use. Recall that Bourgain-Lindenstrauss used the many elements of
the Hecke correspondences to move the small ball around so that it
almost fills out the manifold, and then used the $L^2$
normalization to estimate the total mass. There doesn't seem to
exist a similar mechanism to build up a big mass from the small
balls in the general Anosov case.

\subsection{ \label{REMARK} Some important parameters}

Two important parameters $\kappa, \vartheta$ appear here and in
the statements and proofs in \cite{A}. They represent time scales
relative to the Ehrenfest time:

\begin{itemize}

\item $\kappa$: Semi-classical estimates only work when $n \leq
\kappa |\log \hbar|$, i.e. $\kappa |\log \hbar|$ plays the role of
the Ehrenfest time.

\item $\vartheta$: The main estimate  is only useful when
$h^{-d/2} e^{- n \Lambda/2} << 1$ or $n \geq \vartheta |\log
\hbar|. $

\end{itemize}

\subsection{Cylinder set operators in quantum mechanics}

To get out of the too-small Bowen ball impasse, Anantharaman
\cite{A} and Ananthraman-Nonnenmacher study ``quantum measures''
of {\it cylinder sets} $\ccal$ rather than   small balls. The
emphasis (and results) on quantum cylinder sets is one of the key
innovations in \cite{A}. The classical cylinder sets are not
directly involved in the main estimates of \cite{A,AN}. Rather
they are quantized as cylinder set operators. Then their ``quantum
measure" or quantum entropy is studied. To distinguish
notationally between classical and quantum objects, we put a
$\hat{\cdot}$ on the quantum operator.

To define quantum cylinder set operators, one quantizes the
partition to define a smooth quantum partition of unity
$\hat{P}_k$ by smoothing out the characteristic functions of
$M_k$. Let ${\bf \alpha} = [\alpha_0, \dots, \alpha_{n-1} \in
\Sigma_n$ be a cylinder set in sequence space $\Sigma$. The
corresponding quantum cylinder operator is
\begin{equation} \label{PALPHA} \hat{P}_{{\bf \alpha}} = \hat{P}_{\alpha_{n-1}} (n-1) \hat{P}_{\alpha_n-2}(n-2) \cdots \hat{P}_{\alpha_0}, \end{equation}
where ${\bf \alpha} = [\alpha_0, \dots, \alpha_{n-1}]$  where
\begin{equation} \hat{P}(k) = U^{*k} \hat{P} U^{k}. \end{equation}  Here, $U = e^{i
\sqrt{\Delta}}$ is the propagator at unit time (or in the
semi-classical framework, $U= e^{i \hbar \Delta/2}$ (see \cite{A}
(1.3.4)).

The analogue of the measure of a cylinder set in quantum mechanics
is  given by the matrix element of the cylinder set operator in
the energy state, transported to $\Sigma$. We use the
semi-classical notation of \cite{A}, but it is easily converted to
homogeneous notation.

\begin{defin} \label{MUHDEF} (See \cite{A}, (1.3.4)) Let $\psi_{\hbar}$ be an eigenfunction
of $\Delta$.  Define the associated quantum ``measure" of cylinder
sets $\ccal = [\alpha_0, \dots, \alpha_{n-1}] \in \Sigma_n$ by
\begin{equation} \mu_{\hbar}([\alpha_0, \dots, \alpha_n]) =
\langle \hat{P}_{\alpha_n}(n) \cdots \hat{P}_{\alpha_0}(0)
\psi_{\hbar}, \psi_{\hbar} \rangle. \end{equation}

\end{defin}

Thus, one ``quantizes" the cylinder set $\ccal = [\alpha_0, \dots,
\alpha_{n-1}]$ as the operator
\begin{equation} \label{CCALORD} \hat{\ccal} = U^{-(n-1)} \hat{P}_{\alpha_{n-1}}
U \hat{P}_{\alpha_{n-2}} \cdots U \hat{P}_{\alpha_0}
\end{equation}  For each eigenfunction, one obtains a linear
functional
\begin{equation} \mu_{\hbar}([\alpha_0, \dots, \alpha_{n-1}]) =
\rho_{\hbar} (\hat{\ccal}) \end{equation} in the notation of
states of \S \ref{INVSTATES}. The `quantum measures' are shift
invariant, i.e.
$$\mu_{\hbar}([\alpha_0, \dots, \alpha_{n-1}]) =
\mu_{\hbar}(\sigma^{-1} [\alpha_0, \dots, \alpha_{n-1}]). $$ This
says,
$$\langle U^{-(n-2)} P_{\alpha_{n-1}} U^{n-2} \cdots U P_{\alpha_0}
U^{-1} \phi_{\hbar}, \phi_{\hbar} \rangle = \langle U^{-(n-1)}
P_{\alpha_{n-1}} U^{n-1} \cdots U P_{\alpha_0} \phi_{\hbar},
\phi_{\hbar} \rangle,$$ which is true since $\rho_h(A) = \langle A
\phi_{\hbar}, \phi_{\hbar} \rangle$ is an invariant state.

The matrix element $\langle \hat{P}_{{\bf \alpha}} \psi_{\hbar},
\psi_{\hbar} \rangle$ is the probability (amplitude)  that the
particle in state $\psi_{\hbar}$ visits the phase space regions
$P_{\alpha_0}, P_{\alpha_1}, \dots, P_{\alpha_{n-1}}$ at times $0,
1, \dots, n - 1$.

Thus, the states $\rho_{\hbar}$ or alternatively the Wigner
distributions of $\psi_{\hbar}$ are transported to $\Sigma$ to
define   linear functionals $\mu_{\hbar}$  on the span of the
cylinder functions on $\Sigma$.  They  {\it exactly} invariant
under the shift map.

 The functionals $\mu_{\hbar}$  are not positive measures
since $\hat{\ccal}_{\hbar}$ is not a positive operator. However as
in the proof of the quantum ergodicity Theorem \ref{QE}, one has
the Schwartz inequality,
\begin{equation} |\rho_{\hbar}(\hat{\ccal})|^2 \leq \rho_{\hbar}
(\hat{\ccal}^* \hat{\ccal}), \end{equation} reflecting the
positivity of the state $\rho_{\hbar}.$

\subsection{Main estimate for cylinder sets of $\Sigma$}

The so-called {\it main estimate} of \cite{A} implies that for
every cylinder set $\ccal \in \Sigma_n$, one has
 \begin{equation}\label{MAIN2}  |\mu_{\hbar}(\ccal)| \leq C_{\beta} h^{-d/2} e^{- n
 \Lambda/2}( 1 + O(\epsilon))^n, \;\;\; \mbox{uniformly for }\; n
 \leq \beta  \; |\log \hbar|. \end{equation}
 {\it Here, $\beta$ can be taken arbitrarily large}.
  This
 estimate is similar in spirit to (\ref{BTEST}) but circumvents the uncertainty
 principle.  It is one of  the essential quantum mechanical (or
semi-classical)  ingredient for getting a lower bound on
$h_{top}(supp \nu_0)$ of supports of quantum limit measures (see
\S \ref{TOPENT}).

\subsection{Quantization of cylinder sets versus quantized
cylinder operators}

Let us explain how (\ref{MAIN2}) gets around the uncertainty
principle. A fundamental feature of quantization is that it is not
a homomorphism. Thus, the quantization of the (smoothed)
characteristic function of the cylinder set $P_{\alpha_0} \cap
g^{-1} P_{\alpha_1} \cap \cdots \cap g^{-n} P_{\alpha_n}$ is very
different from $ \hat{P}_{\alpha_{n-1}} U \hat{P}_{\alpha_{n-2}}
\cdots U \hat{P}_{\alpha_0}$, the ordering which occurs in the
quantum cylinder operator $\hat{P}_{{\bf \alpha}}$ in
(\ref{PALPHA}). The problematic Bowen ball estimate (\ref{BTEST})
would arise if one quantized the Bowen ball (or cylinder set) in
the first sense  and then took its matrix element. Quantizing in
the other order makes the main estimate  possible.

\subsection{Quantized
cylinder operators and covering numbers}

As motivation to study quantum measures of cylinder sets,
 Anantharaman observes that the main estimate  suggests an obstruction to a quantum measure
concentrating on a set $F$ with $h_{top}(F) < \frac{\Lambda}{2}. $
For each $n$ such a set can be covered by $C e^{n
(\frac{\Lambda}{2} (1 - \delta)}$ cylinder sets of length $n$. By
the main estimate (\ref{MAIN2}), each of the cylinder sets has
quantum measure  $ |\mu_{\hbar}(\ccal)| \leq C e^{-d/2 \log \hbar
- n
 \Lambda/2} \;\;\; \mbox{uniformly for }\; n
 \leq \kappa \; |\log \hbar|$ for any $\kappa$. As yet, the combination does not
 disprove that supp$\nu_0 = F, $ i.e. $\nu_0(F) = 1$, since the the large factor of $h^{-d/2}$ is not
 cancelled in the resulting estimate
 $\mu_{\hbar}(F) \leq C e^{n  (\frac{\Lambda}{2} (1 - \delta) -d/2 \log \hbar - n
 \Lambda/2}.$ An additional idea is needed to decouple the large time parameter $n$ into independent
 time parameters.
 Roughly speaking, this is done by a sub-multiplicative property.
 In \S \ref{PROOF}, the ideas in the above paragraph are made more precise and effective.

\subsection{\label{TOPENT} Upper semi-continuity, sub-multiplicativity and sub-additivity}

Before getting into the details of the (sketch of) proof, it is
useful to consider why there should  exist a non-trivial lower
bound on $h_{KS}(\nu_0)$ for a quantum limit measure? What
constrains a quantum limit measure to have higher entropy than
some given invariant measure?

First, the classical entropy of an invariant is   upper
semi-continuous with respect to weak convergence.  This is in the
right direction for proving lower entropy bounds for  quantum
limits. If one can obtain a lower bound on the `entropies' of the
quantum measures $\mu_n$, then the entropy of the limit can only
jump up. Of course, the quantum measures are not $g^t$ invariant
measures and so upper semi-continuity does not literally apply.
One needs to define a useful notion of entropy for the quantum
measures and prove some version of upper semi-continuity for it.
The USC property of the classical entropy suggests that there
should exist some analogue on the quantum level.

In \cite{A}, entropies are studied via special  covers by cylinder
sets.  For these, there is a  sub-multiplicative estimate on the
number of elements in the cover.

 In
\cite{AN}, a quite different idea is used: the authors employ a
notion of quantum entropy due to Maassen-Uffink and prove a
certain lower bound for the quantum entropy called the entropy
uncertainty inequality. It is the origin of the lower bound for
the quantum entropies and the KS entropy in Theorem \ref{thethm}  (see \S \ref{EUP}).
  As a replacement for the  USC property of the entropy,
 a certain sub-additivity property  for the quantum entropy is proved (see \S
\ref{SUBA}).

In the next three subsections, we go over the main ingredients in
the proof. Then we give a quick sketch of the full proof.

\subsection{\label{ME} Main estimate in more detail}

The main estimate in \cite{A} (Theorem 1.3.3) was refined in
\cite{AN,ANK}. To state the results, we need some further
notation.  Let $n_E(\hbar)$ denote the Ehrenfest time
$$n_E(\hbar) = \left[ \frac{(1 - \delta) |\log \hbar
|}{\lambda_{\max}} \right] $$ where the brackets denote the
integer part and $\delta$ is a certain small number arising in
energy localization (for the definition of $\lambda_{\max}$ see
(\ref{LAMBDAMAX})). The authors also introduce a discrete
`coarse-grained'
 unstable Jacobian
 $$J_1^u(\alpha_0, \alpha_1) : = \sup \{J^u(\rho): \rho \in T^* \Omega_{\alpha_0} \cap \ecal^{\epsilon}: \; g^t \rho \in T^* \Omega_{\alpha_1} \}, $$
 for $\alpha_0, \alpha_1 = 1, \dots, K$. Here, $\Omega_{j}$ are small open neighborhoods of the partition sets $M_j$. For a sequence
 ${\bf \alpha} = (\alpha_0, \dots, \alpha_{n-1})$ of symbols of length $n$, one defines
 $$J_n^u({\bf \alpha}) : = J_1^u(\alpha_0, \alpha_1) \cdots J_1^u(\alpha_{n-2}, \alpha_{n-1}). $$

\begin{theo} (See Theorem 3.5 of \cite{ANK}) \label{t:main}
Given a partition $\pcal^{(0)}$ and $\delta, \delta'>0$ small
enough, there exists $\hbar_{\cP^{(0)},\delta,\delta'}$ such that,
for any $\hbar\leq \hbar_{\pcal^{(0)},\delta,\delta'}$, for any
positive integer $n\leq n_E(\hbar)$, and any pair of sequences
${\bf \alpha}$, ${\bf \alpha'}$ of length $n$,
\begin{equation}\label{emain}
\left|| \hat{P}_{{\bf \alpha}'}^*\, U^n\, \hat{P}_{{\bf \alpha}}
\Op(\chi^{(n)}) \right||
 \leq C\,\hbar^{-(d-1+c\delta)}\, \sqrt{J^u_n({\bf \alpha}) J^u_n({\bf \alpha}')}\,.
\end{equation}
Here, $d = \dim M$ and the  constants $c$, $C$ only depend on
$(M,g)$.
\end{theo}

There is a more refined version using a sharper energy cutoff in
\cite{AN}.

\begin{theo} \label{t:mainn}
Given a partition $\pcal^{(0)}$, $\kappa > 0$  and $\delta >0$
small enough, there exists $\hbar_{\cP^{(0)},\delta,\kappa}$ such
that, for any $\hbar\leq \hbar_{\pcal^{(0)},\delta,\kappa}$, for
any positive integer $n\leq \kappa |\log \hbar|$, and any pair of
sequences ${\bf \alpha}$, ${\bf \alpha'}$ of length $n$,
\begin{equation}\label{e:main3}
\left|| \hat{P}_{{\bf \alpha}'}^*\, U^n\, \hat{P}_{{\bf \alpha}}
\Op(\chi^{(n)}) \right||
 \leq C\,\hbar^{-(\frac{d-1}{2} +\delta)}\, e^{-(d - 1) n} (1 + O(\hbar^{\delta}))^n\,.
\end{equation}
Here, $d = \dim M$.
\end{theo}

To prove this, one shows  that any state of the form
$\Op(\chi^{(*)})\Psi$  can be decomposed as a superposition of
essentially $\hbar^{-\frac{(d-1)}2}$ normalized Lagrangian states,
supported on Lagrangian manifolds transverse to the stable leaves
of the flow.  The action of the operator $\hat{P}_{{\bf \alpha}}$
on such Lagrangian states is intuitively as follows: each
application of $U$ stretches the Lagrangian in the unstable
direction (the rate of elongation being described by the unstable
Jacobian) whereas each multiplication by $\hat{P}_{\alpha_j}$ cuts
off the  small piece of the Lagrangian in the $\alpha$th cell.
This iteration of stretching and cutting accounts for the
exponential decay. A somewhat more detailed exposition is in \S
\ref{MEEXP}.

This estimate can be reformulated as a statement about matrix
elements of the operators $\hat{P}_{{\bf \alpha}}$ relative to
eigenfunctions (see Theorem 3.1 of \cite{ANK})

\begin{theo}\label{MEGOOD} Let $\hat{P}_k = {\bf 1}_{M_k}^{sm}.$ For any $\kappa > 0$
there exists $\hbar_{\kappa} > 0$ so that, uniformly for $\hbar <
\hbar_{\kappa}$ and all $n \leq \kappa \; |\log \hbar|, $ and for
all ${\bf \alpha} = (\alpha_0, \dots, \alpha_{n-1}) \in [1,
\kappa]^n \cap \Z^n$,
$$||\hat{P}_{{\bf \alpha}}
\psi_h|| \leq 2 (2 \pi \hbar)^{-d/2} e^{- n \frac{\Lambda}{2}} (1
+ O(\epsilon))^n. $$
\end{theo}

The cutoff operator $Op(\chi^{(n)})$ is supported near the energy
surface $\lambda_n^2$, and the eigenfunctions are supported on
their energy surfaces.  Here, $\Lambda$ is the ``smallest"
expansion rate (see  (\ref{LAMBDAnosov})).

We recall that  $\langle \hat{P}_{{\bf \alpha}} \psi_{\hbar},
\psi_{\hbar} \rangle$ is  the probability (amplitude)  that the
particle in state $\psi_{\hbar}$ visits the phase space regions
$P_{\alpha_0}, P_{\alpha_1}, \dots, P_{\alpha_{n-1}}$ at times $0,
1, \dots, n - 1$.  The main estimates show that this probability
decays exponentially fast with $n$, at rate $\frac{\Lambda}{2}$.
However, the decay only starts near the Eherenfest time $n_1 =
\frac{d |\log \hbar|}{\Lambda}$.

\subsection{\label{EUP} Entropy uncertainty principle}

The source of the lower bound for $h_{KS}(\nu_0)$ is most simply
explained in \cite{AN}. It  is based on the   `entropic
uncertainty principle' of Maassen- Uffink. There are several
notions of quantum or non-commutative entropy, but for
applications to eigenfunctions it is important to find one with
good semi-classical properties.

 Let $(\hcal, \langle.,.\rangle)$ be a complex
Hilbert space, and let
$||\psi||=\sqrt{\langle\psi,\psi\rangle}$ denote the associated
norm. The quantum notion of partition is a family
$\pi=(\pi_k)_{k=1,\ldots,\ncal}$  of operators on $\hcal$ such
that $ \sum_{k=1}^{\hcal}\pi_{k}\pi_{k}^{*}=Id.$ If $||\psi||=1$,
the entropy of $\psi$ with respect to the partition $\pi$ is
define by
$$
h_{\pi}(\psi)=-\sum_{k=1}^{\ncal} ||\pi_k^* \psi||^2\log
||\pi_k^*\psi||^2\,.
$$

We note that the quantum analogue of an invariant probability
measure $\mu$ is an invariant state $\rho$, and the direct
analogue of the entropy of the partition would be $\sum
\rho(\pi_{k}\pi_{k}^{*}) \log \rho(\pi_{k}\pi_{k}^{*}). $ If the
state is $\rho(A) = \langle A \psi, \psi \rangle$ then $\rho(\pi_k
\pi_k^*) = ||\pi^*_k \psi||^2. $

 The dynamics is generated by a unitary
operator $\cU$ on $\cH$. We now state a simple version of the
entropy uncertainty inequality of Maasen-Uffink. A more elaborate
version in \cite{A,AN,ANK} gives a lower bound for a certain
`pressure'.

\begin{theo}
\label{t:WEUP}
 For any $\epsilon \geq 0$, for any normalized $\psi\in\cH$,
$$
h_{\pi}\big({\mathcal U} \psi \big) + h_{\pi}\big(\psi\big) \geq -
2 \log c({\mathcal U})\,,
$$
where
$$ c({\mathcal U}) = \sup_{j, k} |\langle e_k, \; {\mathcal U} e_j \rangle|$$
is the supremum of all matrix elements in the orthonormal basis
$\{e_j\}$. In particular, $h_{\pi}(\psi) \geq - \log c({\mathcal
U})$ if $\psi$ is an eigenfunction of ${\mathcal U}$.
\end{theo}

The intuitive idea is that if a unitary matrix has small entries,
then each of its eigenvectors must have large Shannon entropy.

This uncertainty principle is applied to the entropies of the
partitions defined above. For the propagator, we put  ${\mathcal
U} = e^{i T_E \sqrt{\Delta}}$ is the wave operator at the
`Ehrenfest time' $T_E = \frac{\log \lambda}{\lambda_{\max}}$. Or
in the semi-classical framework (where $h = \frac{1}{\lambda}$),
one can use the Hamiltonian is $H = h^2 \Delta$ and  the time
evolution ${\mathcal U} = e^{i n_E(h) h \Delta}$ with $n_E(h) =
\frac{\log \frac{1}{h}}{\lambda_{\max}}. $

Applied to an eigenfunction of $\ucal$, one has
\begin{equation} h_{\pi}(\psi_h) = \sum_{|\alpha| = n_E}
||\hat{P}_{\alpha}^* \psi_h||^2 \log ||\hat{P}_{\alpha}^*
\psi_h||^2,
\end{equation} and obviously
$$ h_{\pi}\big({\mathcal U} \psi \big) + h_{\pi}\big(\psi\big) = 2
h_{\pi}\big( \psi \big). $$ On the right side,
\begin{equation} \label{c} c({\mathcal U}) = \max_{|\alpha| = |\alpha'| = n_E} ||\hat{P}_{\alpha'} U^{n_E} \hat{P}_{\alpha} Op(\chi^{(n_E)})|| \end{equation}
where $\chi^{(n_E)}$ is a very sharp energy cutoff supported in a
tubular neighborhood $\ecal^{\epsilon} := H^{-1}(1 - \epsilon, 1 +
\epsilon)$
  of $\ecal = S^*M$ of width
$2 h^{1 - \delta} e^{n \delta}$ for a given $\delta> 0$.

\subsection{\label{SUBA} Sub-additivity}

Another key component in the proof is that the quantum entropy is
almost sub-additive for $\hbar \leq |\log t|. $

Sub-additivity of the classical KS entropy follows from the
concavity of the function $- x \log x$. It states that the
sequence $\{h_n(\mu, P)\}$ is sub-additive, i.e.
$$h_{n + m} (\mu, P) \leq h_n(\mu, P) + h_m(\mu, P). $$
However, this is false for the quantum measures.

The correct statement is as follows: There exists a function
$R(n_0, \hbar)$ such that $\lim_{\hbar \to 0} R(n_0, \hbar)$ for
all $n_0 \in \Z$ and such that for all $n_0, n \in {\bf N}$ with
$n_0 + n = \frac{(1 - \delta) |\log \hbar|}{\lambda_{\max}}, $ and
for any normalized eigenfunction $\phi_{\lambda}$, one has
\begin{equation} \label{SUBADD} h_{\pcal^{n_0 + n}}(\phi_{\lambda}) \leq
h_{\pcal^{n_0}}(\phi_{\lambda}) + h_{\pcal^{ n}}(\phi_{\lambda}) +
R(n_0, \hbar). \end{equation}

\subsection{Outline of proof of Theorem \ref{thethm}}

We outline the proof from \cite{AN}. Let
\begin{equation} h_{\pcal^{(n)}}(\psi) = - \sum_{|\alpha| = n}
||\hat{P}^*_{\alpha} \psi||^2 \log ||\hat{P}^*_{\alpha} \psi||^2,
\end{equation}
and let
\begin{equation} h_{\tcal^{(n)}}(\psi) = - \sum_{|\alpha| = n}
||\hat{P}_{\alpha} \psi||^2 \log ||\hat{P}
_{\alpha} \psi||^2,
\end{equation}

\begin{enumerate}

\item Let $K = \frac{1 - \delta}{\lambda_{\max}}. $  By the main
estimate and the entropy uncertainty inequality, one has
$$h_{\tcal^{(n)}}(\psi_h) + h_{\pcal^{(n)}}(\psi_h) \geq 2 (d-1) n +
\frac{(d - 1 + 2 \delta)\lambda_{\max}}{(1 - \delta)} n_E(\hbar)+
O(1).
$$

\item Fix $n_0, n_E(\hbar)$. Write $n = q n_0 + r, r \leq n_0$.
Then by sub-additivity,
$$\frac{h_{\pcal^{(n)}}(\psi_h)}{n} \leq
\frac{h_{\pcal^{(n_0)}}(\psi_h)}{n_0} +
\frac{h_{\pcal^{(r)}}(\psi_h)}{n_E(\hbar)} + \frac{R(n_0,
\hbar)}{n_0}. $$

\item As $\hbar \to 0$, this gives
$$\frac{h_{\pcal^{(n_0)}}(\psi_h)}{n_0} + \frac{h_{\tcal^{(n_0)}}(\psi_h)}{n_0}
\geq 2 (d-1) n + \frac{(d - 1 + 2 \delta)\lambda_{\max}}{(1 -
\delta)} n + O(1) - \frac{R(n_0, \hbar)}{2 n_0}
 + O_{n_0}(1/n). $$

\item Now take the limit $\hbar \to 0$. The expressions
$||\hat{P}_{\alpha}\phi_{\lambda}||^2$ and the Shannon entropies
tend to their classical values. Hence, $h_{\pcal^{(n_0)}}(\psi_h)$
tends to the KS entropy of the quantum limit. Thus, the lower
bound implies
$$\frac{h_{n_0}(\mu)}{n_0} \geq d - 1 - \frac{(d - 1 + 2 \delta)
\lambda_{\max}}{2 (1 - \delta)}. $$ Since $\delta, \delta'$ are
arbitrary one can set them equal to zero. Then let $n_0 \to
\infty$ to obtain the KS entropy of the quantum limit.

\end{enumerate}

\subsection{\label{PROOF} Sketch of proof of Theorem \ref{A}  \cite{A}}

The proof of Theorem \ref{A} is less structural than the proof of
Theorem \ref{thethm}. There is no magic weapon like the entropic
uncertainty principle. In its place there is the construction of
special covers by cylinder sets and counting arguments for the
number of elements in the cover. These covering arguments probably
have further applications and (although quantum) seem  more
geometric than the argument based on the entropic uncertainty
principle.

We mainly discuss the statement that $h_{top}(\mbox{supp}\; \nu_0)
\geq \frac{\Lambda}{2}.$  We recall (\S \ref{TOPENTa})  that an
invariant set $F $ is with $h_{top}(F) \leq \frac{\Lambda}{2}(1 -
\delta)$ can be covered by at $C e^{N (\frac{\Lambda}{2} (1 -
\delta)}$ cylinder sets of length $N$. We further recall that, by
the main estimate (\ref{MAIN2}), each of the cylinder sets has
quantum measure  $ |\mu_{\hbar}(\ccal)| \leq C e^{-d/2 \log \hbar
- N
 \Lambda/2} \;\;\; \mbox{uniformly for }\; N
 \leq \vartheta \; |\log \hbar|$ for any $\vartheta$.   We want to combine
 these estimates to show that
 $\nu_0(F) < 1$ for any quantum limit $\nu_0$. But we cannot simply
 multiply the measure
 estimate for  $\ccal $ in the cover times the number of cylinder sets in the cover. The main estimate
 being fixed, the only things left to work with are the covers.

 As will be seen below, one needs to introduce special covers
 adapted to the eigenfunctions. The covering number of such covers has
 a sub-multiplicative property.  We motivate it by going through a
 heuristic proof from \cite{A}. We also add a new step that helps
 explain why special covers are needed. See also \S
 \ref{FGAMMA} where the pitfalls are explored when $F = \gamma$
 (a closed hyperbolic orbit).

\subsubsection{Covers and times, I}

Let $F \subset \Sigma$ be a shift-invariant subset with
$h_{top}(F) \leq \frac{\Lambda}{2} (1 - \delta)$.  Let $W_n
\subset \Sigma_n$ be a  cover of minimal cardinality of $F$ by $n$
cylinder-sets.  Now introduce a new time $N >> n$ and define
\begin{equation} \label{SIGMAN}   \Sigma_N(W_n, \tau) : = \left\{ \begin{array}{l} \{N-
\mbox{cylinders}\; [\alpha_0, \dots, \alpha_{N-1}] \; \mbox{ such
that} \\ \\
\frac{\# \{j \in [0, N - n]: [\alpha_j, \dots, \alpha_{j + n -1}]
\in W_n\}}{N - n + 1} \geq \tau. \end{array}\right.
\end{equation} They correspond to orbits that spend a proportion
$\geq \tau$ of their time in $W_n$. In terms of the (ergodic)
Liouville measure (transported to $\Sigma$), for almost every
point of $S^*M$ (or its image under $I$), the proportion of the
total time its orbit spends in $W_n$ is the relative measure of
(the union of the cylinder sets in)  $W_n$. If this measure is
smaller than $\tau$, then $\Sigma_N(W_n, \tau)$ is a rare event as
$N \to \infty$.  Its cylinders satisfy $\sigma^j(\ccal) \cap W_n =
\ccal$ for a proportion $\tau$ of $0 \leq j \leq N$. If one
imagines partitioning the elements of $W_n$ into $N$ cylinders,
then only a (relatively) sparse subset will go into $\Sigma_N(W_n,
\tau)$. The next Lemma gives an upper bound on that number:

\begin{lem} \label{FCOVER} (see Lemma 2.3.1 of \cite{A})  For all $n_0$ there exist $n \geq n_0$
and $N_0$ so that for $N \geq N_0$ and $\tau \in [0, 1]$, the
  cover $W_{n}$ of
$F$ of minimal cardinality satisfies,
$$\begin{array}{lll} \# \Sigma_N(W_{n}, \tau) & \leq &  e^{N \frac{3 \Lambda \delta}{8}} e^{N h_{top}(F) }  e^{(1 - \tau) N(1 + n) \log \ell}.
\end{array}
$$ where $\ell$ is the number of elements in the partition.
\end{lem}

If $\tau$ is chosen sufficiently close to $1$, this implies:

\begin{cor} \label{FCOVERb}  (See (2.3.1)  of \cite{A}) If $F$ is an invariant
set with  $h_{top}(F) < \frac{\Lambda}{2}(1 - \delta)$, then
$$\# \Sigma_N(W_{n}, \tau) \leq e^{N (\frac{\Lambda}{2} (1 -
\delta))} e^{(1 - \tau) N(1 + n) \log \ell}. $$  \end{cor}

We now re-write this estimate in terms of the semi-classical
parameter $\hbar$. If $\tau$ is sufficiently close to $1$ and
$\varepsilon
> 0$ is sufficiently small, there exists   $\vartheta
> 0 $ so that for $N \geq \vartheta |\log \hbar|,$
\begin{equation} \label{VARBOUND} e^{N (\frac{\Lambda}{2} (1 -
\delta))} e^{(1 - \tau) N(1 + n) \log \ell} <  (1 - \varepsilon)
h^{d/2} e^{N \Lambda/2}. \end{equation} Indeed,  let  $\alpha = -
\delta \frac{\Lambda}{2} + (1 - \tau) (1 + n) \log \ell$. For
$\tau$ very close to $1$ it is negative. We then  choose
$\vartheta$ so that $h^{- \alpha \vartheta}  < (1 - \varepsilon)
h^{d/2}$, i.e. so that $| \alpha \vartheta| > d/2. $

By the main estimate (\ref{MAIN2}) and (\ref{VARBOUND}), we obtain
the following

\begin{cor}\label{COR} (see \cite{A}, (2.0.1))

For $\tau$ sufficiently close to $1$ and   $\varepsilon $ close to
zero, there exists $\vartheta > 0$ so that for  $N \geq \vartheta
|\log \hbar|,$
\begin{equation}\label{USED}  |\mu_{\hbar} (\Sigma_N(W_n, \tau))| \leq 1 -
\varepsilon. \end{equation}

\end{cor}

 Here,   $\mu_h (\Sigma_N(W_n, \tau)) $
 means $\mu_h (\bigcup_{\ccal \in
\Sigma_N(W_n, \tau))} \ccal)$.

\subsubsection{Heuristic Proof}

We now present a heuristic proof of Theorem \ref{A} from  Section
2 of \cite{A}  that explains the need for some rather technical
machinery introduced below.

We recall that $\mu_{\hbar}$ is a shift  ($\sigma$)-invariant
measure on $\Sigma$. Let us temporarily pretend that $\mu_{\hbar}$
is also a positive measure. Then we would have for $N \geq
\vartheta |\log \hbar|,$

\begin{equation} \label{BIGa} \begin{array}{lll} |\mu_{\hbar}(W_n)|
&=  &  |\frac{1}{N - n} \sum_{k = 0}^{N - n - 1}
\mu_{\hbar} (\sigma^{-k} W_n)|  \\ && \\
& = & \left| \mu_{\hbar} \left(\frac{1}{N - n} \sum_{k = 0}^{N - n
- 1} {\bf 1}_{\sigma^{-k} W_n} \right) \right|
\\ && \\ &\leq & \sum_{\ccal \in \Sigma_N(W_n, \tau)}
\mu_{\hbar}(\ccal) + \tau \sum_{\ccal \notin \Sigma_N(W_n, \tau)}
\mu_{\hbar}(\ccal)  \\ && \\ & = & \mu_{\hbar}( \Sigma_N(W_n,
\tau)) + \tau \mu_{\hbar}( \Sigma_N(W_n,
\tau)^c)  \\ && \\
& \leq & (1 - \tau) \left(1 - \varepsilon \right) + \tau.
\end{array} \end{equation}
Here we used (\ref{USED}), and that the sum of the $\mu_h$
measures of $ \Sigma_N(W_n, \tau)$ and its complement add to one.
Since $W_n$ is fixed, the weak convergence $\mu_{\hbar} \to \mu_0$
implies,
\begin{equation}\begin{array}{lll}  |\mu_0(W_n)| & \leq &(1 - \tau) \left(1 - \varepsilon) \right) + \tau < 1. \end{array}
\end{equation} Since $F \subset W_n$, the same estimate
applies to $F$. So $F$ cannot be the support of $\mu_0$. ``QED"

 In the above steps,  we use that $\sigma^{-k}
W_n \subset \Sigma_N(W_n, \tau)$ for $k \leq N - n - 1$ and that
\begin{equation} \label{FACTS} \frac{1}{N - n} \sum_{k = 0}^{N - n - 1} {\bf 1}_{\sigma^{-k}
W_n} \leq 1, \;\;\; \frac{1}{N - n} \sum_{k = 0}^{N - n - 1} {\bf
1}_{\sigma^{-k} W_n}  \leq \tau \;\; \mbox{on}\; \Sigma_N(W_n,
\tau)^c. \end{equation}

\begin{rem} There are two very innovative ideas in this argument.
It is somewhat reminiscent of the averaging argument of Theorem
\ref{QE}, but there is no averaging over eigenvalues and no appeal
to ergodicity (since there is no reason why a weak* limit of
$\mu_{\hbar}$ should be an ergodic measure).  The first idea is to
study the special observables  $\hat{P}_{{\bf \alpha}}$. They are
not expressed in the form $Op_{\hbar}(a)$, i.e. in terms of a
symbol, and it would be counter-productive to do so. Thus, the
argument of (\ref{BIG})  takes place entirely on the quantum level
and does not make use of any properties of the weak* limit (which
are unknown); the semi-classical limit is taken after the
inequality is proved.  The key innovation is the third step, where
the sets $\Sigma_N(W_n, \tau)$ are introduced and the sum is
broken up into one over $\Sigma_N(W_n, \tau)$ and one over its
complement. As is evident from (\ref{FACTS}), $\Sigma_N(W_n,
\tau)^c$ arises naturally as an approximation to the
$\tau$-sublevel set of $\frac{1}{N - n} \sum_{k = 0}^{N - n - 1}
{\bf 1}_{\sigma^{-k} W_n}$. But note again that the the
decomposition occurs in $\Sigma_N$ and is used the on the quantum
level. This is very different from  using symbols of the $\hat{C}$
and their time averages to to define the corresponding  subsets in
$S^*M$!
\end{rem}

The flaw in the argument is in the third line, where we pretended
that $\mu_{\hbar}$ was a positive measure.  We recall that
$\mu_{\hbar}(\ccal)$ converges to $\mu_0(\ccal)$, which is a
positive measure. But this is for a fixed $\ccal$. But it is also
close to being a probability measure for cylinder sets of length
$\leq \kappa |\log \hbar|$. But (\ref{USED}) is only valid for $N
\geq \vartheta |\log \hbar|$. Since $\vartheta > \kappa$, one
cannot use them both simultaneously (see \S \ref{REMARK} for
$\kappa, \vartheta$)

One can attempt to get around this problem using the positivity of
the underlying state $\rho_{\hbar}$. Define the function
$$W_n^N : = \frac{1}{N - n} \sum_{k = 0}^{N - n
- 1} {\bf 1}_{\sigma^{-k} W_n}: \Sigma \to \R_+. $$ Then in the
third line we rigorously have,
\begin{equation} | \mu_{\hbar}(W_n^N )| \leq
\rho_{\hbar}\left( (\hat{W}_n^N)^* (\hat{W}_n^N)\right)^{1/2} = ||
\hat{W}_n^N  Op_{\hbar} \psi_{\hbar}||.
\end{equation}
Here, $\hat{W}_n^N$ is the quantization as a cylinder set
operator.  We further denote by $\widehat{\Sigma_N(W_n, \tau)}$
the operator equal to the sum of $\hat{\ccal}$ for $\ccal \in
\Sigma_N(W_n, \tau)$. Then
$$\widehat{\Sigma_N(W_n, \tau)} + \widehat{\Sigma_N(W_n, \tau)^c}
= I. $$ So (with $Op_{\hbar}(\chi)$ the energy cutoff),
\begin{equation} ||
\hat{W}_n^N  Op_{\hbar}(\chi) \psi_{\hbar}|| \leq || \hat{W}_n^N
\widehat{\Sigma_N(W_n, \tau)} Op_{\hbar}(\chi) \psi_{\hbar}|| + ||
\hat{W}_n^N \widehat{\Sigma_N(W_n, \tau)^c} Op_{\hbar}(\chi)
\psi_{\hbar}||. \end{equation} By (\ref{FACTS}),
$$\left\{ \begin{array}{l} ||
\hat{W}_n^N \widehat{\Sigma_N(W_n, \tau)} Op_{\hbar}(\chi)
\psi_{\hbar}|| \leq ||  \widehat{\Sigma_N(W_n, \tau)}
Op_{\hbar}(\chi) \psi_{\hbar}|| \\ \\ || \hat{W}_n^N
\widehat{\Sigma_N(W_n, \tau)^c} Op_{\hbar} (\chi) \psi_{\hbar}||
\leq \tau || \widehat{\Sigma_N(W_n, \tau)^c} Op_{\hbar}(\chi)
\psi_{\hbar}||.
\end{array}. \right.$$

To complete the proof of the final equality, need an estimate of
the form,
\begin{equation} ||  \widehat{\Sigma_N(W_n, \tau)}
Op_{\hbar}(\chi)
\psi_{\hbar}||  \leq 1 - \theta.  \end{equation} We also need
\begin{equation} ||
\widehat{\Sigma_N(W_n, \tau)^c} Op_{\hbar} \psi_{\hbar}|| \leq 1 -
|| \widehat{\Sigma_N(W_n, \tau)} Op_{\hbar} \psi_{\hbar}|| +
O(\hbar).
\end{equation}
The latter holds formally since $\widehat{\Sigma_N(W_n, \tau)},
\widehat{\Sigma_N(W_n, \tau)^c}$ are semi-classically
complementary projections. But the time $N \geq \vartheta |\log
\hbar|$ may be   too large for this semi-classical estimate.

The  heuristic  arguments are sufficiently convincing to motivate
a difficult technical detour in which a  certain
sub-multiplicativity theorem is used to reduce the $N$ in the
argument from $N \geq \vartheta |\log \hbar|$ (above the Ehrenfest
time) to $N \leq \kappa |\log \hbar|$ (below the Ehrenfest time).
It is the analogue in \cite{A} of the sub-additivity step in
\cite{AN}, and its purpose is to allow the argument above to use
only cylinder sets of length $\kappa |\log \hbar|$ (or balls of
radius $>> \hbar^{1/2}$) where semi-classical analysis is
possible. But one has to make sure that the covering number
estimates in  Lemma \ref{FCOVER} and (\ref{VARBOUND}) do not break
down. One has to choose special covers so that  there are not too
many elements in the cover of $F$.  This step is another crucial
innovation in \cite{A}, and in some ways the subtlest one.

\subsubsection{$(\hbar, 1- \theta, n)$-covers}

The sub-multiplicativity step requires a new concept and a new
parameter $\theta$. Fix $(\hbar, n, \theta)$ and consider a subset
$W \subset \Sigma_n$.

\begin{defin} Say that $W$ is an $(\hbar, 1 - \theta, n)-\; \mbox{cover of
} \;\; \Sigma$ if
\begin{equation}\label{CCALSUM1}  || \sum_{\ccal \in W^c} \hat{\ccal}_{\hbar}
Op_{\hbar}(\chi) \psi_{\hbar} || \leq \theta. \end{equation}
\end{defin}
Here, $W^c$ is the complement of $W$; $W$ is an $(\hbar, 1 -
\theta, n)$ cover of $\Sigma$ if in this quantum sense the measure
of $W^c$ is $< 1$ (it is called ``small" in \cite{A} but in the
end small means $< 1$; See the Remark at the end of \S \ref{FINAL}
).  In other words, the union of the cylinder sets of $W$ might
not cover all of $\Sigma$ (think $S^*M$), but the mass of the
eigenfunctions $\psi_{\hbar}$ outside of the covered  part is
rather small. In this quantum sense, it is intuitively a union of
cylinder sets whose measure is $\geq 1 - \theta$. By definition,
if $W$ is an $(\hbar, 1 - \theta, n)-$ cover, then it is  an
$(\hbar, 1 - \theta', n)-$ cover for any $\theta' > \theta$.

 The covering number of such a cover
is defined by
\begin{equation} \label{NhDEF} N_{\hbar}(n, \theta) = \#\{W: \;\; W\; \mbox{is a}\; (\hbar, 1 - \theta, n)-\; \mbox{cover
of } \;\; \Sigma\}. \end{equation}  The sub-multiplicative
property is

\begin{lem} \label{SUBM} (sub-multiplicativity)
\begin{equation} N_{\hbar}\left(kn, k \theta (1 + O(n \hbar^{\alpha}))
\right) \leq N_{\hbar}(n, \theta)^k. \end{equation}
\end{lem}

We will not attempt to explain this inequality but rather will
point out its consequences.  One is a lower bound on growth of
counting numbers of minimal $(\hbar, 1- \theta, N)$-covers  as $N$
grows (Lemma 2.2.6 of \cite{A}):

\begin{lem} \label{(2.2.6)} Given $\delta > 0$, there exists
$\vartheta$ so that:

\begin{itemize}

\item (i) For $N = \vartheta |\log \hbar|$, $N_{\hbar}(N, \theta)
\geq (1 - \theta) e^{N \frac{\Lambda}{2}(1 - \frac{\delta}{16})}$;

\item (ii)  For $N = \kappa|\log \hbar|, \; \kappa \leq
\vartheta$, $N_{\hbar}(N, \frac{\kappa}{\vartheta}\; \theta) \geq
(1 - \theta)^{\frac{\kappa}{\vartheta}}  e^{N \frac{\Lambda}{2}(1
- \frac{\delta}{16})}$;

\end{itemize}

\end{lem}
Estimate (ii) follows from (i) and sub-multiplicativity. The time
in (ii) is now below the Ehrenfest time. The sub-multiplicative
property then has the following consequence:

\begin{lem} \label{CCALSUM2} (See \cite{A}, (2.3.2); compare (\ref{CCALSUM1})) Let $\kappa, \vartheta$ have the meanings in
\S \ref{REMARK}. Let $W_n \subset \Sigma_n$ be  a cover of $F$ by
$n$-cylinders of  minimal cardinality;  and let $N = \kappa |\log
\hbar|$. Then for any $\theta < 1$,
  $$ || \sum_{\ccal \in \Sigma_N(W_n, \tau)^c} \hat{\ccal}_{\hbar}
Op_{\hbar}(\chi) \psi_{\hbar} || \geq  \frac{\kappa}{\vartheta}
\theta.
$$
\end{lem}

\begin{proof} (Sketch)
The inequality goes in the opposite direction from that in the
definition of $(\hbar, 1 - \theta, N)$-covers in (\ref{CCALSUM}).
Hence the Lemma is equivalent to saying   that $\Sigma_N(W, \tau)$
is not a $(\hbar, 1 - \frac{\kappa}{\vartheta} \theta, N)$-cover.
This follows by comparing the upper bound  in  Lemma \ref{FCOVER}
with the lower bound in  \ref{(2.2.6)}. If $ \Sigma_N(W_n, \tau)$
were such a cover, one would have $\frac{\kappa}{\vartheta} |\log
(1 - \theta)| \geq N \frac{N \Lambda}{2}$, which can only occur
for finitely many $N$ if $\theta < 1$.  The cardinality of
$\Sigma_N(W, \tau)$ is too small to be $(\hbar, 1 - \theta,
N)$-cover when $N = \kappa |\log \hbar|$.
\end{proof}

The lower bound is a key step in the proof of Theorem \ref{A}.
Intuitively, it gives a lower bound on eigenfunction mass outside
the set of orbits which spend a proportion $\tau$ of their time in
$W_n$.  The lower bound on the mass outside $\Sigma_N(W_n, \tau)$
eventually prohibits the concentration of the quantum limit
measure on $F$.

Using that $\hat{C}_{\hbar} Op_{\hbar}(\chi)$ are roughly
orthogonal projectors with orthogonal images for distinct
cylinders $\ccal$, this leads to:

\begin{lem}\label{CCALSUM} (see \cite{A}, (2.4.5)) Continue with the notation of Lemma \ref{CCALSUM2}.  For $N = \kappa
|\log \hbar|, $ we have
$$\left\{ \begin{array}{l} \sum_{\ccal \in \Sigma_N} |\mu_{\hbar}(\ccal)| = 1 + O(\hbar^{\delta}), \\ \\
 \sum_{\ccal \in \Sigma_N(W_n, \tau)^c}
|\mu_{\hbar}(\ccal)| \geq \left( \frac{\kappa}{\vartheta} \theta \right)^2 + O(\hbar^{\kappa}), \\ \\
\implies \; \sum_{\ccal \in \Sigma_N(W_n, \tau)}
|\mu_{\hbar}(\ccal)| \leq 1 - \left( \frac{\kappa}{\vartheta}
\theta \right)^2 + O(\hbar^{\kappa}). \end{array} \right.$$
\end{lem}

The  second inequality of Lemma \ref{CCALSUM} again goes in the
opposite direction to the defining inequality (\ref{CCALSUM1}) of
a $(\hbar, 1- \theta, n)$ cover.

\subsubsection{\label{FINAL} Completion of Proof}

We now sketch the proof of the Theorem (see \cite{A}, (2.5.1) -
(2.5.5)). It is similar to (\ref{BIGa}) but circumvents the
positivity problem by decreasing $N$ to below the Ehrenfest time.

\begin{proof} (Outline, quoted almost verbatim from \cite{A})

As above, the  quantum measure $\mu_{\hbar}$ is  shift invariant
(see \cite{A}, Proposition 1.3.1 (ii).)

We now fix $n$ and let $N = \kappa |\log \hbar|$. The crucial
point is that this is less than Ehrenfest time, unlike the choice
of $N$ in the first attempt.  As above, $W_n \subset \Sigma_n$
denotes an $(\hbar, 1 - \theta, n)$- cover of minimal cardinality
of $F$ by $n$-cylinders. In the following, we use the notation
\begin{equation} \mu_{\hbar} (W_n): = \sum_{\ccal \in W_n }
\mu_{\hbar}(\ccal), \;\;\;\; \mu_{\hbar} (\Sigma_N(W_n, \tau)): =
\sum_{\ccal \in \Sigma_N(W_n, \tau)} \mu_{\hbar}(\ccal).
\end{equation} We then run through the steps of (\ref{BIGa})
again;

\begin{equation} \label{BIG} \begin{array}{lll} |\mu_{\hbar}(W_n)|
&=  &  |\frac{1}{N - n} \sum_{k = 0}^{N - n - 1}
\mu_{\hbar} (\sigma^{-k} W_n)|  \\ && \\
& = & \left| \mu_{\hbar} \left(\frac{1}{N - n} \sum_{k = 0}^{N - n
- 1} {\bf 1}_{\sigma^{-k} W_n} \right) \right|
\\ && \\ &\leq & \sum_{\ccal \in \Sigma_N(W_n, \tau)}
|\mu_{\hbar}(\ccal)| + \tau \sum_{\ccal \notin \Sigma_N(W_n,
\tau)} |\mu_{\hbar}(\ccal)|  \\ && \\ & \leq & \sum_{C \in
\Sigma_N(W_n, \tau) } |\mu_{\hbar}(\ccal)|
+ \tau  + O(h^{\kappa}) \\ && \\
& \leq & (1 - \tau) \left(1 - (\frac{\kappa}{\vartheta} \theta)^2
\right) + \tau  + O(h^{\kappa}).
\end{array} \end{equation}

Above,
 Lemma \ref{CCALSUM}  is used to circumvent the positivity problem in (\ref{BIGa}).
 We also used (again)   (\ref{FACTS}) and    the main estimate (\ref{USED}) to
bound the $N$-cylinder sets in $\Sigma_N(W_n, \tau)$ with $N =
\kappa |\log \hbar|$, but with the time $n$ fixed.

Now let $\hbar \to 0$. Since $W_n$ is fixed, the weak convergence
$\mu_{\hbar} \to \mu_0$ implies,
\begin{equation}\begin{array}{lll}  \mu_0(W_n) & \leq &(1 - \tau) \left(1 - (\frac{\kappa}{\vartheta} \theta)^2 \right) + \tau. \end{array}
\end{equation} Since $F \subset W_n$, the same estimate
applies to $F$. Since it holds for all $\theta < 1$, it follows
that
\begin{equation} \label{MUF} \mu_0(F) \leq (1 - \tau) \left(1 - (\frac{\kappa}{\vartheta})^2 \right) + \tau = 1 -
(\frac{\kappa}{\vartheta})^2 (1 -\tau) < 1. \end{equation} This
completes the proof that supp $\nu_0 \not= F$.  QED

To obtain the measure statement of Theorem \ref{A}, let
\begin{equation} I_{\delta} = \{x \in S^*M: h_g(\mu_0^x) \leq
\frac{\Lambda}{2}(1 - \delta) \}. \end{equation} It follows from
the Shannon-McMillan theorem that for all $\alpha > 0$, there
exists $I^{\alpha}_{\delta} \subset I_{\delta}$ with
$\mu_0(I_{\delta} \backslash I_{\delta}^{\alpha}) \leq \alpha$
such that $I_{\delta}^{\alpha}$ can be covered by $e^{n
(\frac{\Lambda}{2}(1 - \delta + \alpha)} $ n-cylinders for
sufficiently large $n$. It follows from (\ref{MUF}) that
\begin{equation} \mu_0(I_{\delta}^{\alpha}) \leq
(1 - \tau (\delta - \alpha) )  \left(1 - (\frac{\kappa}{\vartheta
(\delta - \alpha)})^2_+ \right) + \tau (\delta - \alpha).
\end{equation} Let $\alpha \to 0$ and one has,
\begin{equation} \nu_0 (S^*M \backslash I_{\delta}) \geq (1 -
\tau(\delta)) (\frac{\kappa}{\vartheta (\delta)})^2.
\end{equation}

\end{proof}

\subsection{\label{FGAMMA} {\bf $F = \gamma$}}

To help digest the proof, let us run through the proof in  the
simplest case where $F$ is a single closed hyperbolic geodesic
$\gamma$ and $M$ is a compact hyperbolic manifold  of constant
curvature $-1$. We would like to show that $\nu_0 \not=
\mu_{\gamma}$, or equivalently that $\nu_0(\gamma) < 1$. What we
need to show is that, if $\nu_0(F) = 1$, then  it requires a lot
of cylinder sets to cover $F$. Hence we try to get a contradiction
from the fact that only a few cylinder sets are needed to cover
$\gamma$. In fact,
 for each $n$, $\gamma$ may be covered by one
$n$-cylinder $$W^0_n(\gamma) = \{[\alpha_{0}, \dots, \alpha_0,
\dots, \alpha_{n}]\}, $$ namely the one specified by the indices
of the cells $P_{\alpha}$  that $\gamma$ passes through starting
from some fixed $x_0 \in \gamma$. To obtain a tube around the
orbit, we  take the union $W_n(\gamma) := \bigcup_{j = 0}^{n-1}
g^{-j} W^0_n(\gamma)$.  We could use the two-sided symbols, going
symmetrically forward and backward in time to obtain a transverse
ball of radius $e^{-n/2}$.

First, let us try a naive heuristic argument. The main   estimate
(\ref{MAIN2} )  gives
\begin{equation}\label{MAIN3} |\mu_{\hbar}(W_n(\gamma))| \leq C h^{-d/2}
e^{- n
 d /2}( 1 + O(\epsilon))^n, \;\;\; \mbox{uniformly for }\; n
 \leq K  \; |\log \hbar|, \end{equation}
 for any $K > 0$.
It is now tempting to put $n = (1 +\epsilon)  |\log \hbar|$ with
$\epsilon
> 0$  to obtain $|\mu_{\hbar}(W_{n(\hbar, \theta)})| \leq \hbar^{\theta}$.
But we cannot take the semi-classical limit of this estimate. If
we did so formally (not rigorously), it would seem to imply that
$\nu_0(\gamma) = 0$, which is stronger than the rigorous result
$\nu_0(\gamma) < 1$ and possibly false in some examples (if weak
scarring in fact occurs).

The problem is that the classical  set $W_{n(\hbar, \theta)}$ is a
kind of $\hbar$-dependent  Bowen ball or tube around $\gamma$ of
radius $e^{- \half n(\hbar, \theta)} = \hbar^{1/2 + \theta}$. The
uncertainty principle prohibits microlocalization to tubes of
radius $< h^{\half}$ and (as discussed in \S \ref{HYP}), it is
difficult even to work with neighborhoods smaller than $ h^{\half
- \delta}$ or at best $ h^{\half} |\log \hbar|. $ In particular,
we do do not have control over the  weak limits of
$|\mu_{\hbar}(W_{n(\hbar, \theta)})| $, and there is no obvious
reason why it should have semi-classical asymptotics. It is not a
positive measure on this length scale, so  we do not know if it
implies $\nu_0(\gamma) = 0$. On the other hand, if we decrease $n$
below the Ehrenfest time, so that
  the matrix element has a classical limit,  then
the bound is trivial.

The strategy of \cite{A} is to use the sub-multiplicative estimate
for  $(\hbar, 1 - \theta, N)$-covers of $\gamma$ of minimal
cardinality to decrease the length $N$ below the Ehrenfest time
(see Lemma \ref{(2.2.6)}). We first consider the cover $W_n :=
W_n(\gamma)$, which is a minimal cover of $\gamma$ by cylinders of
length $n$. We then form $\Sigma_N(W_n, \tau)$.  At least
intuitively, $W_n$ corresponds to a Bowen ball of radius
$e^{-n/2}$, i.e. a transverse cube of this radius in the stable
and unstable directions. The cylinders of $\Sigma_N(W_n, \tau)$
satisfy $\sigma^j(\ccal) \cap W_n = \ccal$ for a proportion $\tau$
of $0 \leq j \leq N$. This carves out a very porous set  in the
stable/unstable cross section, suggesting why the lower bound in
Lemma \ref{CCALSUM} should be true.

It would be interesting to connect the argument of Theorem \ref{A}
in this case to the problem in \S \ref{HYP} (see \S
\ref{MASSGAMMA})  of estimating the eigenfunction mass inside a
shrinking tubular neighborhood of $\gamma$, i.e. to give an upper
bound for the  mass $(Op_{\hbar}(\chi_{1}^{\delta})
\phi_{\mu},\phi_{\mu)})$ in the shrinking tube around $\gamma$ or
(better) the mass $(Op_{\hbar}(\chi_{2}^{\delta})\phi_{\mu},
\phi_{\mu})$ in a shrinking phase space tube. However, the methods
of \cite{A} apply equally to the cat map setting and in that case,
there are eigenfunctions whose mass in shrinking tubes around a
hyperbolic fixed point have the same order of magnitude as in the
integrable case up to a factor of $1/2$ \cite{FNB}. So unless one
can use  a property of $\Delta$-eigenfunctions which is not shared
by quantum cat map eigenfunctions, one cannot expect to prove a
sharper upper bound for eigenfunction mass in shrinking tubes
around hyperbolic orbits in the Anosov case than one has in the
integrable case. 

\subsection{\label{MEEXP} Some ideas on the main estimate}

The main estimate is independently useful. For instance,  it has
been used to obtain bounds on decay rates on quantum scattering
\cite{NZ}. In this section, we sketch some ideas of the proof. We
do not go so far as to explain a key point, that the estimate is
valid for times $t \leq \kappa |\log \hbar|$ for any $\kappa$.

The goal is to estimate the norm of the  operators $P_{{\bf
\alpha}'}^*\, U^n\, P_{{\bf \alpha}} \Op(\chi^{(n)}) $ in terms of
$\hbar$ and $n$. Since $U(n)$ is unitary it suffices to estimate
the norm of  the cylinder set operators
$$K_n : = \hat{P}_{\alpha_n} U \hat{P}_{\alpha_n -1} \cdots U
\hat{P}_{\alpha_0} Op(\chi). $$ Let $K_n(z,w)$ be the Schwartz
kernel of this operator. By  the Schur inequality, the operator
norm is bounded above by $$\sup_z \int_M |K_n(z,w)| dV(w).
$$
The inequality only applies when  $K_n(z,w) \in L^1(M)$, but by
harmless smoothing approximation we may assume this.

We may get an intuitive idea of the exponential decay of the norm
of the cylinder set operators as follows. The wave group (or the
semi-classical Schr\"odinger propagator) $U^t$ on a manifold
without conjugate points can be lifted to the universal cover (see
(\ref{UTGAMMA}))  where it has a global Hadamard parametrix
(\ref{HD}).  This parametrix shows that the wave kernel on the
universal cover  is closely related to the spherical means
operator $L_t f(x) = \int_{S^*_x M} f(g^t(x, \omega)
d\mu_x(\omega),$ where $d\mu_x$ is the Euclidean surface measure
on $S^*_x M$ induced by the metric. Both operators have the same
wave front relation, and indeed the wave front set of $U(t, \cdot,
y)$ is precisely the set of (co-)normal directions to the distance
sphere $S_t(y)$ centered at $y$. The operator relating $U^t$ and
$L_t$ is a $d$th order (pseudo-differential) operator in $(t,x)$,
and in the Fourier transformed picture in \cite{A} it gives rise
to the factor $h^{-d/2}$.  As time evolves, the surface volume  of
the distance sphere in the universal cover grows exponentially.
Since $U^t$ is unitary, its amplitude must decay exponentially. In
fact this is evident from (\ref{HD}) since $\Theta(x,y)$ grows
exponentially in the distance from $x$ to $y$, so
$\Theta^{-1/2}(x,y)$ decays exponentially.

The cylinder set operators  $\hat{\ccal}_{\hbar} =
\hat{P}_{\alpha_n}(n) \cdots \hat{P}_{\alpha_1}(1)
\hat{P}_{\alpha_0}(0)$  are given by a sequence of alternating
applications of the propagator $U$ and the cutting off operators
$P_{\alpha_j}$ (which are multiplications). At least
heuristically, in estimating the Schur integral for $K_n(x, y)$,
one is first cutting off to a piece of the distance of the
distance sphere $S_1(y)$ of radius one centered at $y$. Then one
applies $U$ and this piece expands to $S_2(y)$. Then one cuts it
off with $P_{\alpha_2}$. As this picture evolves, one ends up
after $n$ iterates with a piece of the distance sphere  $S_{n}(y)$
of radius $n$ centered at $y$ which is has the same size as the
initial piece. But the spherical means and wave group are
normalized by dividing by the volume, so the contribution to the
Schur integral is exponentially decaying at the rate $e^{- n}$.

\subsection{Clarifications to \cite{A}}

\begin{enumerate}

\item The $\theta$ on p. 447 is a new parameter $< 1$ which is
implicitly defined above (2.0.1).

\item The $\theta $ in (2.3.2) is also implicitly defined by this
inequality. $W_{n_1}$ is not assumed to be a $\theta$ cover.

\item  In  (2.5.1) -- (2.5.5), the $\Sigma$ in  $\sigma^{-k}
\Sigma(W_{n_1})$ should be removed.

\item As mentioned above, the  notation  $\mu_h (\Sigma_N(W_n,
\tau)) $ is not quite defined but means $\mu_h (\bigcup_{\ccal \in
\Sigma_N(W_n, \tau))} \ccal)$.

\end{enumerate}

%
%

\section{\label{ANALYTIC} Applications to nodal hypersurfaces of eigenfunctions}

This survey has been devoted to the asymptotics of matrix elements
$\langle A \phi_j, \phi_j\rangle$. In the final section, we
address the question, ``what applications does the study of these
matrix elements have?"

In automorphic forms, the matrix elements are related to
L-functions, Rankin-Selberg zeta functions, etc.  So there is
ample motivation for their study in the arithmetic case. Such
applications lie outside the scope of this survey, so we refer the
reader to Sarnak's expository articles \cite{Sar2,Sar3} for
background on and references to the large literature  in
arithmetic quantum chaos.

Here we are concerned with applications to the classical and
geometric analysis of eigenfunctions on general Riemannian
manifolds. We present an application of quantum ergodicity to the
study of nodal (zero) sets of eigenfunctions on real analytic $(M,
g)$ with ergodic geodesic flow. It is not implausible that one
could use related methods to count critical points of
eigenfunctions in the ergodic case.

The nodal hypersurface of $\phi_j$ is the zero set $\ncal_j = \{x
\in M: \phi_j(x) = 0\}$. We are interested in its distribution as
$\lambda_j \to \infty$. We study the distribution through the
integrals
$$\int_{\ncal_j} f dS$$
of $f$ over the nodal line with respect to the natural surface
measure. If  $f = 1_U$ is the characteristic (indicator) function
of an open set $U \subset M$, then
\begin{equation} \label{UINTN} \int_{\ncal_j} 1_U ds =
\hcal^{n-1}( U \cap \ncal_j)
\end{equation}  is $(n-1)$ dimensional Hausdorff measure of the part of  $\ncal_j$
which intersects $U$.  When $U = M$ the integral (\ref{UINTN})
gives the total surface volume of the nodal set. The principal
result on volumes is due to
 Donnelly-Fefferman \cite{DF}, proving a conjecture of S. T.
 Yau \cite{Y1}:
\begin{theo} \label{DFNODAL} \cite{DF} (see also \cite{Lin})  Let $(M, g)$ be a compact real analytic  Riemannian
manifold, with or without boundary. Then there exist $c_1, C_2$
depending only on $(M, g)$ such that
\begin{equation} \label{DF} c_1 \lambda \leq {\mathcal
H}^{m-1}(Z_{\phi_{\lambda}}) \leq C_2 \lambda, \;\;\;\;\;\;(\Delta
\phi_{\lambda} = \lambda^2 \phi_{\lambda}; c_1, C_2 > 0).
\end{equation}
\end{theo}

The question is whether one can obtain asymptotics results by
imposing a global dynamical condition on the geodesic flow.

The hypothesis of real analyticity is crucial:  The analysis is
based on the analytic continuation of eigenfunctions to the
complexification $M_{\C}$  of $M$. The articles \cite{DF,Lin} also
used analytic continuation to study volumes of nodal
hypersurfaces. It is the role of ergodicity in the complex setting
that gives strong equidistribution results.

\subsection{Complexification of $(M,g)$ and Grauert tubes}

By a theorem of Bruhat and Whitney, a  real analytic manifold $M$
always possesses a unique complexification $M_{\C}$,  an open
complex manifold in  which $M$ embeds $\iota: M \to M_{\C}$ as a
totally real submanifold. Examples are real algebraic subvarieties
of $\R^n$ which complexify to complex algebraic subvarieties of
$\C^n$. Other simple examples include the complexification of a
quotient such as  $M = \R^m/\Z^m$ to  $M_{\C} = \C^m/\Z^m$. The
complexification of hyperbolic space can be defined using the
 hyperboloid model
$\Hh^n = \{ x_1^2 + \cdots x_n^2 - x_{n+1}^2 = -1, \;\; x_n > 0\}.
$ Then, $\Hh^n_{\C} = \{(z_1, \dots, z_{n+1}) \in \C^{n+1}:  z_1^2
+ \cdots z_n^2 - z_{n+1}^2 = -1\}. $

The Riemannian metric determines a special kind of
plurisubharmonic exhaustion function on  $M_{\C}$ \cite{GS1, GS2,
LemS1}, namely the Grauert tube radius $\sqrt{\rho} =
\sqrt{\rho}_g$ on $M_{\C}$, defined  as the unique solution of the
Monge-Amp\`ere equation
$$(\ddbar \sqrt{\rho})^m = \delta_{M_{\R}, dV_g}, \;\; \iota^*
(i \ddbar \rho) = g. $$ Here, $\delta_{M_{\R}, dV_g}$ is the
delta-function on the real $M$ with respect to the volume form
$dV_g$, i.e. $f \to \int_M f dV_g$. It is shown in \cite{GS1} that
$\sqrt{\rho}(\zeta) = i \sqrt{ r^2(\zeta, \bar{\zeta})}$ where
$r^2(x, y)$ is the squared distance function in a neighborhood of
the diagonal in $M \times M$.

The  Grauert tubes $M_{\tau}$ is defined by $$M_{\tau}  = \{\zeta
\in M_{\C}: \sqrt{\rho}(\zeta) \leq \tau\}. $$  The maximal value
of  $\tau_0$ for which $\sqrt{\rho}$ is well defined  is known as
the Grauert tube radius. For $\tau \leq \tau_0$, $M_{\tau}$ is a
strictly pseudo-convex domain in $M_{\C}$. Using the
 complexified exponential map $(x, \xi) \to exp_x i \xi$
one may identify $M_{\tau}$ with the co-ball bundle $B_{\tau}^* M$
to $M_{\tau}$. Under this identification,  $\sqrt{\rho}$
corresponds to $|\xi|_g$. The one-complex dimensional null
foliation of $\ddbar \sqrt{\rho}$, known as the `Monge-Amp\`ere'
or Riemann foliation, are the complex curves $t + i \tau \to \tau
\dot{\gamma}(t)$, where $\gamma$ is a geodesic. The geometric
Grauert tube radius is the maximal radius for which the
exponential map has a holomorphic extension and defines a
diffeomorphism.

\subsection{Analytic Continuation of eigenfunctions}

The eigenfunctions are real analytic and therefore possess
analytic continuations to some Grauert tube $M_{\epsilon}$
independent of the eigenvalue. To study  analytic continuation of
eigenfunctions, it is useful to relate them to analytic
continuation of the  wave group at imaginary time,
  $U(i\tau , x, y) = e^{-
\tau \sqrt{\Delta}}(x, y)$. This  is the Poisson operator of $(M,
g)$.  We denote its analytic continuation in the $x$ variable to
$M_{\epsilon} $ by $U_{\C}(i \tau, \zeta, y)$. In terms of the
eigenfunction expansion, one has
\begin{equation} \label{UI} U(i \tau, \zeta, y) = \sum_{j = 0}^{\infty} e^{-
\tau  \lambda_j} \phi_{ j}^{\C} (\zeta) \phi_j(y),\;\;\; (\zeta,
y) \in M_{\epsilon}  \times M.  \end{equation} To the author's
knowledge, the largest $\epsilon$ has not been determined at this
time. In particular, it is not known if the radius of analytic
continuation of the wave kernel and of the eigenfunctions is the
same as the geometrically defined Grauert tube radius.

Since
\begin{equation} U_{\C} (i \tau) \phi_{\lambda} = e^{- \tau \lambda}
\phi_{\lambda}^{\C}, \end{equation} the analytic continuability of
the Poisson operator to $M_{\tau}$  implies that  every
eigenfunction analytically continues to the same Grauert tube. It
follows that the analytic continuation operator to $M_{\tau}$ is
given by  \begin{equation} \label{ACO} A_{\C}(\tau) = U_{\C}(i
\tau) \circ e^{\tau \sqrt{\Delta}}. \end{equation} Thus, a
function   $f \in C^{\infty}(M)$ has a holomorphic extension to
the closed  tube $\sqrt{\rho}(\zeta) \leq \tau$ if and only if $f
\in Dom(e^{\tau \sqrt{\Delta}}), $ where $e^{\tau \sqrt{\Delta}}$
is the backwards `heat operator' generated by $\sqrt{\Delta}$
(rather than $\Delta$).

Let us consider examples of holomorphic continuations of
eigenfunctions:

\begin{itemize}

\item On  the flat torus $\R^m/\Z^m,$   the real eigenfunctions
are $\cos \langle k, x \rangle, \sin \langle k, x \rangle$ with $k
\in 2 \pi \Z^m.$ The complexified torus is $\C^m/\Z^m$ and the
complexified eigenfunctions are $\cos \langle k, \zeta \rangle,
\sin \langle k, \zeta \rangle$ with $\zeta  = x + i \xi.$

\item On the unit sphere $S^m$, eigenfunctions are restrictions of
homogeneous harmonic functions on $\R^{m + 1}$. The latter extend
holomorphically to holomorphic harmonic polynomials on $\C^{m +
1}$ and restrict to holomorphic function on $S^m_{\C}$.

\item On $\H^m$, one may use the hyperbolic plane waves $e^{ (i
\lambda + 1) \langle z, b \rangle}$, where  $\langle z, b \rangle$
is the (signed) hyperbolic distance of the horocycle passing
through $z$ and $b$ to $0$. They  may be holomorphically extended
to the maximal tube of radius $\pi/4$. See e.g. \cite{BR}.

\item
 On  compact hyperbolic quotients ${\bf
H}^m/\Gamma$, eigenfunctions can be then represented by Helgason's
generalized Poisson integral formula,
$$\phi_{\lambda}(z) = \int_B e^{(i \lambda + 1)\langle z, b
\rangle } dT_{\lambda}(b). $$
 To  analytically continue $\phi_{\lambda}$ it suffices  to
analytically continue $\langle z, b\rangle. $ Writing the latter
as  $\langle \zeta, b \rangle,  $ we have:
\begin{equation} \label{HEL} \phi_{\lambda}^{\C} (\zeta) = \int_B e^{(i \lambda +
1)\langle \zeta, b \rangle } dT_{\lambda}(b). \end{equation}

\item In the case of 1D Schr\"odinger operators, several authors
have studied analytic continuations of eigenfunctions and their
complex zeros, see e.g. \cite{Hez,EGS}.

\end{itemize}

\subsection{Maximal plurisubharmonic functions and growth of $\phi_{\lambda}^{\C}$}

In the case of domains $\Omega \subset \C^m$, the maximal PSH
(pluri-subharmonic)  function (or pluri-complex Green's function)
relative to a subset $E \subset \Omega$ is defined by
$$V_{E}(\zeta) = \sup\{u(z): u \in PSH(\Omega), u|_{E}  \leq 0, u
|_{\partial \Omega} \leq 1\}. $$ An alternative definition due to
Siciak takes the supremum only with respect to polynomials $p$. We
denote by $\pcal^N$ the space of all complex analytic polynomials
of degree $N$ and put  $\pcal_K^N = \{p \in \pcal^N: ||p||_K \leq
1, \;\; ||p||_{\Omega} \leq e\}. $ Then define
$$\log \Phi_E^N(\zeta) = \sup\{ \frac{1}{N} \log |p_N(\zeta)|: p \in \pcal_E^N \}, \;\;\; \log \Phi_E = \limsup_{N \to \infty} \log \Phi_E^N. $$
Here, $||f||_K = \sup_{z \in K} |f(z)|$.
 Siciak proved  that $\log \Phi_E = V_E$.

On a real analytic Riemannian manifold, the natural analogue of
$\pcal^N$ is the space
$$\hcal^{\lambda} = \{p =  \sum_{j: \lambda_j \leq \lambda} a_j
\phi_{\lambda_j}, \;\; a_1, \dots, a_{N(\lambda)} \in \R  \} $$
spanned by eigenfunctions with frequencies $\leq \lambda$. Rather
than using the sup norm, it is convenient  to work with $L^2$
based norms than sup norms, and so we define
$$ \hcal^{\lambda}_M = \{p =  \sum_{j: \lambda_j \leq \lambda} a_j
\phi_{\lambda_j}, \;\;||p||_{L^2(M)} =  \sum_{j = 1}^{N(\lambda)}
|a_j|^2 = 1 \}. $$ We define the   $\lambda$-Siciak approximate
extremal function by
 $$ \Phi_M^{\lambda} (z) = \sup \{|\psi(z)|^{1/\lambda} \colon
\psi \in \hcal_{\lambda};   \|\psi \|_M \le 1 \},  $$ and the
extremal function by
$$\Phi_M(z) = \sup_{\lambda} \Phi_M^{\lambda}(z). $$

It is not hard to show that
\begin{equation} \Phi_M(z) = \lim_{\lambda \to \infty} \frac{1}{\lambda} \log  \Pi_{[0,
\lambda}(\zeta, \bar{\zeta}) = \sqrt{\rho}. \end{equation}

\subsection{Nodal hypersurfaces in the case of ergodic geodesic flow}

The complex nodal hypersurface of an eigenfunction is   defined by
\begin{equation} Z_{\phi_{\lambda}^{\C}} = \{\zeta \in
B^*_{\epsilon_0} M: \phi_{\lambda}^{\C}(\zeta) = 0 \}.
\end{equation}
There exists  a natural current of integration over the nodal
hypersurface in any ball bundle $B^*_{\epsilon} M$ with $\epsilon
< \epsilon_0$ , given by
\begin{equation}\label{ZDEF}  \langle [Z_{\phi_{\lambda}^{\C}}] , \phi \rangle =  \frac{i}{2 \pi} \int_{B^*_{\epsilon} M} \ddbar \log
|\phi_{\lambda}^{\C}|^2 \wedge \phi =
\int_{Z_{\phi_{\lambda}^{\C}} } \phi,\;\;\; \phi \in \dcal^{ (m-1,
m-1)} (B^*_{\epsilon} M). \end{equation} In the second equality we
used the Poincar\'e-Lelong formula. The notation $\dcal^{ (m-1,
m-1)} (B^*_{\epsilon} M)$ stands for smooth test $(m-1,
m-1)$-forms with support in $B^*_{\epsilon} M.$

The nodal hypersurface $Z_{\phi_{\lambda}^{\C}}$ also carries a
natural volume form $|Z_{\phi_{\lambda}^{\C}}|$ as a complex
hypersurface in a K\"ahler manifold. By Wirtinger's formula, it
equals the restriction of $\frac{\omega_g^{m-1}}{(m - 1)!}$ to
$Z_{\phi_{\lambda}^{\C}}$.

\begin{theo} \label{ZEROS}   Let $(M, g)$ be  real analytic, and let $\{\phi_{j_k}\}$ denote a quantum ergodic sequence
of eigenfunctions of its Laplacian $\Delta$.  Let
$(B^*_{\epsilon_0} M, J)$ be the maximal Grauert tube around $M$
with complex structure $J_g$ adapted to $g$. Let $\epsilon <
\epsilon_0$. Then:
$$\frac{1}{\lambda_{j_k}} [Z_{\phi_{j_k}^{\C}}] \to  \frac{i}{ \pi} \ddbar \sqrt{\rho}\;\;
\mbox{weakly in }\;\;  \dcal^{' (1,1)} (B^*_{\epsilon} M), $$ in
the sense that,   for any continuous test form $\psi \in \dcal^{
(m-1, m-1)}(B^*_{\epsilon} M)$, we have
$$\frac{1}{\lambda_{j_k}} \int_{Z_{\phi_{j_k}^{\C}}} \psi \to
 \frac{i}{ \pi} \int_{B^*_{\epsilon} M} \psi \wedge \ddbar
\sqrt{\rho}. $$ Equivalently, for any  $\phi \in C(B^*_{\epsilon}
M)$,
$$\frac{1}{\lambda_{j_k}} \int_{Z_{\phi_{j_k}^{\C}}} \phi \frac{\omega_g^{m-1}}{(m -
1)!}  \to
 \frac{i}{ \pi} \int_{B^*_{\epsilon} M} \phi  \ddbar
\sqrt{\rho}  \wedge \frac{\omega_g^{m-1}}{(m - 1)!} . $$
\end{theo}

\begin{cor}\label{ZEROCOR}  Let $(M, g)$ be a real analytic with ergodic  geodesic
flow.  Let $\{\phi_{j_k}\}$ denote a full density ergodic
sequence. Then for all $\epsilon < \epsilon_0$,
$$\frac{1}{\lambda_{j_k}} [Z_{\phi_{j_k}^{\C}}] \to  \frac{i}{ \pi} \ddbar \sqrt{\rho},\;\;
 \mbox{weakly in}\;\; \dcal^{' (1,1)} (B^*_{\epsilon} M). $$
\end{cor}

\begin{rem} The limit form is singular with respect to the \kahler
form $i \ddbar \rho$ on $M_{\C}$ adapted to the metric. Its
highest exterior power is the delta function on the zero section
(i.e. the real domain $M$). \end{rem}

In outline the steps in the proof are;

\begin{enumerate}

\item By the Poincar\'e-Lelong formula, $[Z_{\phi_{\lambda}^{\C}}]
= i \ddbar \log |\phi_{\lambda}^{\C}|. $ This reduces the theorem
to determining the limit of $\frac{1}{\lambda} \log
|\phi_{\lambda}^{\C}|$.

\item $|\phi_{\lambda}^{\C}|^2$, when properly $L^2$ normalized on
each $\partial M_{\tau}$ is a quantum ergodic sequence on
$\partial M_{\tau}$.

\item The ergodic property of complexified eigenfunctions implies
that the $L^2$ norm of $|\phi_{\lambda}^{\C}|^2$
  on $\partial M_{\tau}$ is asymtotically $\sqrt{\rho}$.
Thus, normalized log moduli of ergodic eigenfunctions are
asympotically maximal PSH functions. They are $\lambda$-Siciak
extremal functions.

\item $\frac{1}{\lambda} \log
\frac{|\phi_{\lambda}^{\C}(\zeta)|^2}{||\phi_{\lambda}^{\C}||^2_{L^2(\partial
M_{\tau}}} \to 0$ for an ergodic sequence of eigenfunctions.
Hence,

\item $\frac{i}{\lambda} \ddbar \log |\phi_{\lambda}^{\C}| \to i
\ddbar \sqrt{\rho}, $

\end{enumerate}
concluding the proof.

\begin{rem} We expect a similar result for nodal lines of Dirichlet or Neumann
eigenfunctions of a piecewise smooth domain $\Omega \subset \R^2$
with ergodic billiards. We also expect similar results for nodal
lines of such domains which touch the boundary (like the angular
`spokes' of eigenfunctions of the disc). In \cite{TZ1}, we proved
that the number of such nodal lines is bounded above by
$C_{\Omega} \lambda$. This is rather like the Donnelly-Fefferman
upper bound. However the example of the disc shows that there is
no lower bound above zero. That is, there are sequences of
eigenfunctions (with low angular momentum) with a uniformly
bounded number of nodal lines touching the boundary. We believe
that ergodicity is a mechanism producing lower bounds on numbers
of zeros and critical points. It causes eigenfunctions to
oscillate uniformly in all directions and hence to produce a lot
of zeros and critical points. We expect this to be one of the
future roles for ergodicity of eigenfunctions.
\end{rem}

\begin{rem} Precursors to Theorem \ref{ZEROS} giving limit distributions of
zero sets for holomorphic eigen-sections of \kahler quantizations
of ergodic symplectic maps on \kahler manifolds can be found in
\cite{NV2,SZ}. In that setting, the zeros become uniformly
distributed with respect to the K\"ahler  form. In the Riemannian
setting, they become uniformly distributed with respect to a
singular $(1,1)$ current relative to the K\"ahler form $\ddbar
\rho$ on $M_{\C}$.
\end{rem}

\end{document}